\documentclass[preprint]{elsarticle}       
\usepackage{lipsum}
\makeatletter
\def\ps@pprintTitle{%
 \let\@oddhead\@empty
 \let\@evenhead\@empty
 \def\@oddfoot{}%
 \let\@evenfoot\@oddfoot}
\makeatother
\usepackage{lineno,hyperref}
\usepackage{lscape}
\usepackage{enumitem}
\usepackage{rotating}
\usepackage{makeidx}
\usepackage{float}
\usepackage{epsfig}
\usepackage{amsmath,amssymb}
\usepackage{amsthm}
\usepackage{enumitem}
\usepackage{multirow}
\usepackage[utf8]{inputenc}
\usepackage{epic,eepic}
\usepackage{overpic}
\usepackage{color}
\usepackage{amsfonts}
\usepackage{graphicx}
\usepackage{enumitem}

\usepackage{mathtools}
\mathtoolsset{showonlyrefs}

\theoremstyle{plain}
\newtheorem{theorem}{Theorem}[section]
\newtheorem{lemma}[theorem]{Lemma}
\newtheorem{proposition}[theorem]{Proposition}
\newtheorem{corollary}[theorem]{Corollary}

\theoremstyle{definition}
\newtheorem{definicion}{Definition}

\theoremstyle{remark}
\newtheorem{remark}{Remark}[section]

\def\ba{\mathbf{a}}
\def\balpha{\boldsymbol{\alpha}}
\def\bbeta{\boldsymbol{\beta}}
\def\bepsilon{\boldsymbol{\epsilon}}
\def\bff{\mathbf{f}}

\def\bq{\mathbf{q}}

\def\bw{\mathbf{w}}
\def\bW{\mathbf{W}}
\def\bx{\mathbf{x}}
\def\bX{\mathbf{X}}

\def\N{\mathbb{N}}
\def\Z{\mathbb{Z}}
\def\R{\mathbb{R}}

\newcommand{\liZ}{\ell_\infty(\Z)}

\newcommand{\cC}{\mathcal{C}}

\newcommand{\cO}{\mathcal{O}}

\DeclareMathOperator*{\argmin}{arg\,min}

\DeclareMathOperator*{\var}{var}

\begin{document}

\begin{frontmatter}

\title{Subdivision schemes based on weighted local polynomial regression. \\ A new technique for the convergence analysis.}\tnotetext[label1]{This research has been supported by project CIAICO/2021/227 (funded by Conselleria de Innovación, Universidades, Ciencia y Sociedad digital, Generalitat Valenciana) and by grant PID2020-117211GB-I00 (funded by MCIN/AEI/10.13039/501100011033).}

\author[UV1]{Sergio L\'opez-Ure\~na}
\ead{sergio.lopez-urena@uv.es}
\author[UV1]{Dionisio F. Y\'a\~nez}
\ead{dionisio.yanez@uv.es}
\date{}
\journal{J. Sci. Comput.}
\address[UV1]{Departament de Matem\`atiques.  Universitat de Val\`encia (EG) (Spain)}

\begin{abstract}
The generation of curves and surfaces from given data is a well-known problem in Computer-Aided Design that can be approached using subdivision schemes. They are powerful tools that allow obtaining new data from the initial one by means of simple calculations. However, in some applications, the collected data are given with noise and most of schemes are not adequate to process them. In this paper, we present some new families of binary univariate linear subdivision schemes using weighted local polynomial regression. We study their properties, such as convergence, monotonicity, polynomial reproduction and approximation and denoising capabilities. For the convergence study, we develop some new theoretical results. Finally, some examples are presented to confirm the proven properties.
\end{abstract}


\end{frontmatter}

\section{Introduction}

In past years, many techniques have been designed and developed in order to construct curves or surfaces with some properties such as polynomial reproduction or monotonicity-preservation. For example, splines, non-uniform rational B-splines (NURBS) and others (see, e.g. \cite{arandigadona_rbf,cripshussain}).
In this context, linear subdivision schemes appears as useful and efficient instruments due to their simple computation (see e.g. \cite{cavdamich,dynlevin,michelliprautzsch}). They consist in obtaining new points from given data using refinement operators and can be classified depending on such operators: if a single operator is used for all the iterations, then the subdivision scheme is called stationary or level-independent (see e.g. \cite{cavdamich,dynheardhormannsharon}), otherwise it is denominated non-stationary or level dependent (see e.g. \cite{cohendyn,contidyn,dynlevin2}). They are also classified by the linearity of the operators (see e.g. \cite{donatlopez,dyn}).

There is a vast literature on the generation of subdivision schemes and the study of their properties. An essential property is convergence, which means that the process converges uniformly to a continuous function, for any initial values. Deslauriers and Dubuc, in \cite{desdub}, analysed that the scheme based on centered Lagrange interpolation is convergent using Fourier transform techniques to prove it.

One of the most common studied properties is the reproduction of polynomials, i.e., if the given data are point-values of a polynomial, then the subdivision scheme generates more point-values of such polynomial. This is studied in detail in \cite{dynhormannsabinshen}. Its study is interesting since the reproduction is linked with convergence properties and the approximation capability of the scheme.

In some real applications, the given data come from measures that are contaminated by noise and, as a consequence, a suitable subdivision scheme should be used to converge to an appropriate limit function. To this purpose, Dyn et al. in \cite{dynheardhormannsharon} propose a new linear scheme based on least-square methods where the noise is reduced by applying the scheme several times. These schemes are determined by two parameters $m$ and $d$ with $d<m$: For each $m$ consecutive data values, $(y_1,\ldots,y_{m})$, attached to some equidistant knots $(x_1,\ldots,x_{m})$, a polynomial regression is performed. The search is constrained to polynomials of degree $d$ and leads to a unique solution, $\hat p$, that minimizes the regression error concerning the $\ell^2$-norm (least-squares):
\begin{equation}\label{problemadyn}
\hat{p}=\argmin_{p \in \Pi_{d}(\mathbb{R})} \sum_{l=1}^m (y_{l}-p(x_l))^2.
\end{equation}
The subdivision refinement rules can be obtained by evaluating $\hat p$ at a certain point, which, in this work, is assumed to be 0 without loss of generalization.
The resulting schemes are linear, which implies some benefits and drawbacks. In \cite{dynheardhormannsharon}, the convergence is proved for $d=0,1$, as well as some properties such as polynomials reproduction.

In many applied situations, the location of the data is relevant to obtain the approximation, hence a weight function is considered to assign values depending on the distance from the knots $x_l$ to 0. These methods, as Shepard's algorithm (see \cite{shepard}), are called moving least squares (see \cite{levin}).
In \cite{arandigayanez13,arandigayanez14}, the weighted local polynomial regression (WLPR) was used to design a \emph{prediction operator} for a multiresolution algorithm, leading to good results on image processing when the data was contaminated with some noise. Prediction operators can be considered subdivision operators and their properties can be studied \cite{cohendyn}. In this paper, we study the family of subdivision schemes based on the prediction operators in \cite{arandigayanez13} and develop a new technique to study their convergence based on some asymptotic behaviour. Also, some properties such as polynomial reproduction,  the Gibbs phenomenon in discontinuous data, monotonicity preservation and denoising and approximation capabilities are analysed. We provide some examples to check the theoretical results.

The paper is organized as follows: Firstly, we briefly review the classical components of linear subdivision schemes with the aim to be self-contained in this work. In Section \ref{section3}, we explain the WLPR and define a general form, leading to new subdivision schemes definitions. Afterward, we study different properties in some particular cases: Starting with $d=0,1$, we analyse the convergence, the smoothness of the limit functions, the monotonicity preservation and the Gibbs phenomenon when the initial data present large gradients. In Section \ref{sec:tools}, we develop a new technique to study the convergence of a family of schemes and apply it to the case $d=2,3$.
We analyse the approximation and noise reduction capabilities of the new schemes in Sections \ref{approx} and \ref{noise}. Finally, some numerical experiments are performed to confirm the theoretical properties, in Section \ref{expnum}, and some conclusions and future work are proposed.

\section{Preliminaries: A brief review of linear subdivision schemes}

Let us denote by $\liZ$ the set of bounded real sequences with indices in $\Z$. A \emph{linear binary univariate subdivision operator} $S_\ba:\liZ \to \liZ$ with finitely supported \emph{mask} $\ba=\{a_l\}_{l\in \mathbb{Z}}\subset\R$ is defined to refine the data on the level $k$, $\bff^k=\{f^k_j\}_{j\in \Z} \in \liZ$, as:
\begin{equation}\label{definicionmascara}
f^{k+1}_{2j+i} :=(S_\ba \bff^k)_{2j+i}:=\sum_{l\in \Z} a_{2l-i}f^k_{j+l}, \quad j\in \Z, \quad i=0,1.
\end{equation}
In this work, we only consider level-independent subdivision schemes, meaning that the successive application of a unique operator $S_\ba$ constitutes the \emph{subdivision scheme}. Hence, we will refer to $S_\ba$ as the subdivision scheme as well.
The \emph{binary} adjective refers to the two formulas/rules of \eqref{definicionmascara} (corresponding to $i=0$ and $i=1$) which are characterized by the \emph{even} mask $\ba^0=\{a_{2l}\}_{l\in \mathbb{Z}}$ and the \emph{odd} mask $\ba^1=\{a_{2l-1}\}_{l\in \mathbb{Z}}$. It is called \emph{length} of a mask to the number of elements that are between the first and the last non-zero elements, both included.

\begin{remark} \label{rmk:separation}
If a linear subdivision scheme is applied to some data $\widetilde{\mathbf{g}} = \{G(j) + \epsilon_j\}_{j\in\Z}$, where $G$ is a smooth function and $\bepsilon = \{\epsilon_j\}_{j\in\Z}$ is random data, also called \emph{noise}, the result is
\[ S_\ba \widetilde{\mathbf{g}}  = S_\ba \mathbf{g} + S_\ba \bepsilon,\]
which implies that we can study separately the smooth and the pure noisy cases.
\end{remark}

If we apply these rules recursively to some initial data $\bff^0$, it is desirable that the process converges to a continuous function, in the following sense.
\begin{definicion}
A subdivision scheme $S_\ba$ is \emph{uniformly convergent} if for any initial data $\bff^0\in\liZ$, there exists a continuous function $F:\R\to\R$ such that
\begin{equation*}
\lim_{k\to\infty} \sup_{j \in \mathbb{Z}}|(S_\ba^k \bff^0)_j - F(2^{-k}j)|=0.
\end{equation*}
Then, we denote by $S^\infty_{\ba}\bff^0=F$ to the limit function generated from $\bff^0$. We write $S_\ba\in\cC^d$ if all the limit functions have such smoothness, $S^\infty_{\ba}\bff^0 \in \cC^d$, $\forall \bff^0 \in \liZ$.
\end{definicion}

A usual tool for the analysis of linear schemes is the \emph{symbol}, that we define as follows.
\begin{definicion}
The \emph{symbol} of a subdivision scheme $S_\ba$ is the Laurent polynomial
\(
a(z) = \sum_{j\in \Z} a_j z^{-j}.
\)
\end{definicion}

We can determine if a subdivision scheme is convergent depending on the sum of the absolute values of some even and odd masks. Therefore, we use the norm of the operator $S_\ba$.
\begin{lemma}
The norm of $S_\ba:\liZ \to \liZ$, as a linear endomorphism in the space of bounded sequences, is the maximum between $\|\ba^0\|_1$ and $\|\ba^1\|_1$:
\[
\|S_\ba\|_\infty = \max_{i=0,1} \{ \sum_{j\in\Z} |a_{2j-i}|\} = \max \{ \|\ba^0\|_1,\|\ba^1\|_1\}.
\]
\end{lemma}

According to the Definition 5.1 of \cite{dynhormannsabinshen}, a subdivision scheme $S_\ba$ is \emph{odd-symmetric} if
\(
a_j = a_{-j}, \ \forall j\in\Z,
\)
or \emph{even-symmetric} if
\(
a_j = a_{1-j}, \ \forall j\in\Z.
\)
In terms of the symbol, these is translated as $a(z)=a(1/z)$ or $a(z)=z a(1/z)$, respectively.
The schemes, that we will construct in this paper, are odd-symmetric, but to simplify some equations, we consider a more relaxed definition of odd-symmetry and even-symmetry.

\begin{definicion}
A subdivision scheme $S_\ba$ is \emph{symmetric} if
\(
a_j = a_{j_0-j}, \ \forall j\in\Z,
\)
for some $j_0 \in \Z$. It is \emph{even(odd)-symmetric} if $j_0$ is odd (even).
\end{definicion}

A useful property for a subdivision scheme is the reproduction of polynomials.

\begin{definicion}
A subdivision scheme $S_\ba$ \emph{reproduces}\footnote{Technically, this is the definition of \emph{step-wise reproduction}, which is a stronger condition, \cite{dynhormannsabinshen}.} $\Pi_d$ (polynomials up to degree $d$) if
\[
S_\ba \{p(2 j)\}_{j\in\Z} = \{p(j)\}_{j\in\Z}, \qquad \forall p\in\Pi_d.
\]
\end{definicion}

A necessary condition for convergence is the reproduction of constants. The following lemma determines the relation between the mask, the symbol and the reproduction of the constants.
\begin{lemma} \label{lem:sum1}
The following facts are equivalent:
\begin{description}
\item [(a)]	$S_\ba$ reproduces $\Pi_0$ (constant functions).
\item [(b)] \(\sum_{j\in\Z} a^0_j = \sum_{j\in\Z} a^1_j = 1\).
\item [(c)] $a(z) = (1+z)q(z)$ for some Laurent polynomial $q$.
\end{description}
In such case, the $S_\bq$ scheme is well-defined and called \emph{difference scheme}. If $\|S_{\bq}\|< 1$, then $S_\ba$ is convergent.
\end{lemma}

There exists a direct relationship between the symmetry of $S_\ba$ and the symmetry of its difference scheme, $S_\bq$. We introduce it in the following result.
\begin{lemma} \label{lem:even_difference}
If a scheme is odd-symmetric, then its difference scheme is even-symmetric.
\end{lemma}
\begin{proof}
It can be easily checked using the \emph{symbols}.
\end{proof}

Next theorem by Dyn and Levin, \cite{dynlevin}, links the smoothness of $S_\ba$ and $S_{2\bq}$.
\begin{theorem}\label{teoremadyn1}
If the scheme based on $S_{2\bq}$ is convergent and $\mathcal{C}^{m-1}$, then $S_{\ba}$ is convergent and $\mathcal{C}^{m}$.
\end{theorem}

\begin{remark} \label{rmk:two_situations}
We give now a more explicit formula to compute $\bq$ for the kind of schemes we consider in this paper. We will analyze odd-symmetric subdivision schemes, which implies that the length of the mask is always odd and two possible situations may occur, depending on which sub-mask has the largest support.

Since the sub-masks $\ba^0$ and $\ba^1$ are finitely supported, from now on we will treat them as vectors containing \emph{only} their support, which will be important for the theoretical results in Section \ref{sec:tools}.
The first situation is that, for some $n\in\N$, the sub-masks are \(\ba^{0} = \{a^0_{l}\}_{l=1-n}^{n-1}\) and \(\ba^{1} = \{a^1_{l}\}_{l=1-n}^{n}\), while the second one corresponds to \(\ba^{0} = \{a^0_{l}\}_{l=-n}^{n}\) and \(\ba^{1} = \{a^1_{l}\}_{l=1-n}^{n}\) (pay attention to the supports).

To compute $\bq$ with a unique formula for both cases, we redefine the mask for the second case, consisting in $\bar a^0_l := a^1_{l}$, $l=1-n,\ldots,n$, and $\bar a^1_l := a^0_{l-1}$, $l=1-n,\ldots,n+1$. Now the first indices of the supports are $1-n$, in both situations, and the last indices are $n-1$ and $n$ (for the first and second sub-mask, respectively) for the first situation and $n$ and $n+1$ for the second one. Now, in both cases, the second sub-masks is the largest and we can affirm that there exists some $n\in\N$ such that
\begin{equation} \label{eq:our_schemes}
(S_\ba \bff)_{2j} = \sum_{l=1-n}^{L_n} a^{0}_{l} f_{j+l}, \quad (S_\ba \bff)_{2j + 1} = \sum_{l=1-n}^{L_n+1} a^{1}_{l} f_{j+l}, \quad j\in\Z,
\end{equation}
with $L_n = n-1$ or $L_n=n$, so that \(\ba^{0} = \{a^0_l\}_{l=1-n}^{L_n}\) and \(\ba^{1} = \{a^1_{l}\}_{l=1-n}^{L_n+1}\). In any case, now the odd-symmetry is written as
\begin{equation*} \label{eq:a_symmetry}
a^0_{l} = a^0_{L_n+1-n - l}, \quad  a^1_{l} = a^1_{L_n+2-n - l}.
\end{equation*}
Finally, for a subdivision operator written as \eqref{eq:our_schemes}, the difference mask $\bq$ can be computed as follows:
\begin{align} \label{eq:difference_scheme}
\begin{split}
q^0_{j} &:= q_{2j} = \sum_{l=-n+1}^{j} a^{n,0}_{l} - a^{n,1}_{l}, \qquad j = 1-n, \ldots, L_n,\\
q^1_{j} &:= q_{2j+1} = \sum_{l=j}^{L_n} a^{n,0}_{l} - a^{n,1}_{l+1}, \qquad j = 1-n, \ldots, L_n.
\end{split}
\end{align}

According to Lemma \ref{lem:even_difference}, $S_\bq$ is an even-symmetric scheme. In particular,
\begin{equation} \label{eq:q_symmetry}
q^0_{j} = q^1_{L_n+1-n - j}, \qquad j = 1-n, \ldots, L_n.
\end{equation}
\end{remark}

\section{Weighted local polynomial regression (WLPR)}\label{section3}

The schemes analysed in the present work has been applied to image processing in a multiresolution context as prediction operator both for point-values as for cell-average discretizations, (see, e.g. \cite{arandigayanez13,arandigayanez14}). They are based on weighted local polynomial regression (WLPR) and they can be defined by inserting a weight function in the minimization problem \eqref{problemadyn},
which emphasizes the points closer to where the new data is attached.
In this section, we briefly introduce  WLPR and describe some of its properties. For a more detailed description, see \cite{Hastie,Loader}.

Firstly, we fix the space of functions where the regression is performed: $\Pi_d$, the space of polynomials of degree at most $d$. Other function spaces could be used as well (see \cite{Hastie}). We can parametrize the polynomials in $\Pi_d$ as
\[ p(x) = \beta_0 + \beta_1 x + \dots + \beta_d x^d = A(x)^T \bbeta\]
where the superscript $T$ is the matrix transposition, $A(x)^T=(1,x,\ldots,x^d)$ and $\bbeta\in\R^{d+1}$. The vectors are considered column vectors in order to perform the matrix multiplication. With this notation, the regression problem \eqref{problemadyn} can be expressed as
\begin{equation}\label{problemadyn_mod}
\begin{split}
\hat \bbeta =\argmin_{\bbeta \in \mathbb{R}^{d+1}} \sum_{i=1}^m L_2(y_i, A(x_i)^T \bbeta),\quad \hat p = A(x)^T \hat \bbeta, \quad L_2(s,t) := (s-t)^2.
\end{split}
\end{equation}

The second ingredient is the weight function, $\omega: \R \to [0,1]$, which assigns a value to the distance between $x_i$ and 0, which is the location where $\hat p$ is evaluated in this work. We define $\omega$ as
\begin{equation*}
\omega(x)=\begin{cases}
\phi(|x|), & |x|\leq 1,\\
0, & \text{in other case},
\end{cases}
\end{equation*}
and we impose that $\phi:[0,1] \to [0,1]$ is a decreasing function such that $\phi(0)=1$.
With these assumptions it is clear that $\omega$ has compact support, $[-1,1]$, is even, increasing in $[-1,0]$ and decreasing in $[0,1]$, and it reaches the maximum at point $x=0$. The choice $w(0)=1$ assigns the highest weight to the point where $\hat p$ is evaluated. In \cite{Loader}, some functions are proposed, which we compile in Table \ref{tabla1nucleos}. Observe that many of them have the form $\phi(x)=(1-x^p)^q$ with $p,q>0$.
\begin{table}[!h]
\begin{equation*}
\begin{array}{lll}
\hline
\text{{\tt rect}}              & \phi(x) =1\\
\text{{\tt tria}}              & \phi(x) =1-x \\
\text{{\tt epan}}              & \phi(x) =1-x^2\\
\text{{\tt bisq}}              & \phi(x) =(1-x^2)^2\\
\text{{\tt tcub}}              & \phi(x) =(1-x^3)^3\\
\text{{\tt trwt}}              & \phi(x) =(1-x^2)^3\\
\text{{\tt exp}}(\xi)  			&\phi(x) =e^{-\xi x}& \xi\in\mathbb{R}_+\\
\hline
\end{array}
\end{equation*}\caption{Weight functions, see \cite{Loader}.}
\label{tabla1nucleos}
\end{table}

The third component is the \emph{bandwidth}, $\lambda\in\mathbb{R}_+\backslash\N$. We define
\begin{equation}\label{omega1}
	\bw^\lambda = \{w_l^\lambda\}_{l\in\Z}, \quad w_l^\lambda :=\omega\left(\frac {l}{\lambda}\right)=\phi\left(\frac {|l|}{\lambda}\right), \,\,l\in\mathbb{Z}.
\end{equation}
The parameter $\lambda$ determines how many data values are used in the regression and allows to distribute the weights of the points used in the rank $[-\lambda,\lambda]$.
By the properties of the function $\omega$, if $\lambda_1\leq\lambda_2$, then $w^{\lambda_1}_{l}\leq w^{\lambda_2}_{l}$ for any $l\in\mathbb{Z}$.

Finally, we choose a vector norm, typically $\ell^2$ is taken for its simplicity, but any $\ell^p$-norm can be used depending on the characteristics of the problem. The loss function is defined accordingly: $L_p(s,t) = |s-t|^p$.

With the above elements, we propose these two problems to design the two subdivision rules:
\begin{equation}\label{problema1}
\begin{split}
\hat \bbeta^i=\argmin_{\bbeta \in \mathbb{R}^{d+1}} \sum\{ w^\lambda_{2l-i}L_p(f^k_{j+l},A(2l-i)^T\bbeta) \, : \, l\in\Z, |2l-i| < \lambda \},\quad i=0,1.
\end{split}
\end{equation}

Once the fitted polynomial is obtained, it is evaluated at 0 to define the new data:
\begin{equation}\label{esquemasubdivision}
(S_{d,\mathbf{w^\lambda}} f^k)_{2j+i}=A(0)^T \hat\bbeta^i = (1,0,\ldots,0)\hat\bbeta^i = \hat\beta^i_0 ,\quad i=0,1,
\end{equation}
so that only the first coordinate of $\hat\bbeta^i$ is needed.

\begin{proposition}\label{propnumero}
Moreover, for $d = -1 + 2\left \lfloor \frac{\lambda+1}2 \right \rfloor$, the resulting subdivision scheme is the Deslauriers-Dubuc subdivision scheme.
\end{proposition}
\begin{proof}

Let us discuss when this scheme is well defined. Two situations may occur, depending on whether or not $d$ (the polynomial degree) is smaller than the amount of data $f^{k}_{j+l}$  in the minimization problem \eqref{problema1}.

For $i=0$, if
$d < 2\left \lfloor \frac{\lambda}2 \right \rfloor$
then \eqref{problema1} is a least square problem and there is a unique solution \cite{boyd}, otherwise it can be found a polynomial that interpolates the data. Even if the interpolating polynomial is not unique, its evaluation at 0 is exactly $f^k_j$. Hence, the even rule is well defined for any $\lambda\in\R_+\backslash\N$, coinciding with the even rule of the Deslauriers-Dubuc subdivision scheme for $d \geq 2\left \lfloor \frac{\lambda}2 \right \rfloor$, i.e. $f^{k+1}_{2j}=f^k_j$.

For $i=1$, a least of square problem is solved if
$d+1 < 2\left \lfloor \frac{\lambda+1}2 \right \rfloor$ and an interpolation problem with unique solution is solved when the equality is reached, coinciding with the Deslaurier-Dubuc odd rule in the last case. However, nor the polynomial neither its value at zero are unique when $d +1 > 2\left \lfloor \frac{\lambda+1}2 \right \rfloor$, so that the scheme is not well defined in this case.

As conclusion, only if the polynomial degree is $d = -1 + 2\left \lfloor \frac{\lambda+1}2 \right \rfloor$, the resulting scheme is the Deslauriers-Dubuc interpolatory subdivision scheme, independently the choice of $\omega$ and the loss function $L_p$.
\end{proof}

The scheme \eqref{esquemasubdivision} coincides with the proposed by Dyn et al. in \cite{dynheardhormannsharon} if $p=2$ and $\phi(x)=1$ are used (corresponding to {\tt rect} in Table \ref{tabla1nucleos}). Also, the  non-linear subdivision scheme presented by Mustafa et al. in \cite{mustafa} can be obtained with the same choice of $\phi(x)=1$ but with $p=1$.

We will analyse the properties of our schemes specifically for the polynomial degrees $d=0,1,2,3$, the loss function $L_2$ and several choices of $\phi$.

We will study how the choice of $\phi$ affects the approximation and noise reduction capabilities. We will show that it is not possible to define a $\phi$ giving the best approximation and the greatest denoising. In fact, one may decide how much importance to adjudge to each property and find an equilibrium. This decision may be based on the magnitude of the noise and the smoothness of the underlying function.

Observe that, when $2n-1 < \lambda < 2n$, for some $n\in\N$, the even rule ($i=0$) support is shorter than the odd ($i=1$) one, and just the opposite occurs when $2n < \lambda < 2n+1$. To simplify, we will discuss in detail the first case, where even and odd masks have lengths $2n-1$ and $2n$, respectively, since the second one is analogue and the Remark \ref{rmk:two_situations} can be taken into account for the consequent analysis. 
Nevertheless, we deal with both situations along the paper when it can be do it without additional effort.

To give a more explicit definition of the schemes, we solve the quadratic problem posed in \eqref{problema1} with $p=2$. In this case, it is a weighted least square problem and its solution is well-known.
Let us start with the derivation of the odd sub-mask, $\ba^1$, for $2n-1<\lambda<2n+1$.
For the sake of simplicity, we omit the dependence on $d,\omega,\lambda$ for the following vectors and matrices. If we denote as $\mathbf{W}^1$ the diagonal matrix consisting on the vector
\begin{equation}\label{vectorw1}
\mathbf{w}^1=(w^\lambda_{2n-1},\hdots,w^\lambda_{1},w^\lambda_{1},\hdots,w^\lambda_{2n-1}),
\end{equation}
we call
\begin{equation} \label{eq:x}
\bx^1=
\left(\begin{array}{c}
-2n+1\\
\vdots	 \\
-1\\
1\\
\vdots \\
2n-1\\
\end{array}\right), \quad
\bX^1=\left((\bx^1)^0, (\bx^1)^1, \ldots ,(\bx^1)^d \right)
=
\left(\begin{array}{c}
A(-2n+1)^T\\
\vdots	 \\
A(-1)^T\\
A(1)^T\\
\vdots \\
A(2n-1)^T\\
\end{array}\right)
,
\end{equation}
where the powers $(\bx^1)^t$, $t=0,\ldots,d$, are computed component-wisely, so that $\bX^1$ is a $2n\times (d+1)$ matrix,
and we denote $\bff^{1,j,k}=(f^k_{j-n+1},\hdots,f^k_{j},f^k_{j+1},\hdots,f^k_{j+n})^T$, then the problem of
\eqref{problema1}
can be write as:
\begin{equation}\label{equation111}
\hat \bbeta^1=\argmin_{\bbeta \in \mathbb{R}^{d+1}}||(\mathbf{W}^1)^\frac{1}{2}\bff^{1,j,k}-(\mathbf{W}^1)^\frac{1}{2}\mathbf{X}^1 \bbeta||_2^2,
\end{equation}
whose solution is
\begin{equation}\label{solucion111}
\hat \bbeta^1=((\mathbf{X}^1)^T\mathbf{W}^1\mathbf{X}^1)^{-1}(\mathbf{X}^1)^T\mathbf{W}^1\bff^{1,j,k}.
\end{equation}
For the sake of clarity, we write down the above terms:
\begin{equation} \label{eq:XWX1}
(\mathbf{X}^1)^T\mathbf{W}^1\mathbf{X}^1 =
\left(\begin{array}{llllll}
\sum_{i=-n+1}^{n} w^\lambda_{2i-1} & \sum_{i=-n+1}^{n} w^\lambda_{2i-1} (2i-1) & \cdots & \sum_{i=-n+1}^{n} w^\lambda_{2i-1} (2i-1)^d \\
\sum_{i=-n+1}^{n} w^\lambda_{2i-1} (2i-1) & \sum_{i=-n+1}^{n} w^\lambda_{2i-1} (2i-1)^2 & \cdots & \sum_{i=-n+1}^{n} w^\lambda_{2i-1} (2i-1)^{d+1} \\	
\vdots & \vdots & \vdots & \vdots \\	
\sum_{i=-n+1}^{n} w^\lambda_{2i-1} (2i-1)^d & \sum_{i=-n+1}^{n} w^\lambda_{2i-1} (2i-1)^{d+1} & \cdots & \sum_{i=-n+1}^{n} w^\lambda_{2i-1} (2i-1)^{2d}
\end{array}\right)
\end{equation}
and
\[
(\mathbf{X}^1)^T\mathbf{W}^1\bff^{1,j,k} = (\sum_{i=-n+1}^{n} w^\lambda_{2i-1}f^k_{j+i} , \sum_{i=-n+1}^{n} w^\lambda_{2i-1} (2i-1)f^k_{j+i},  \cdots, \sum_{i=-n+1}^{n} w^\lambda_{2i-1} (2i-1)^d f^k_{j+i})^T.
\]
Since we only need the first coordinate $\hat \beta^1_0$, we can use the Cramer's formula instead of solving the full system:
\[
(S_{d,\mathbf{w^\lambda}} \bff^k)_{2j+1}=\hat \beta^1_0 = \frac{
\left|\begin{array}{llllll}
\sum_{i=-n+1}^{n} w^\lambda_{2i-1}f^k_{j+i} & \sum_{i=-n+1}^{n} w^\lambda_{2i-1} (2i-1) & \cdots & \sum_{i=-n+1}^{n} w^\lambda_{2i-1} (2i-1)^d \\
\sum_{i=-n+1}^{n} w^\lambda_{2i-1} (2i-1)f^k_{j+i} & \sum_{i=-n+1}^{n} w^\lambda_{2i-1} (2i-1)^2 & \cdots & \sum_{i=-n+1}^{n} w^\lambda_{2i-1} (2i-1)^{d+1} \\	
\vdots & \vdots & \ddots & \vdots \\	
\sum_{i=-n+1}^{n} w^\lambda_{2i-1} (2i-1)^d f^k_{j+i} & \sum_{i=-n+1}^{n} w^\lambda_{2i-1} (2i-1)^{d+1} & \cdots & \sum_{i=-n+1}^{n} w^\lambda_{2i-1} (2i-1)^{2d}
\end{array}\right|
}{
|(\mathbf{X}^1)^T\mathbf{W}^1\mathbf{X}^1|}.
\]
Observe that, since the vector $\bw^1$ is symmetric, $w^\lambda_{2i-1} = w^\lambda_{1-2i}$, then $\sum_{i=-n+1}^{n} w^\lambda_{2i-1} (2i-1)^t = 0$ for any odd value of $t$, and $\sum_{i=-n+1}^{n} w^\lambda_{2i-1} (2i-1)^p = 2\sum_{i=1}^{n} w^\lambda_{2i-1} (2i-1)^t$ for the even values.
Thus, the above expressions can be simplified by placing many zeros and by shorting the range of the remaining sums.
Using the linearity of the determinant respect to the first column,
\[
(S_{d,\mathbf{w^\lambda}} \bff^k)_{2j+1} =
	\sum_{l=-n+1}^{n} \frac{w^\lambda_{2l-1}f^k_{j+l}}{|(\mathbf{X}^1)^T\mathbf{W}^1\mathbf{X}^1|}
	\left|\begin{array}{llllll}
		1 & \sum_{i=-n+1}^{n} w^\lambda_{2i-1} (2i-1) & \cdots & \sum_{i=-n+1}^{n} w^\lambda_{2i-1} (2i-1)^d \\
		2l-1 & \sum_{i=-n+1}^{n} w^\lambda_{2i-1} (2i-1)^2 & \cdots & \sum_{i=-n+1}^{n} w^\lambda_{2i-1} (2i-1)^{d+1} \\	
		\vdots & \vdots & \ddots & \vdots \\	
		(2l-1)^d & \sum_{i=-n+1}^{n} w^\lambda_{2i-1} (2i-1)^{d+1} & \cdots & \sum_{i=-n+1}^{n} w^\lambda_{2i-1} (2i-1)^{2d}
	\end{array}\right|
,
\]
we conclude that the sub-masks coefficients are
\[
a^1_l = |(\mathbf{X}^1)^T\mathbf{W}^1\mathbf{X}^1|^{-1}
w^\lambda_{2l-1}
\left|\begin{array}{llllll}
	1 & \sum_{i=-n+1}^{n} w^\lambda_{2i-1} (2i-1) & \cdots & \sum_{i=-n+1}^{n} w^\lambda_{2i-1} (2i-1)^d \\
	2l-1 & \sum_{i=-n+1}^{n} w^\lambda_{2i-1} (2i-1)^2 & \cdots & \sum_{i=-n+1}^{n} w^\lambda_{2i-1} (2i-1)^{d+1} \\	
	\vdots & \vdots & \ddots & \vdots \\	
	(2l-1)^d & \sum_{i=-n+1}^{n} w^\lambda_{2i-1} (2i-1)^{d+1} & \cdots & \sum_{i=-n+1}^{n} w^\lambda_{2i-1} (2i-1)^{2d}
\end{array}\right|.
\]
By \eqref{solucion111}, it can also be expressed as $\ba^1 = (\bbeta^1)^T \mathbf{e}_1 = \mathbf{W}^1\mathbf{X}^1((\mathbf{X}^1)^T\mathbf{W}^1\mathbf{X}^1)^{-1} \mathbf{e}_1$, where $\mathbf{e}_1$ is the first element of the canonical basis of $\R^{d+1}$.

Analogously, for $2n-2<\lambda<2n$, we can prove that $\ba^0 = \mathbf{W}^0\mathbf{X}^0((\mathbf{X}^0)^T\mathbf{W}^0\mathbf{X}^0)^{-1} \mathbf{e}_1$, so that
\[
a^0_l = |(\mathbf{X}^0)^T\mathbf{W}^0\mathbf{X}^0|^{-1}
w^\lambda_{2l}
\left|\begin{array}{llllll}
	1 & \sum_{i=-n+1}^{n-1} w^\lambda_{2i} (2i) & \cdots & \sum_{i=-n+1}^{n-1} w^\lambda_{2i} (2i)^d \\
	2l & \sum_{i=-n+1}^{n-1} w^\lambda_{2i} (2i)^2 & \cdots & \sum_{i=-n+1}^{n-1} w^\lambda_{2i} (2i)^{d+1} \\	
	\vdots & \vdots & \ddots & \vdots \\	
	(2l)^d & \sum_{i=-n+1}^{n-1} w^\lambda_{2i} (2i)^{d+1} & \cdots & \sum_{i=-n+1}^{n-1} w^\lambda_{2i} (2i)^{2d}
\end{array}\right|,
\]
where
$\mathbf{W}^0$ is the diagonal matrix with diagonal
\begin{equation}\label{vectorw0}
\mathbf{w}^0=(w^\lambda_{2n-2},\hdots,w^\lambda_2,1,w^\lambda_2,\hdots,w^\lambda_{2n-2}),
\end{equation}
and
\begin{equation*}
\bx^0=\left( -2(n-1), \ldots, 2,0,2, \ldots, 2(n-1)
\right)^T, \quad
\bX^0=\left((\bx^0)^0, (\bx^0)^1, \ldots ,(\bx^0)^d \right).
\end{equation*}
Collecting these developments, for $2n-1<\lambda<2n$, we can define our weighted local polynomial regression-based subdivision as:
\begin{equation}\label{deffinal}
(S_{d,\mathbf{w^\lambda}} f^k)_{2j+i}= \sum_{l=1-n}^{n-1+i} a^i_l f^k_{j+l}, \quad i=0,1.
\end{equation}

A direct consequence, by construction, is that the scheme reproduces polynomials up to degree $d$.
\begin{proposition} \label{prop:reproduction}
The scheme $S_{d,\mathbf{w^\lambda}}$ reproduces $\Pi_d$.
\end{proposition}

	\begin{remark}
		Observe that we have considered $\{1,x,\ldots,x^d\}$ as basis of $\Pi_d$, which has led a the linear system with matrix \eqref{eq:XWX1}. It is possible to consider an orthonormal basis of $\Pi_d$ in a way that the matrix is diagonal, leading to a cleaner mathematical description. However, we preferred the basis $\{1,x,\ldots,x^d\}$ because the resulting expression of the subdivision operator is more explicit. A possible benefit of considering an orthonormal basis is that the next results might be more intuitive.
	\end{remark}

Now we prove that $(\bW^i)^{-1}\ba^i$, $i=0,1$, are exactly the evaluations of some polynomial at the grid points $\bx^i$.
\begin{lemma} \label{lem:eval_poly}
For $i=0,1$, the sub-masks are
\begin{equation} \label{eq:eval_poly}
	\ba^i = \mathbf{W}^i \mathbf{X}^i \balpha^i = \left \{w^\lambda_{2j-i}\sum_{t=0}^d \alpha^i_t x^t_j\right \}_{j=1-n}^{L_n+i}
\end{equation}
That is,
the vector \( (\bW^i)^{-1} \ba^i \) coincides with the evaluation of the polynomial $A(x)^T \balpha^i$ at the points $\bx^i$, being \(\balpha^i=((\mathbf{X}^i)^T\mathbf{W}^i\mathbf{X}^i)^{-1}\mathbf{e}_1\), which expression depends on $n,\lambda,\omega$.
\end{lemma}
\begin{proof}
By the previous computations,
\[
\ba^i = \mathbf{W}^i \mathbf{X}^i ((\mathbf{X}^i)^T\mathbf{W}^i\mathbf{X}^i)^{-1}\mathbf{e}_1 = \mathbf{W}^i \mathbf{X}^i \balpha^i.
\]
For $\ba^1$ (for $\ba^0$ is analogous), using \eqref{eq:x} we obtain
\[
(\bW^1)^{-1} \ba^1 = \bX^1 \balpha^1 =
\left(\begin{array}{c}
	A(-2n+1)^T\\
	\vdots	 \\
	A(-1)^T\\
	A(1)^T\\
	\vdots \\
	A(2n-1)^T\\
\end{array}\right)\balpha^1
	=
	\left(\begin{array}{c}
		A(-2n+1)^T \balpha^1\\
		\vdots	 \\
		A(-1)^T \balpha^1\\
		A(1)^T \balpha^1\\
		\vdots \\
		A(2n-1)^T \balpha^1\\
	\end{array}\right).
	\]
	That is, the coordinates of \((\bW^1)^{-1} \ba^1\) are the evaluations of the polynomial $A(x)^T \balpha^1 $ at the $\bx^1$ grid points.
\end{proof}

Moreover, these sub-masks are the only ones that lead to polynomial reproduction and verify that $(\bW^i)^{-1}\ba^i$ are polynomial evaluations. This property can be used in practice to easily determine the sub-masks, as we do in Section \ref{sec:d23}.

\begin{theorem} \label{thm:unique}
	The scheme $S_{d,\mathbf{w^\lambda}}$ is the unique scheme that reproduces $\Pi_d$ polynomials and its sub-masks have the form $\ba^i = \mathbf{W}^1 \mathbf{X}^1 \balpha^i$, for some $\balpha^i\in\R^{d+1}$, $i=0,1$.
\end{theorem}
\begin{proof}
	It is a consequence of Lemma \ref{lem:eval_poly} together with Proposition \ref{prop:reproduction}. Suppose that some rule $\hat \ba = \mathbf{W}^i \mathbf{X}^i \hat\balpha$, for some $\hat\balpha\in\R^{d+1}$, fulfils the reproduction conditions for $\Pi_d$. Then
	\[
	\sum_j \hat a_j (x^i_j)^t = \delta_{0,t}, \qquad t=0,1,\ldots,d,  \quad i=0,1,
	\]
	or, written with matrix multiplications,
	\(
	(\bX^i)^T \hat\ba =  \mathbf{e}_1.
	\)
	Then,
	\[
	(\bX^i)^T \hat\ba = (\bX^i)^T \bW^i \bX^1 \hat \balpha = \mathbf{e}_1 \to \hat \balpha = ((\bX^i)^T \bW^i \bX^i)^{-1}\mathbf{e}_1 = \balpha^i.
	\]
\end{proof}

The symmetry of the scheme is another consequence of being based on a polynomial regression problem.
\begin{lemma} \label{lem:symmetry}
	The scheme $S_{d,\mathbf{w^\lambda}}$ is odd-symmetric.
\end{lemma}
\begin{proof}
	We prove that \(a^1_{j} = a^1_{1-j}\) for $2n-1<\lambda<2n+1$ (it can be analogously proven that \(a^0_{j} = a^0_{-j}\) for $2n-2<\lambda<2n$). Let us consider $\bff^j = \{\delta_{j,l} \ : \ l\in\{-n+1,\ldots,n\} \}$.
	The coordinates of the sub-mask $\ba^1$ can be obtained by applying the rule to $\bff^j$, for $j=-n+1,\ldots,n$, and take the first coordinate,
	\[
	a^1_j = \sum_{l=-n+1}^n a^1_{l}\delta_{j,l} = \sum_{l=-n+1}^n a^1_{l}f^j_{l} = (S_{d,\mathbf{w^\lambda}}\bff^j)_1 = \hat{p}^j(0),
	\]
	where, by \eqref{problema1} and \eqref{esquemasubdivision},
	\begin{equation} \label{eq:pj}
		\hat{p}^j=\argmin_{p \in \Pi_{d}(\mathbb{R})} \sum_{l=-n+1}^n w^\lambda_{2l-1}(\delta_{j,l}-p(2l-1))^2.
	\end{equation}
	Then, $a^1_{j} = a^1_{1-j}$ provided that $(S_{d,\mathbf{w^\lambda}}\bff^j)_1 = (S_{d,\mathbf{w^\lambda}}\bff^{1-j})_1 $, or in other words, $\hat p^j(0) = \hat p^{1-j}(0)$.
	Observe that,
	\begin{equation*}
		\begin{split}
			\hat{p}^{1-j}=\argmin_{q \in \Pi_{d}(\mathbb{R})} \sum_{l=-n+1}^n w^\lambda_{2l-1}(\delta_{1-j,l}-q(2l-1))^2,
		\end{split}
	\end{equation*}
	and, performing the change in the summation index $l$ by $1-l$ and using $w^\lambda_{2l-1} = w^\lambda_{1-2l}$ and $\delta_{j,l} = \delta_{1-j,1-l}$,
	\begin{equation}
		\hat{p}^{1-j}=\argmin_{q \in \Pi_{d}(\mathbb{R})} \sum_{l=-n+1}^n w^\lambda_{1-2l}(\delta_{1-j,1-l}-q(1-2l))^2
		=\argmin_{q \in \Pi_{d}(\mathbb{R})} \sum_{l=-n+1}^n w^\lambda_{2l-1}(\delta_{j,l}-q(1-2l))^2. \label{eq:p1j}
	\end{equation}
	Observe the similarity between \eqref{eq:pj} and \eqref{eq:p1j}. Since the minimum is unique, it is reached in \eqref{eq:p1j} by $\hat p^{j}(-t)$. Thus $\hat p^{1-j}(t) = \hat p^{j}(-t)$ and, then, $a^1_j = \hat p^j(0) = \hat p^{1-j}(0) = a^1_{1-j}$.
\end{proof}

By Lemma \ref{lem:eval_poly}, we know that \( (\bW^1)^{-1} \ba^i \) are the evaluations of a polynomial at $\bx^i$. To take profit of the symmetry, let us write as in \eqref{eq:eval_poly}:
\begin{equation} \label{eq:eval_poly2}
(\bW^0)^{-1} \ba^0 = \left\{ \sum_{t=0}^d \alpha^0_t (2l)^t \right\}_{l=-n+1}^{n-1}, \quad (\bW^1)^{-1} \ba^1 = \left\{ \sum_{t=0}^d \alpha^1_t (2l-1)^t \right\}_{l=-n+1}^n.
\end{equation}
Since \(a^0_{i} = a^0_{-i}\), $\forall i=-n+1,\ldots,n-1$, and \(a^1_{i} = a^1_{1-i}\), $\forall i=-n+1,\ldots,n$, and $\omega$ is even, it can be deduced that the polynomials only have even powers. That is
\begin{equation} \label{eq:odd_is_zero}
	\alpha^i_{2t-1} = 0, \quad \forall 1 \leq t \leq (d+1)/2.
\end{equation}

A direct consequence is that the subdivision schemes obtained for any weight function of degree $d$ (even number) coincides with the one for $d+1$, proven in the following lemma.

\begin{proposition} \label{prop:even_odd}
	Let $\omega$ a weight function, $d\in 2\Z_{+}$ and $\lambda\in\R_+\backslash\N$ be such that $d \leq -2 + 2\left \lfloor \frac{\lambda+1}2 \right \rfloor$,
then
	\[
	S_{d,\mathbf{w^\lambda}} = S_{d+1,\mathbf{w^\lambda}}.
	\]
\end{proposition}
\begin{proof}
	The sub-masks of $S_{d+1,\mathbf{w^\lambda}}$ can be written in terms of the evaluation of a $(d+1)$-degree polynomial, according to Lemma \ref{lem:eval_poly}. Since the odd coefficients are zero, then the leading coefficient is zero, for both rules $i=0,1$. Then, both $S_{d,\mathbf{w^\lambda}}$ and $S_{d+1,\mathbf{w^\lambda}}$ fulfils the conditions of Theorem \ref{thm:unique}, for the same polynomial degree $d$, hence they must coincide. 
\end{proof}

Therefore, we can just study the properties of the subdivision schemes based on the space of polynomials $\Pi_d(\mathbb{R})$ with $d$ an even number.

\section{WLPR-Subdivision schemes for $d=0,1$} \label{sec:d01}

In this section we present the WLPR-Subdivision schemes for $d=0,1$ and their properties, by Proposition \ref{prop:even_odd}, we can just consider $d=0$. To simplify the notation, in this section we omit $d$, $\bw$ and $\lambda$ in some variables, such as $S := S_{0,\bw^\lambda}= S_{1,\bw^\lambda}$. In this case, the coefficients of the subdivision schemes are easily obtained from $\omega$ thanks to Lemma \ref{lem:eval_poly}:					
If we denote as
$||\bw^i||_1$ the sum of the components of the vector $\bw^i$ with $i=0,1$, defined in \eqref{vectorw0} and \eqref{vectorw1},
\begin{equation}\label{normaws}
	||\bw^0||_1=1+2\sum_{l=1}^{n-1}w_{2l}^\lambda,\quad ||\bw^1||_1=2\sum_{l=0}^{n-1}w_{2l+1}^\lambda,
\end{equation}
\[
\ba^i = \mathbf{W}^i \mathbf{X}^i \balpha^i = \bw^i \balpha^i, \quad \balpha^i = ((\bX^i)^T \bW^i \bX^i)^{-1}\mathbf{e}_1 = ||\bw^i||_1^{-1},
\]
thus $\ba^i = \bw^i/||\bw^i||_1$.
Another way to obtain $\balpha^i$ is based on Theorem \ref{thm:unique}: Since $\ba^i = \bw^i \balpha^i$ and the scheme must reproduce $\Pi_0$ (constant functions), then $1 = \sum_j a^i_j = \balpha^i \|\bw^i\|_1$ by Lemma \ref{lem:sum1}, thus $\balpha^i= ||\bw^i||_1^{-1}$.

The explicit form of the resulting WLPR-subdivision scheme is, if $2n-1<\lambda<2n$,
\begin{equation}\label{eq:definition_d1_1}
	(S f^k)_{2j}=\sum_{l=1-n}^{n-1} \left(\frac{w^\lambda_{2l}}{||\bw^0||_1}\right) f^k_{j+l},\quad (S f^k)_{2j+1}=
	\sum_{l=1-n}^{n}\left(\frac{w_{2l-1}^\lambda}{||\bw^1||_1}\right)f^k_{j+l},
\end{equation}
and, for $2n<\lambda<2n+1$, it can be written in the following way, in agreement with Remark \ref{rmk:two_situations},
\begin{equation}\label{eq:definition_d1_2}
	(S f^k)_{2j}=
	\sum_{l=1-n}^{n}\left(\frac{w_{2l-1}^\lambda}{||\bw^1||_1}\right)f^k_{j+l},
	\quad
	(S f^k)_{2j+1}=\sum_{l=1-n}^{n+1} \left(\frac{w^\lambda_{2l-2}}{||\bw^0||_1}\right) f^k_{j+l}.
\end{equation}

Note that if $\lambda\in (1,2)$ then $\bw^0=1$ and $\bw^1=(\frac12,\frac12)$, so that the mask for any function $\omega$ of the subdivision scheme is $\ba=[1,2,1]/2$, in other words, the interpolatory Deslauriers-Dubuc scheme \cite{desdub} (as stated in Proposition \ref{propnumero}).
For $\lambda>2$, if $\omega(x)=1$ for $|x|\leq 1$, then the schemes presented by Dyn et al. in \cite{dynheardhormannsharon} are recovered as we mentioned above. These schemes are for $2n-1<\lambda<2n$:
\begin{equation}\label{eqeje}
	(S_{{\tt rect}} f^k)_{2j+1}=
	\sum_{l=1-n}^{n}\frac{f^k_{j+l}}{2n},\quad(S_{\mathbf{\tt rect}} f^k)_{2j}=\sum_{l=-n+1}^{n-1}\frac{f^k_{j+l}}{2n-1}.
\end{equation}
We list some masks for the weight function $\omega(x)=1-|x|$, $|x|\leq 1$, and for several values of $\lambda$:
\begin{equation*}
	\begin{split}
		\ba_{0,\mathbf{tria^{1.5}}}&=[1,2,1]/2,\\
		\ba_{0,\mathbf{tria^{2.5}}}&=[1/7, 1/2, 5/7, 1/2, 1/7],\\
		\ba_{0,\mathbf{tria^{3.5}}}&=[1/12,           3/13,           5/12,           7/13,           5/12,          3/13,           1/12],\\
		\ba_{0,\mathbf{tria^{4.5}}}&=[1/21,           3/20,           5/21,           7/20,           3/7,            7/20,           5/21,           3/20,           1/21],\\
		\ba_{0,\mathbf{tria^{5.5}}}&=[ 1/30,           3/31,           1/6,           7/31,           3/10,          11/31,           3/10, 7/31,           1/6,            3/31,           1/30].
	\end{split}
\end{equation*}
As we can see, all subdivision schemes in this section present a positive mask, since $\omega$ is a positive function. Then, the following result on convergence proved in \cite{gonsor}, (see also \cite{michelliprautzsch,Zhou}) can be applied.
\begin{proposition}{(\cite{gonsor})}\label{propgonsor}
	Let $\ba=\{a_l\}_{l\in\mathbb{Z}}$ be a mask with support $[q,q+k]$, being $q$ and $k$ fixed integers, $k\geq 3$. Suppose that $a_q,a_{q+1},\ldots,a_{q+k-1},a_{q+k}>0$ and
	$\sum_{l\in\Z}a_{2l}=\sum_{l\in\Z}a_{2l+1}=1,$
	then the subdivision scheme converges.
\end{proposition}
As a direct consequence, the schemes in this section,
\eqref{eq:definition_d1_1} and \eqref{eq:definition_d1_2}, are convergent because the masks are positive. Observe that the condition $k\geq 3$ in Proposition \ref{propgonsor} requires considering $\lambda>2$.
\begin{corollary}\label{cor1}
	The subdivision scheme $S_{1,\bw^\lambda}$, defined in
\eqref{eq:definition_d1_1} or \eqref{eq:definition_d1_2}, is convergent for any $\lambda\in (1,+\infty)\backslash\N$ and any positive function $\omega$ with support $[-1,1]$.
\end{corollary}
In Figure \ref{fig:exp2} we show some examples of the limit functions for some weight functions, $\lambda \in \{3.2,3.4,3.6,3.8\}$ and $\bff^0 = \{\delta_{0,l}\}_{l\in\Z}$. The support of all these limit functions is $[-3,3]$ because the mask support length does not vary.
\begin{figure}[h]
	\begin{center}
		\begin{tabular}{ccc}
			\includegraphics[width=5.0cm,height=3.8cm]{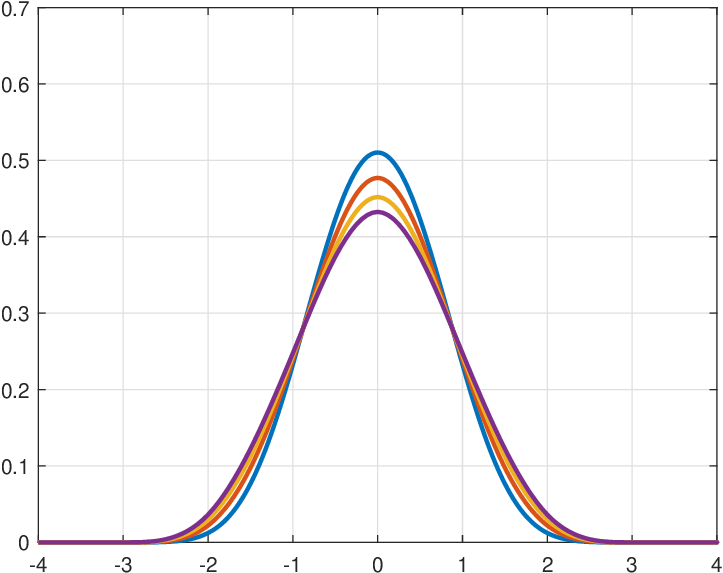} & \includegraphics[width=5.0cm,height=3.8cm]{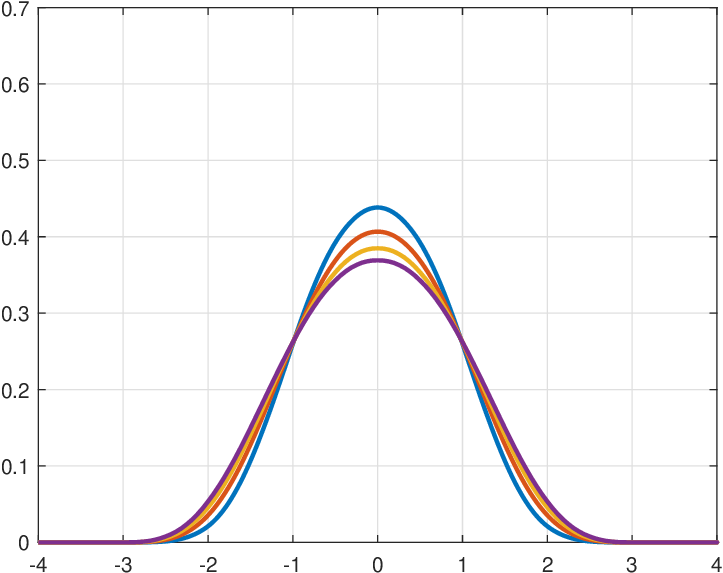} &   \includegraphics[width=5.0cm,height=3.8cm]{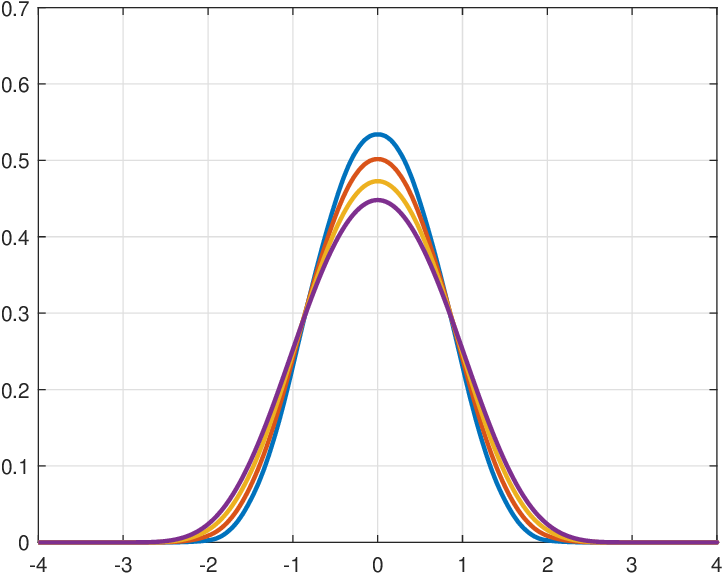}\\
			{\tt tria} & {\tt epan} & {\tt bisq} \\
			\includegraphics[width=5.0cm,height=3.8cm]{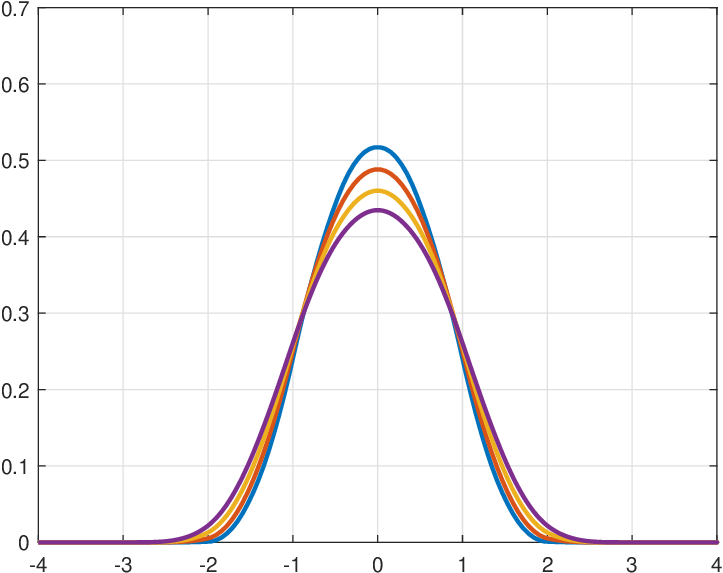}&  \includegraphics[width=5.0cm,height=3.8cm]{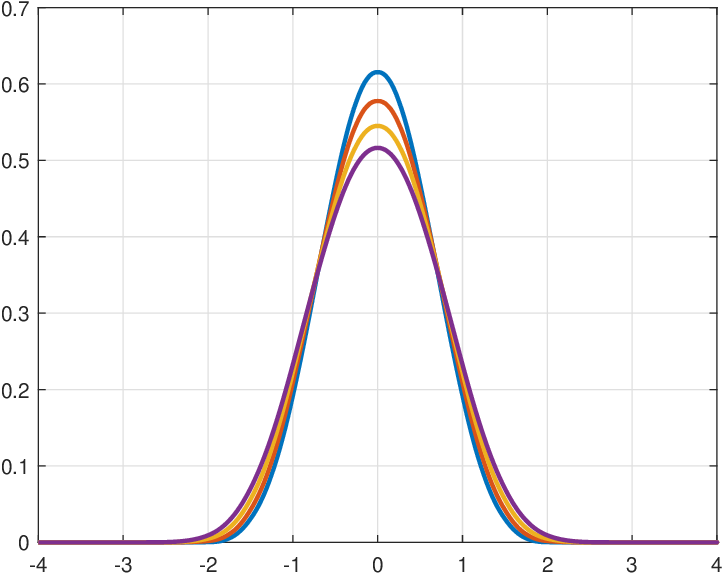}& \includegraphics[width=5.0cm,height=3.8cm]{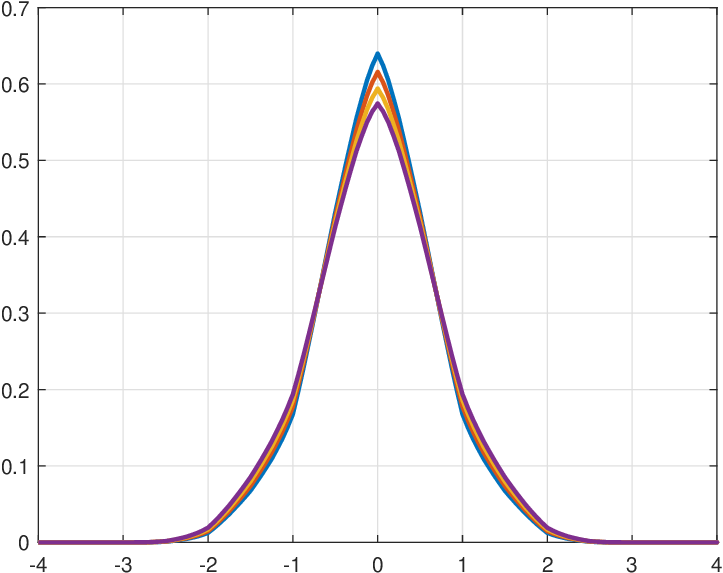}\\
			{\tt tcub} & {\tt trwt} & {\tt exp}(3)\\
		\end{tabular}
	\end{center}
	\caption{Limit functions of the subdivision schemes $S_{1,\bw^\lambda}$ for some weight functions (see Table \ref{tabla1nucleos}) and $\lambda=3.2$ (blue), $\lambda=3.4$ (orange), $\lambda=3.6$ (yellow) and $\lambda=3.8$ (purple).}
	\label{fig:exp2}
\end{figure}

To analyse the smoothness of the limit functions generated by $S$, we consider the Theorem \ref{teoremadyn1}.
In particular, we will prove that the mask of the difference scheme $S_{\mathbf{q}}$ is positive and apply again Proposition \ref{propgonsor}.
Thanks to the odd-symmetry of the scheme, the study can be reduced to a half of its coefficients.
\begin{lemma}\label{lema1}
	Let $n$ be a natural number, $n\geq 2$, $\lambda\in(2n-1,2n)$ and $\omega$ a weight function. The coefficients of the difference scheme $S_\bq$ are positive if
	\begin{equation*}
		\frac{\sum_{l=j_0}^{n-1}w^\lambda_{2l+1}}{\sum_{l=j_0}^{n-1} w^\lambda_{2l}} < \frac{||\bw^1||_1}{||\bw^0||_1}<\frac{\sum_{l=j_1}^{n-1}w^\lambda_{2l+1}}{\sum_{l=j_1+1}^{n-1} w^\lambda_{2l}}, \qquad j_0=1,\hdots,n-1, \quad j_1=1,\hdots,n-2.
	\end{equation*}
\end{lemma}
\begin{proof}
	
	By Lemma \ref{lem:even_difference}, $S_\bq$ is even-symmetric. Since $2n-1<\lambda<2n$, then $L_n = n-1$ in \eqref{eq:q_symmetry} and we have
	\begin{equation} \label{eq:q_symmetrydd}
		q^0_j=q^1_{-j}, \qquad j = 1-n, \ldots, n-1.
	\end{equation}
	Then, if $q^0_{j}>0$, for $j=1-n,\ldots,0$, and $q^1_{j}>0$, for $j=1-n,\ldots,-1$, the result is proved.

	First, we check that the coefficients $q^0_0$ and $q^1_{1-n}$ are always positive.
	$$ q^0_0 =\sum_{l=-n+1}^{0} \frac{w^\lambda_{2l}}{||\bw^0||_1} - \frac{w^\lambda_{2l-1}}{||\bw^1||_1}=\sum_{l=-n+1}^{0} \frac{w^\lambda_{2l}}{||\bw^0||_1} - \frac{1}{2}=\frac{1 + \sum_{l=-n+1}^{-1} w^\lambda_{2l}}{1+2\sum_{l=-n+1}^{-1}w^\lambda_{2l}} - \frac{1}{2}>0,$$
	since $w^\lambda_0=1$ and $||\bw^0||_1=1+2\sum_{l=-n+1}^{-1}w^\lambda_{2l}.$
	Analogously,  by
\eqref{eq:difference_scheme}, we have that
	$$q^1_{1-n}=\sum_{l=1-n}^{n-1} \frac{w^\lambda_{2l}}{||\bw^0||_1} - \frac{w^\lambda_{2l+1}}{||\bw^1||_1}=1-\sum_{l=1-n}^{n-1} \frac{w^\lambda_{2l+1}}{||\bw^1||_1}=\frac{w^\lambda_{1-2n}}{||\bw^1||_1}>0.$$

	Now we check $q^0_{j}>0$, for $j=1-n,\ldots,-1$. From
\eqref{eq:difference_scheme}, we have that
	$$	q^0_{j} =
\sum_{l=-n+1}^{j}a^{n,0}_l-a^{n,1}_l=\sum_{l=-n+1}^{j} \frac{w^\lambda_{2l}}{||\bw^0||_1} - \frac{w^\lambda_{2l-1}}{||\bw^1||_1}, \qquad j = 1-n, \ldots, -1,\\
	$$
	then
	\begin{equation}\label{eqlemad1}
		0<q^0_{j}=\sum_{l=-n+1}^{j} \frac{w^\lambda_{2l}}{||\bw^0||_1} - \frac{w^\lambda_{2l-1}}{||\bw^1||_1} \Leftrightarrow \sum_{l=-n+1}^{j}\frac{w^\lambda_{2l-1}}{||\bw^1||_1} < \sum_{l=-n+1}^{j} \frac{w^\lambda_{2l}}{||\bw^0||_1} \Leftrightarrow \frac{\sum_{l=-n+1}^{j}w^\lambda_{2l-1}}{\sum_{l=-n+1}^{j} w^\lambda_{2l}} < \frac{||\bw^1||_1}{||\bw^0||_1}.
	\end{equation}
	As $w^\lambda_l=w^\lambda_{-l}$ for all $l\in\mathbb{Z}$, we have that, if $j=1-n,\hdots, -1$,
	\begin{equation}\label{eqparafinal}
		\sum_{l=-n+1}^{j}w^\lambda_{2l-1}=\sum_{l=-n+1}^{j}w^\lambda_{1-2l}=\sum_{l=-j}^{n-1}w^\lambda_{2l+1},\quad \sum_{l=-n+1}^{j}w^\lambda_{2l}=\sum_{l=-n+1}^{j}w^\lambda_{-2l}=\sum_{l=-j}^{n-1}w^\lambda_{2l}.
	\end{equation}
	Therefore, by \eqref{eqlemad1} we obtain:
	\begin{equation}\label{parafinal1}
		0<q^0_{j} \Leftrightarrow \frac{\sum_{l=j}^{n-1}w^\lambda_{2l+1}}{\sum_{l=j}^{n-1} w^\lambda_{2l}} < \frac{||\bw^1||_1}{||\bw^0||_1}, \quad j=1,\hdots,n-1.
	\end{equation}
		
	Now we check $q^1_{j}>0$, for $j=2-n,\ldots,-1$. By \eqref{eq:difference_scheme}:
	\begin{equation}
		\begin{split}
			q^1_{j} &
			= \sum_{l=j}^{n-1} a^{n,0}_{l} - a^{n,1}_{l+1} = \sum_{l=j}^{n-1}  \frac{w^\lambda_{2l}}{||\bw^0||_1} - \sum_{l=j}^{n-1}\frac{w^\lambda_{2l+1}}{||\bw^1||_1} \qquad \\
			&=1-\sum_{l=1-n}^{j-1}  \frac{w^\lambda_{2l}}{||\bw^0||_1}-\left(1-\sum_{l=-n}^{j-1}\frac{w^\lambda_{2l+1}}{||\bw^1||_1}\right)\\
			&=\sum_{l=-n}^{j-1}\frac{w^\lambda_{2l+1}}{||\bw^1||_1} -\sum_{l=1-n}^{j-1}  \frac{w^\lambda_{2l}}{||\bw^0||_1}.
		\end{split}
	\end{equation}
	Following the same reasoning:
	\begin{equation}\label{eqlemad2}
		0<q^1_{j}=\sum_{l=-n}^{j-1}\frac{w^\lambda_{2l+1}}{||\bw^1||_1} -\sum_{l=1-n}^{j-1}  \frac{w^\lambda_{2l}}{||\bw^0||_1} \Leftrightarrow \sum_{l=1-n}^{j-1}  \frac{w^\lambda_{2l}}{||\bw^0||_1} < \sum_{l=-n}^{j-1}\frac{w^\lambda_{2l+1}}{||\bw^1||_1}  \Leftrightarrow  \frac{||\bw^1||_1}{||\bw^0||_1}< \frac{\sum_{l=-n}^{j-1}w^\lambda_{2l+1}}{ \sum_{l=1-n}^{j-1}  w^\lambda_{2l}}.
	\end{equation}
	Again, by \eqref{eqparafinal}, we have
	\begin{equation}\label{parafinal2}
		0<q^1_{j}\Leftrightarrow\frac{||\bw^1||_1}{||\bw^0||_1}<
		\frac{\sum_{l=1-n}^{j}w^\lambda_{2l-1}}{ \sum_{l=1-n}^{j-1}  w^\lambda_{2l}}=\frac{\sum_{l=-j}^{n-1}w^\lambda_{2l+1}}{\sum_{l=1-j}^{n-1} w^\lambda_{2l}}, \quad j=2-n, \ldots, -1.
	\end{equation}
Collecting conditions \eqref{parafinal1} and \eqref{parafinal2}, we get the result:
	\begin{equation*}
		\frac{\sum_{l=j_0}^{n-1}w^\lambda_{2l+1}}{\sum_{l=j_0}^{n-1} w^\lambda_{2l}} < \frac{||\bw^1||_1}{||\bw^0||_1}<\frac{\sum_{l=j_1}^{n-1}w^\lambda_{2l+1}}{\sum_{l=j_1+1}^{n-1} w^\lambda_{2l}},
	\end{equation*}
	with $j_0=1,\hdots,n-1$ and $j_1=1,\hdots,n-2$.
\end{proof}

\begin{lemma}
	Let $n\in\N$, $n\geq 2$, $\lambda\in (2n-1,2n)$, and $\omega$ a weight function be. Let us consider
	\begin{equation}
		p^{\omega^\lambda}_0:[0,n-1]\to \mathbb{R}, \quad p^{\omega^\lambda}_0(l):=\frac{\phi\left(\frac {2l+1}{\lambda}\right)}{\phi\left(\frac {2l}{\lambda}\right)},
	\end{equation}
so that $p^{\omega^\lambda}_0(l)=\frac{w^\lambda_{2l+1}}{w^\lambda_{2l}}$ for $l=0,1,\ldots,n-1$.
	If $p^{\omega^\lambda}_0$ is a decreasing function, then the coefficients of the difference scheme are positive.
\end{lemma}
\begin{proof}
	Note that:
	$$\frac{||\bw^1||_1}{||\bw^0||_1}=\frac{2\sum_{l=0}^{n-1}w^\lambda_{2l+1}}{1+2\sum_{l=0}^{n-1} w^\lambda_{2l}}=\frac{\sum_{l=0}^{n-1}w^\lambda_{2l+1}}{\frac12+\sum_{l=0}^{n-1} w^\lambda_{2l}}$$
	Consider this basic property: For any $a,b,c,d>0$,
	\begin{equation}\label{propiedadbasica}
		\frac ab\leq\frac cd \Rightarrow \frac ab\leq \frac{a+c}{b+d}\leq \frac cd.
	\end{equation}
	Firstly, since $p^{\omega^\lambda}_0$ is decreasing, we get by \eqref{propiedadbasica}:
	$$\frac{w^\lambda_{2n-1}}{w^\lambda_{2n-2}}=p^{\omega^\lambda}_0(n-1)\leq p^{\omega^\lambda}_0(n-2)=\frac{w^\lambda_{2n-3}}{w^\lambda_{2n-4}}\Rightarrow \frac{w^\lambda_{2n-1}}{w^\lambda_{2n-2}}\leq \frac{w^\lambda_{2n-1}+w^\lambda_{2n-3}}{w^\lambda_{2n-2}+w^\lambda_{2n-4}}\leq \frac{w^\lambda_{2n-3}}{w^\lambda_{2n-4}}.$$
	And, again using the monotony of function $p^{\omega^\lambda}_0$ and \eqref{propiedadbasica}:
	$$\frac{w^\lambda_{2n-1}+w^\lambda_{2n-3}}{w^\lambda_{2n-2}+w^\lambda_{2n-4}}\leq \frac{w^\lambda_{2n-3}}{w^\lambda_{2n-4}}\leq \frac{w^\lambda_{2n-5}}{w^\lambda_{2n-6}}\Rightarrow
	\frac{w^\lambda_{2n-1}+w^\lambda_{2n-3}}{w^\lambda_{2n-2}+w^\lambda_{2n-4}}\leq \frac{w^\lambda_{2n-1}+w^\lambda_{2n-3}+w^\lambda_{2n-5}}{w^\lambda_{2n-2}+w^\lambda_{2n-4}+w^\lambda_{2n-6}}
	\leq \frac{w^\lambda_{2n-5}}{w^\lambda_{2n-6}}
	$$
	Repeating this process, we get by \eqref{propiedadbasica}:
	\begin{equation}
		\begin{split}
			&\frac{w^\lambda_{2n-1}}{w^\lambda_{2n-2}}\leq \frac{w^\lambda_{2n-1}+w^\lambda_{2n-3}}{w^\lambda_{2n-2}+w^\lambda_{2n-4}}\leq \hdots\leq
			\frac{\sum_{l=1}^{n-1}w^\lambda_{2l+1}}{\sum_{l=1}^{n-1}w^\lambda_{2l}}\leq \frac{w^\lambda_{1}}{w^\lambda_{0}}\Rightarrow\\
			&\frac{w^\lambda_{2n-1}}{w^\lambda_{2n-2}}\leq \frac{w^\lambda_{2n-1}+w^\lambda_{2n-3}}{w^\lambda_{2n-2}+w^\lambda_{2n-4}}\leq \hdots
			\leq \frac{\sum_{l=0}^{n-1}w^\lambda_{2l+1}}{w^\lambda_0+\sum_{l=1}^{n-1}w^\lambda_{2l}} <
			\frac{\sum_{l=0}^{n-1}w^\lambda_{2l+1}}{\frac{1}{2}+\sum_{l=1}^{n-1}w^\lambda_{2l}}=\frac{||\bw^1||_1}{||\bw^0||_1}.
		\end{split}
	\end{equation}
	Secondly, we define $p^{\omega^\lambda}_1:[0,n-2]\to \mathbb{R}$ as
\[
	p_1^{\omega^\lambda}(l)= 1/p_0^{\omega^\lambda}(l+1/2), \quad \forall l\in[0,n-2],
	\]
which is an increasing function since $p^{\omega^\lambda}_0$ is decreasing. We have that
	$$\frac{w^\lambda_{2n-5}}{w^\lambda_{2n-4}}=p^{\omega^\lambda}_1(n-3)\leq p^{\omega^\lambda}_1(n-2)=\frac{w^\lambda_{2n-3}}{w^\lambda_{2n-2}}<\frac{w^\lambda_{2n-3}+w^\lambda_{2n-1}}{w^\lambda_{2n-2}}\Rightarrow \frac{w^\lambda_{2n-5}}{w^\lambda_{2n-4}}<\frac{w^\lambda_{2n-1}+w^\lambda_{2n-3}+w^\lambda_{2n-5}}{w^\lambda_{2n-2}+w^\lambda_{2n-4}}<\frac{w^\lambda_{2n-3}+w^\lambda_{2n-1}}{w^\lambda_{2n-2}}.$$
	Again, using the same strategy, we get:
	\begin{equation*}
		\begin{split}
			&\frac{w^\lambda_{1}}{w^\lambda_{2}}<\frac{\sum_{l=1}^{n-1}w^\lambda_{2l+1}}{\sum_{l=2}^{n-1}w^\lambda_{2l}}<\hdots<\frac{w^\lambda_{2n-3}+w^\lambda_{2n-1}}{w^\lambda_{2n-2}}\Rightarrow
			\frac{\sum_{l=0}^{n-1}w^\lambda_{2l+1}}{\sum_{l=1}^{n-1}w^\lambda_{2l}}<\frac{\sum_{l=1}^{n-1}w^\lambda_{2l+1}}{\sum_{l=2}^{n-1}w^\lambda_{2l}}<\hdots<\frac{w^\lambda_{2n-3}+w^\lambda_{2n-1}}{w^\lambda_{2n-2}}\Rightarrow\\
			&\frac{||\bw^1||_1}{||\bw^0||_1}=\frac{\sum_{l=0}^{n-1}w^\lambda_{2l+1}}{\frac12+\sum_{l=1}^{n-1}w^\lambda_{2l}}<\frac{\sum_{l=0}^{n-1}w^\lambda_{2l+1}}{\sum_{l=1}^{n-1}w^\lambda_{2l}}<\frac{\sum_{l=1}^{n-1}w^\lambda_{2l+1}}{\sum_{l=2}^{n-1}w^\lambda_{2l}}<\hdots<\frac{w^\lambda_{2n-3}+w^\lambda_{2n-1}}{w^\lambda_{2n-2}}.\\
		\end{split}
	\end{equation*}
	Then, by Lemma \ref{lema1}, we conclude that the coefficients of the difference scheme, \eqref{eq:difference_scheme}, are positive.
\end{proof}

Next lemma allows to easily check the monotonicity of $p^{\omega^\lambda}_0$.
\begin{lemma}
	Let $n\geq 2$ and $\lambda\in(2n-1,2n)$ be.
	If $\phi:[0,1]\to [0,1]$ is continuous and differentiable in $(0,1)$ and the quotient function $\phi'/\phi$ is decreasing, then $p_0^{\omega^\lambda}$ is decreasing.
\end{lemma}
\begin{proof}
By hypothesis, the function $p_0^{\omega^\lambda}(l) = \frac{\phi((2l+1)/\lambda)}{\phi(2l/\lambda)}$ is continuous for $[0,n-1]$ and differentiable in $(0,n-1)$ (observe that we are considering $l$ a real number here). Hence, it is decreasing provided that its derivative is negative. In addition,
\begin{align*}
&p_0^{\omega^\lambda\prime}(l) = \frac2\lambda\frac{\phi'((2l+1)/\lambda)\phi(2l/\lambda)-\phi((2l+1)/\lambda)\phi'(2l/\lambda)}{\phi(2l/\lambda)^2} <0 \\
\Leftrightarrow \quad & \phi'((2l+1)/\lambda)\phi(2l/\lambda)-\phi((2l+1)/\lambda)\phi'(2l/\lambda)<0 \\
\Leftrightarrow \quad & \frac{\phi'((2l+1)/\lambda)}{\phi((2l+1)/\lambda)} < \frac{\phi'(2l/\lambda)}{\phi(2l/\lambda)},
\end{align*}
and $\phi'/\phi$ is decreasing by hypothesis.
\end{proof}
\begin{table}[H]
\begin{equation*}
	\arraycolsep=3.5pt\def\arraystretch{1.5}
\begin{array}{llll}
\hline
\phi & 1 & (1-x^p)^q & \exp(-\xi x) \\
\hline
p_0^{\omega^\lambda} & 1 & \left(1-\left(\frac{2l}{\lambda }\right)^p\right)^{-q} \left(1-\left(\frac{2 l+1}{\lambda }\right)^p\right)^q & e^{-\frac{\xi}{\lambda}} \\
\phi'/\phi & 0 & -\frac{p q x^{p-1}}{1-x^p} & -\xi \\
(\phi'/\phi)' & 0 & -\frac{p q x^{p-2} \left(x^p+p-1\right)}{\left(1-x^p\right)^2}  & 0 \\
\hline
\end{array}
\end{equation*}
\caption{Functions  $p_0^{\omega^\lambda}$, $\phi'/ \phi$ and its derivative, being $\phi$ the functions presented in Table \ref{tabla1nucleos} and $2n-1<\lambda<2n$.}
	\label{tablap0p1}
\end{table}

Therefore, we prove the following corollary.

\begin{corollary}[$\mathcal{C}^1$ limit functions] \label{cor:C1}
	Let $n\in\N$, $n\geq 2$ and $\lambda\in(2n-1,2n)$ be. The scheme $S_{1,\bw^\lambda}$ is $\cC^1$ for $\phi(x) = 1$ and for any weight function $\phi(x) = (1-x^p)^q$ with $p\geq 1$ and $q>0$.
\end{corollary}

\begin{proof}
	From Table \ref{tablap0p1}, the function $p_0^{\omega^\lambda}$ is decreasing for any $\phi(x) = (1-x^p)^q$ with $p\geq 1$ and $q>0$. Then, by Lemma \ref{lema1} the coefficients of the difference scheme are positive and by Proposition \ref{propgonsor} the subdivision scheme $S_{2\bq}$ is convergent. The case $\phi(x)=1$ is studied in \cite{dynheardhormannsharon}.
\end{proof}
				
In order to finish this section, we study two properties. Firstly, we analyse if the new family of schemes conserves monotonicity. In our case, the result presented by Yad-Shalom in \cite{shalom} can be used:

\begin{proposition}
	Let $S_\ba$ be a convergent subdivision scheme and $S_{\bq}$ its corresponding difference scheme with a positive mask. If the initial data, $\bff^0$, is non-decreasing then the limit function $S^\infty \bff$ is non-decreasing.
\end{proposition}

With this proposition, we can enunciate the following corollary.
\begin{corollary}[Monotonicity preservation] \label{cor:monotone}
	For $\lambda\in(1,+\infty)\backslash\N$ and any weight function introduced in Table \ref{tabla1nucleos}, the scheme $S_{1,\bw^\lambda}$ conserves the monotonicity.
\end{corollary}

Finally, when the initial data presents an isolated discontinuity and a linear subdivision scheme is applied several times some non-desirable effects may appear near the discontinuity, some kind of Gibbs phenomenon (see e.g. \cite{amatetal}).
In \cite{amatetal} it is proved that if the mask of the scheme is non-negative then the Gibbs phenomenon does not appear in the limit function.

\begin{corollary}[Avoiding Gibbs phenomenon] \label{cor:gibbs}
	For $\lambda\in(1,+\infty)\backslash\N$, the scheme $S_{1,\bw^\lambda}$ avoids the Gibbs phenomenon.
\end{corollary}

In Section \ref{expnum}, we present some examples checking these theoretical results. For $d=0,1$, the resulting mask is positive and we have used classic tools to study its properties. However, for $d\geq 2$, the mask are no longer positive. In the next section, we will develop a novel technique based on numerical integration for this goal and we will apply it to prove the convergence of the schemes based on weighted-least squares.

\section{A tool for the convergence analysis} \label{sec:tools}

The purpose of this section is to provide new theoretical results to analyse the convergence. In Section \ref{sec:d01}, the convergence was easily proven by the positivity of the mask. However, in Section \ref{sec:d23} we will prove the convergence of the scheme based on the regression with polynomials of degrees $d=2,3$, which are no longer positive, so that we cannot follow the same strategy. Nevertheless, as a consequence of Lemma \ref{lem:eval_poly}, the sub-masks can be seen as the evaluation of a second degree polynomial and this fact is advantageous and we will take profit of it in this section.

For any particular value of $n$, a fixed $\omega$ and considering some $\lambda_n$ such that $2n-1 <\lambda_n < 2n+1$, $\lambda_n\neq 2n$, it can be easily computed the difference scheme using the formula \eqref{eq:difference_scheme} and checked if its norm is less than 1, which would imply convergence. 
Let us call this method the \emph{direct inspection}. But it serves to prove convergence only for the chosen $n$, and we wish to prove it for all $n\in\N$. Our strategy will consist in proving converge asymptotically, that is, to prove convergence for $\forall n> n_0$, for some $n_0\in\N$, and then check the converge for each $n \leq n_0$ by direct inspection.

First, we would like to give a general idea about this asymptotic convergence. Thanks to the properties of the space of polynomials $\Pi_d$, the problem \eqref{problema1} can be formulated using equidistant knots in the interval $[-1,1]$, such as
\begin{equation}\label{problema2}
	\begin{split}
		\hat \bbeta^i=\argmin_{\bbeta \in \mathbb{R}^{d+1}} \sum_{l=1-n}^{n-1+i} \omega((2l-i)/{\lambda_n})L_p(f^k_{j+l},A\left(\frac{2l-i}{2n}\right)^T\bbeta),\quad i=0,1.
	\end{split}
\end{equation}
The last sum is, in fact, a composite integration rule. So that, if $n\to\infty$, then $2n/\lambda_n\to1$ and the problem \emph{seems} (this is not a rigorous argument, but it serves to understand the situation) to converge to
\begin{equation}\label{problema3}
	\begin{split}
		\argmin_{\bbeta \in \mathbb{R}^{d+1}} \int_{-1}^1  L_p(f(x),A(x)^T\bbeta) \ \omega(x) dx,
	\end{split}
\end{equation}
for both $i=0,1$. On the one hand, the given data is now a function $f(x)$ which is approximated by a polynomial $A(x)^T\bbeta\in\Pi_d$ in the $L_p$ norm with a weight function $\omega$.
On the other hand, by Lemma \ref{lem:eval_poly} the corresponding subdivision sub-masks, say $\ba^{n,i}$, fulfils $a^{n,i}_l = \omega((2l-i)/{\lambda_n})^{-1}A(x)^T\balpha^{n,i}$, for some coefficients $\balpha^{n,i}\in\R^{d+1}$. Then, the sub-masks also seem to converge to some continuos function, if some normalization is performed since the sub-masks supports increase with $n$ (see later Remark \ref{rmk:compute_r} and Section \ref{sec:d23} for more details). The results presented in this section exploit this kind of situations.

From now on, we consider a family of subdivision schemes $\{S_{\ba^n}\}_{n=1}^\infty$ as in \eqref{eq:our_schemes}.
The results in this section allow to prove convergence for $n > n_0$, for some $n_0\in\N$, and also provides the value of $n_0$, so that it can be checked convergence for $n\leq n_0$ by direct inspection. Combining both proofs, we obtain convergence for all $n\in\N$.
In particular, $\lim_{n\to\infty}\|S_{\bq^n}\|_\infty$ will be computed, which ensures the asymptotic convergence when that limit is less than 1. Here we denote by $\ba^{n,0},\ba^{n,1},\bq^{n,0},\bq^{n,1}$ the sub-masks of the masks $\ba^n,\bq^n$.

\begin{theorem} \label{teo:d0}
	Let $\{S_{\ba^n}\}_{n=1}^\infty$ be a sequence of subdivision schemes that reproduces $\Pi_0$, which odd rules are longer than (or as long as) the even rules, as in \eqref{eq:our_schemes}.
	Let $r:[-1,1]\to \R$ be a $\cC^1$ function and let $R(t) := \int_{-1}^{t} r(s) ds$ be.
	If
	\begin{align}
		a^{n,0}_{j} - a^{n,1}_{j} &= r(j/n)n^{-2} + \varepsilon^n_j, & j = 1-n,\ldots,L_n, \label{teo:d0:1}\\
		|\varepsilon^n_j| &\leq \mu n^{-\alpha}, & j = 1-n, \ldots, L_n,\label{teo:d0:2}\\
		\|R\|_1 &= \int_{-1}^{1} |R(t)| dt <1, & \label{teo:d0:3}
	\end{align}
	for some $\alpha > 2$, $\mu> 0$, then the first sub-masks of the difference schemes fulfil
	\[ \lim_{n\to\infty} \|\bq^{n,0}\|_1 = \lim_{n\to\infty} \sum_{l=1-n}^{L_n} |q^{n,0}_{l} | \leq \|R\|_1,\]
	thus there exists $n_{0}\in\N$ such that
	\begin{equation} \label{eq:sum_less_1}
		\|\bq^{n,0}\|_1 < 1, \quad \forall n > n_{0}.
	\end{equation}
	
	Moreover, if \eqref{teo:d0:2} holds true for $\alpha = 3$, then
	\begin{align} \label{eq:n0}
		n_0 = \begin{cases} \displaystyle
			\frac{\sqrt{(\|r\|_\infty+2 (\mu +\|r'\|_\infty))^2+4 (\|R\|_1-1) (\mu +\|r'\|_\infty)}+\|r\|_\infty+2 (\mu +\|r'\|_\infty)}{2 (1-\|R\|_1)}, & \text{if }L_n = n-1,\\[10pt] \displaystyle
			\frac{\sqrt{(\|r\|_\infty+2 (\mu +\|r'\|_\infty))^2+4 (1-\|R\|_1) \mu }+\|r\|_\infty+2 (\mu +\|r'\|_\infty)}{2 (1-\|R\|_1)}, & \text{if }L_n = n,
		\end{cases}
	\end{align}
	where
	\[ \|r\|_\infty = \max_{t\in[-1,1]} |r(t)|, \quad \|r'\|_\infty = \max_{t\in[-1,1]} |r'(t)|. \]
\end{theorem}
\begin{proof}
	First, we may write $q^{n,0}_{j}$ in terms of $r$:
	\[
	q^{n,0}_{j} = \sum_{l=1-n}^{j}\{ a^{n,0}_{l} - a^{n,1}_{l} \}
	= \sum_{l=1-n}^{j} \{  r(l/n)n^{-2} + \varepsilon^n_l \}.
	\]
	Using the composite (backward) rectangle rule, we obtain
	\[
	n^{-1} \sum_{l=1-n}^{j} r(l/n)  =
	\int_{-1}^{j/n} r(t) dt + \theta^n_j=
	R(j/n) + \theta^n_j,
	\]
	where $\theta^n_j$ is the integration error, which fulfils
	\(
	|\theta^n_j| \leq
	n^{-1}\|r'\|_\infty.
	\)
	Then,
	\[
	q^{n,0}_{j} = n^{-1}R(j/n) + n^{-1}\theta^n_j + \sum_{l=1-n}^{j} \varepsilon^n_l.
	\]
	With this computation, we will prove \eqref{eq:sum_less_1} first:
	\begin{align*}
		\|\bq^{n,0}\|_1 &= \sum_{j=1-n}^{L_n} |n^{-1}R(j/n) + n^{-1}\theta^n_j + \sum_{l=1-n}^{j} \varepsilon^n_l| \\
		& \leq n^{-1} \sum_{j=1-n}^{L_n} |R(j/n)| + n^{-1}\sum_{j=1-n}^{L_n}|\theta^n_j| + \sum_{j=1-n}^{L_n}\sum_{l=1-n}^{j} | \varepsilon^n_l|.
	\end{align*}
	Now, if $L_n=n-1$, we use that $R(-1) = 0$ and the composite (forward) rectangle rule, thus obtaining that
	\[
	n^{-1}\sum_{j=1-n}^{L_n} |R(j/n)| = n^{-1}\sum_{j=-n}^{n-1} |R(j/n)| = \int_{-1}^1 |R(t)| dt + \rho^n = \|R\|_1 + \rho^n,
	\]
	where $\rho^n$ is the integration error of $R(t)$,
	\[
	|\rho^n| \leq n^{-1} \max_{t\in[-1,1]} |R'(t)| = n^{-1} \|r\|_\infty.
	\]
	If $L_n = n$, we use the composite (backward) rectangle rule and we obtain a similar result:
	\[
	n^{-1}\sum_{j=1-n}^{L_n} |R(j/n)| = n^{-1}\sum_{j=1-n}^{n} |R(j/n)| = \|R\|_1 + \tilde \rho^n, \quad
	|\tilde \rho^n| \leq n^{-1} \|r\|_\infty.
	\]
	Using all the upper bounds we found, we obtain:
	\begin{align} \label{eq:bound-of-convergence}
		\begin{split}
			\|\bq^{n,0}\|_1
			& \leq \|R\|_1 + n^{-1} \|r\|_\infty + n^{-2}(L_n+n)\|r'\|_\infty
			+ \frac{1}{2} (L_n+n) (L_n+n+1) \mu n^{-\alpha}.
		\end{split}
	\end{align}
	From here we deduce that, if $\alpha > 2$, then
		the limit when $n\to\infty$ of the right part of \eqref{eq:bound-of-convergence} is $\|R\|_1$, which is less than 1. Hence, there exists $n_0\geq 1$ such that $\|\bq^{n,0}\|_1<1$, $\forall n> n_0$.
		In particular, for $\alpha = 3$, we can find for which value of $n_0$ the right part of \eqref{eq:bound-of-convergence} is equal to 1, by solving a second degree equation, arriving to \eqref{eq:n0}.
	\end{proof}
	
	\begin{remark} \label{rmk:compute_r}
		In practice, if the expressions of $a^{n,0}_j,a^{n,1}_j$ are well defined for any $j\in\R$ (this is the case of $S_{3,\mathbf{w^\lambda}}$, see \eqref{eq:def23}), then a practical way to compute $r(t)$ is
		\[
		r(t) := \lim_{n\to\infty} (a^{n,0}_{t n} - a^{n,1}_{t n})n^2.
		\]
		In Section \ref{sec:d23}, a complete example of the application of the results of this section will be performed.
	\end{remark}
	
	A similar condition will be derived from the last result to ensure that \(\|\bq^{n,1}\|_1 < 1\). First, we prove a result that will be useful for symmetric subdivision operators.
	
	\begin{theorem}\label{teo:d1}
		Let $\{S_{\ba^n}\}_{n=1}^\infty$ be as in \eqref{eq:our_schemes} and consider a flipped version of them, $\{S_{\bar \ba^n}\}_{n=1}^\infty$, defined as
		\begin{align*}
			\bar a^{n,0}_{j} &:= a^{n,0}_{L_n+1-n - j}, &j=1-n,\ldots,L_n, \\
			\bar a^{n,1}_{j} &:= a^{n,1}_{L_n+2-n - j},  &j=1-n,\ldots,L_n+1 .
		\end{align*}
		Then
		\begin{align*}
			q^{n,0}_{j} & = \bar q^{n,1}_{L_n+1-n - j}, \qquad  \|\bq^{n,0}\|_1  = \|\bar \bq^{n,1}\|_1.
		\end{align*}
		Moreover, $\{S_{\ba^n}\}_{n=1}^\infty$ fulfil the conditions of Theorem \ref{teo:d0} if, and only if, $\{S_{\bar \ba^n}\}_{n=1}^\infty$ fulfil
		\begin{align}
			\bar a^{n,0}_{j} - \bar a^{n,1}_{j+1} &= \bar r(j/n)n^{-2} + \bar \varepsilon^n_j, & j = 1-n,\ldots,L_n, \label{teo:d1:1}\\
			|\bar \varepsilon^n_j| &\leq \mu n^{-\alpha}, & j = 1-n, \ldots, L_n, \label{teo:d1:2}\\
			\|\bar R\|_1 &<1, &  \label{teo:d1:3}
		\end{align}
		where $\bar r(t) := r(-t)$, $\bar \varepsilon^n_j = \varepsilon^n_{-j}$ and $\bar R(t) := \int_{t}^{1} \bar r(s) ds = R(-t)$.
	\end{theorem}
	\begin{proof}
		Observe that
		\begin{align*}
			\bar a^{n,0}_{j} - \bar a^{n,1}_{j+1} &= a^{n,0}_{L_n+1-n - j} - a^{n,1}_{L_n+2-n - (j+1)}= a^{n,0}_{L_n+1-n - j} - a^{n,1}_{L_n+1-n - j}, \qquad j = 1-n,\ldots,L_n,
		\end{align*}
		so that, defining $\bar r(t) := r(-t)$, $\bar \varepsilon^n_{j} := \varepsilon^n_{-j}$, the equivalence between \eqref{teo:d0:1}-\eqref{teo:d0:2} and \eqref{teo:d1:1}-\eqref{teo:d1:2} is clear.
		Then
		\begin{align*}
			\bar R(t) &= \int_{t}^{1} \bar r(s) ds = \int_{t}^{1} r(-s) ds \overset{[u = -s]}{=} \int_{-t}^{-1} - r(u) du =  \int_{-1}^{-t} r(u) du = R(-t),
		\end{align*}
		and
		\[
		\int_{-1}^{1} |\bar R(t)| dt = \int_{-1}^{1} \left| R(-t) \right| dt = \int_{-1}^{1} \left| R(t) \right| dt,
		\]
		thus, the equivalence between \eqref{teo:d0:3} and \eqref{teo:d1:3} also holds true.
		
		On the other hand, $S_\ba^{n}$ reproduces $\Pi_0$ if, and only if, $S_{\bar \ba^{n}}$ does. Hence, the finite difference scheme exists and can be computed with the formula \eqref{eq:difference_scheme}.
		\begin{align*}
			\bar q^{n,0}_{j} &= \sum_{l=1-n}^{j} a^{n,0}_{L_n+1-n - l} - a^{n,1}_{L_n+2-n - l}
			\overset{[k = L_n+1-n - l]}{=} \sum_{k=L_n+1-n - j}^{L_n} a^{n,0}_{k} - a^{n,1}_{k+1} = q^{n,1}_{L_n+1-n - j}.
		\end{align*}
		Hence,
		\[
			\sum_{j=1-n}^{L_n} |q^{n,1}_{j} | = \sum_{j=1-n}^{L_n} |q^{n,1}_{L_n+1-n - j} | = \sum_{j=1-n}^{L_n} |\bar q^{n,0}_{j}|.
		\]
		Since \( R(t) = \bar R(-t) \) and \( r(t) = \bar r(-t) \), we deduce that the formula to compute $n_0$, \eqref{eq:n0}, can be used here as well.
	\end{proof}
		
	The next result is a direct consequence of the previous one.
	\begin{corollary}\label{cor:d1}
		Let $\{S_{\ba^n}\}_{n=1}^\infty$ be, as in \eqref{eq:our_schemes}, that reproduce $\Pi_0$. Let $r:[-1,1]\to \R$ be a $\cC^1$ function and let $R(t) := \int_{t}^{1} r(s) ds$ be.
		If
		\begin{align*}
			a^{n,0}_{j} - a^{n,1}_{j+1} &= r(j/n)n^{-2} + \varepsilon^n_j, & j = 1-n,\ldots,L_n,\\
			|\varepsilon^n_j| &\leq \mu n^{-\alpha}, & j = 1-n, \ldots, L_n,\\
			\|R\|_1 & <1, &
		\end{align*}
		for some $\alpha > 2$, $\mu>0$, then there exists $n_{0}\in\N$ such that
		\[
		\|\bq^{n,1}\|_1 < 1, \quad \forall n \geq n_{0}.
		\]
		In case that $\alpha = 3$, $n_0$ can be obtained as in \eqref{eq:n0}.
	\end{corollary}
	\begin{proof}
		By the Theorem \ref{teo:d1}, the flipped version of this scheme fulfils Theorem \ref{teo:d0} and the claimed inequality is true.
	\end{proof}
	
	For odd-symmetric subdivision operators, due to Theorem \ref{teo:d1}, the satisfaction of the hypothesis of Theorem \ref{teo:d0} or Corollary \ref{cor:d1} is sufficient to ensure convergence.
	
	\begin{theorem} \label{teo:sym}
		Let $\{S_{\ba^n}\}_{n=1}^\infty$ be a set of odd-symmetric subdivision schemes fulfilling the hypothesis of Theorem \ref{teo:d0}. Then, the subdivision scheme $S_{\ba^n}$ is convergent if $n> n_0$ with $n_0$ as in \eqref{eq:n0}.
	\end{theorem}

	\section{WLPR-Subdivision schemes for $d=2,3$} \label{sec:d23}
	
	We consider $\{\lambda_n\}_{n\geq 2}$ such that $2n-1<\lambda_n<2n$, then $L_n = 1-n$. The following computations could be done for  $2n<\lambda_n<2n+1$ as well.
	First, we compute the coefficients of $S_{3,\bw^\lambda}$ (denote it by $S^n$ from now on).
	According to Lemma \ref{lem:eval_poly}, the sub-masks are \(\ba^{n,i} = \mathbf{W}^i \mathbf{X}^i \balpha^i\), $i=0,1$, where \(\balpha^i=((\mathbf{X}^i)^T\mathbf{W}^i\mathbf{X}^i)^{-1} \mathbf{e}_1\). Then, to compute $\alpha$ we may solve the system
	\[
	(\mathbf{X}^i)^T\mathbf{W}^i\mathbf{X}^i \balpha^i = \mathbf{e}_1.
	\]
	
	We start with $i=1$. Using \eqref{eq:XWX1} and the symmetry of $\bw^1$ and $\bx^1$,
	\[
	(\mathbf{X}^1)^T\mathbf{W}^1\mathbf{X}^1 =
	\left(\begin{array}{llllll}
		\|\bw^1\|_1 & 0 &  2\sum_{i=1}^{n} w^\lambda_{2i-1} (2i-1)^2 \\
		0 & 2\sum_{i=1}^{n} w^\lambda_{2i-1} (2i-1)^2 & 0 \\
		2\sum_{i=1}^{n} w^\lambda_{2i-1} (2i-1)^2 & 0 & 2\sum_{i=1}^{n} w^\lambda_{2i-1} (2i-1)^{4}
	\end{array}\right),
	\]
	\[
	\Delta^1 := \left|(\mathbf{X}^1)^T\mathbf{W}^1\mathbf{X}^1\right| = 4 \left(\|\bw^1\|_1\sum_{i=1}^{n} w^\lambda_{2i-1} (2i-1)^{4} - 2\left(\sum_{i=1}^{n} w^\lambda_{2i-1} (2i-1)^2\right)^2\right)\sum_{i=1}^{n} w^\lambda_{2i-1} (2i-1)^2.
	\]
	Hence, using the Kramer's formula, the three coefficients of $\balpha^1$ are:
	\begin{align*}
		\alpha^1_0 &= (\Delta^1)^{-1}
		\left|\begin{array}{llllll}
			1 & 0 &  2\sum_{l=1}^{n} w^\lambda_{2l-1} (2l-1)^2 \\
			0 & 2\sum_{l=1}^{n} w^\lambda_{2l-1} (2l-1)^2 & 0 \\
			0 & 0 & 2\sum_{l=1}^{n} w^\lambda_{2l-1} (2l-1)^{4}
		\end{array}\right| \\
		&= 4 (\Delta^1)^{-1} \left(\sum_{l=1}^{n} w^\lambda_{2l-1} (2l-1)^2\right)\left(\sum_{l=1}^{n} w^\lambda_{2l-1} (2l-1)^{4}\right) \\
		&= \frac{\sum_{l=1}^{n} w^\lambda_{2l-1} (2l-1)^{4}}{\|\bw^1\|_1\sum_{l=1}^{n} w^\lambda_{2l-1} (2l-1)^{4} - 2\left(\sum_{l=1}^{n} w^\lambda_{2l-1} (2l-1)^2\right)^2}\\
		&= \frac{\sum_{l=1}^{n} w^\lambda_{2l-1} (l-\frac12)^{4}}{\|\bw^1\|_1\sum_{l=1}^{n} w^\lambda_{2l-1} (l-\frac12)^{4} - 2\left(\sum_{l=1}^{n} w^\lambda_{2l-1} (l-\frac12)^2\right)^2},
	\end{align*}
	\begin{align*}
		\alpha^1_1 &= 0, \\
		\alpha^1_2 &= (\Delta^1)^{-1}
		\left|\begin{array}{llllll}
			\|\bw^1\|_1 & 0 &  1 \\
			0 & 2\sum_{l=1}^{n} w^\lambda_{2l-1} (2l-1)^2 & 0 \\
			2\sum_{l=1}^{n} w^\lambda_{2l-1} (2l-1)^2 & 0 & 0
		\end{array}\right| \\
		&= -4(\Delta^1)^{-1} \left(\sum_{l=1}^{n} w^\lambda_{2l-1} (2l-1)^2\right)^2\\
		&= -\frac{\sum_{l=1}^{n} w^\lambda_{2l-1} (2l-1)^2}{\|\bw^1\|_1\sum_{l=1}^{n} w^\lambda_{2l-1} (2l-1)^{4} - 2\left(\sum_{l=1}^{n} w^\lambda_{2l-1} (2l-1)^2\right)^2}\\
		&= -\frac14\frac{\sum_{l=1}^{n} w^\lambda_{2l-1} (l-\frac12)^2}{\|\bw^1\|_1\sum_{l=1}^{n} w^\lambda_{2l-1} (l-\frac12)^{4} - 2\left(\sum_{l=1}^{n} w^\lambda_{2l-1} (l-\frac12)^2\right)^2}.
	\end{align*}
	Then, by \eqref{eq:eval_poly2}, the sub-mask coefficients are
	\begin{align}
		a^{n,1}_j &= w^\lambda_{2j-1} (\alpha^1_0 + \alpha^1_2 (2j-1)^2)= w^\lambda_{2j-1} (\alpha^1_0 + 4\alpha^1_2 (j-\frac12)^2)\\
		&= w^\lambda_{2j-1} \frac{\sum_{l=1}^{n} w^\lambda_{2l-1} (l-\frac12)^{4} - (j-\frac12)^2\sum_{l=1}^{n} w^\lambda_{2l-1} (l-\frac12)^2}{\|\bw^1\|_1\sum_{l=1}^{n} w^\lambda_{2l-1} (l-\frac12)^{4} - 2\left(\sum_{l=1}^{n} w^\lambda_{2l-1} (l-\frac12)^2\right)^2},
		& j=1-n,\ldots,n. \label{eq:d23_1_coeffs}
	\end{align}
	Similarly,
	\begin{align}\label{eq:d23_0_coeffs}
		a^{n,0}_j &= w^\lambda_{2j} (\alpha^0_0 + 4\alpha^0_2 j^2)
		=w^\lambda_{2j} \frac{\sum_{l=1}^{n-1} w^\lambda_{2l} l^{4} - j^2\sum_{l=1}^{n-1} w^\lambda_{2l} l^2}{\|\bw^0\|_1\sum_{l=1}^{n-1} w^\lambda_{2l} l^{4} - 2\left(\sum_{l=1}^{n-1} w^\lambda_{2l} l^2\right)^2},
		&j=1-n,\ldots,n-1,
	\end{align}
	where
	\begin{align*}
		\alpha^0_0 &= \frac{\sum_{l=1}^{n-1} w^\lambda_{2l} l^{4}}{\|\bw^0\|_1\sum_{l=1}^{n-1} w^\lambda_{2l} l^{4} - 2\left(\sum_{l=1}^{n-1} w^\lambda_{2l} l^2\right)^2}, \quad
		\alpha^0_2 = -\frac14\frac{\sum_{l=1}^{n-1} w^\lambda_{2l} l^2}{\|\bw^1\|_1\sum_{l=1}^{n-1} w^\lambda_{2l} l^{4} - 2\left(\sum_{l=1}^{n-1} w^\lambda_{2l} l^2\right)^2}.
	\end{align*}
	
	We first prove convergence $\forall n\geq2$ in the simplest case, $\phi(x)=1$, in order to be used to the new convergence analysis tools, and later we discuss the general case.
	
	\subsection{Convergence of the subdivision schemes based on weighted least squares with $d=2,3$ and $\phi(x)=1$}
	
	In this case, $w_l=1$, $1-2n\leq l \leq 2n-1$, so that the mask coefficients can be simplified to
	\begin{align} \label{eq:def23}
		\begin{split}
			a^{n,0}_{j} &= -\frac{3 \left(5 j^2-3 n^2+3 n+1\right)}{8 n^3-12 n^2-2 n+3}, \qquad j = -n+1, \ldots, n-1, \\
			a^{n,1}_{j} &= \frac{15 (j-1) j-9 n^2+9}{8 n-8 n^3}, \qquad j = -n+1, \ldots, n.
		\end{split}
	\end{align}
	It can be easily checked that these operators are odd-symmetric, which for sure we knew by Lemma \ref{lem:symmetry}. Hence, to prove convergence we can apply Theorem \ref{teo:sym}.
	
	Observe that the algebraic expressions of $a^{n,0}_{j}$ and $a^{n,1}_{j}$ are well defined even for $j\in\R$. Then, for any $t\in[-1,1]$, we define
	\[
	r(t) := \lim_{n\to\infty} (a^{n,0}_{t n} - a^{n,1}_{t n})n^2 = -\frac{45 t^2}{16}-\frac{15 t}{8}+\frac{9}{16},
	\]
	that we obtained with the aid of a symbolic computation program. We also computed that
	\begin{align} \label{eq:eps} \begin{split}
			n^3 \varepsilon^n_j =& n^3(a^{n,0}_{j} - a^{n,1}_{j} - r(j/n)n^{-2}) = \\
			& 3 \rho(n)^{-1}(-120 j^2 n^4-120 j^2 n^3+225 j^2 n^2+30 j^2 n-45 j^2-80 j n^4+120 j n^3\\
			&+20 j n^2-30 j n+32 n^6-12 n^5-41 n^4+12 n^3+9 n^2),
		\end{split}
	\end{align}
	where
	\[ \rho(n) := 16 (n-1) n (n+1) (2 n-3) (2 n-1) (2 n+1).\]
	Now we should find $\mu$ such that $|\varepsilon^n_j n^{3}|\leq \mu $ for $1-n\leq j \leq n-1$. On the one hand,
	\[ \rho(n) > 16 (n-2)^3 (2n-4) (2n-4) (2n-4) = 128 (n-2)^6 \geq 0, \qquad \forall n\geq 2. \]
	On the other hand, the numerator of \eqref{eq:eps} can be easily bounded using that $|j|\leq n$ and increasing to 6 the degree of every monomial:
	\begin{align*}
		|\rho(n) n^3 \varepsilon^n_j/3| &\leq 120 n^6 + 120 n^6 +225 n^6+30 n^6 + 45 n^6 + 80 n^6+ 120 n^6\\
		&+20 n^6 +30 n^6+32 n^6 + 12 n^6 + 41 n^6+12 n^6+9 n^6\\
		&= 896 n^6.
	\end{align*}
	As conclusion,
	\[ |n^3 \varepsilon^n_j| \leq 3 \frac{896 n^6}{128 (n-2)^6} = \frac{21 n^6}{(n-2)^6}, \quad \forall n>2.\]
	Then, for any $n_1 \geq 3$,
	\begin{equation} \label{eq:n1_bound}
		|\varepsilon^n_j| \leq n^{-3} \mu_1, \quad \mu_1 = \frac{21 n_1^6}{(n_1-2)^6}, \qquad \forall n \geq n_1.
	\end{equation}
	
	To compute $n_0$, it is also necessary to compute:
	\begin{align*}
		\|R\|_1 &= \int_{-1}^1 \left| \int_{-1}^t r(s) ds \right| dt = \frac{1}{10} \left(3 \sqrt{15}-5\right) \simeq 0.661895, \\
		\|r\|_\infty &= \max_{t\in[-1,1]} |r(t)| = 33/8, \quad \|r'\|_\infty = \max_{t\in[-1,1]} |r'(t)| = 15/2.
	\end{align*}
	Now, using formula \eqref{eq:n0} (case $L_n = n-1$) for $\mu = \mu_1$,
	\[
	n_0 = \frac{1}{6} \left(\sqrt{15}+5\right) \left(\frac{42 n_1^6}{(n_1-2)^6}+\sqrt{\left(\frac{42 n_1^6}{(n_1-2)^6}+\frac{153}{8}\right)^2+\frac{6}{5} \left(\sqrt{15}-5\right) \left(\frac{21 n_1^6}{(n_1-2)^6}+\frac{15}{2}\right)}+\frac{153}{8}\right).
	\]
	It is desirable to prove convergence for as much values of $n$ as possible, so $n_1$ should be chosen such that $n_0$ is as small as possible, but greater or equal than $n_1$, due to \eqref{eq:n1_bound}. We computationally found that the compromise is achieved for $n_1 = 188$, leading to $n_0 \simeq 188.506$. Hence, according to Theorem \ref{teo:sym}, the subdivision schemes are convergent for $n \geq 189$. For smaller values of $n$, we have computationally checked that
	\[
	\|\bq^{n,0}\|_1  = \|\bq^{n,1}\|_1 \leq 29/42 \simeq 0.690476, \qquad \forall 2 \leq n \leq 189.
	\]
	This symbolic computation is quick and without rounding errors, so this can be considered a rigorous proof of the convergence.
	
	We can perform some additional computations in order to provide an upper bound of $\|\bq^{n,0}\|_1  = \|\bq^{n,1}\|_1$ valid for any $n\geq 2$.
	According to \eqref{eq:bound-of-convergence},
	\[ \|\bq^{n,0}\|_1  \leq \frac1{10}(3 \sqrt{15}-5) + n^{-1} 33/8 + (2n-1)n^{-2} 15/2+ (2n^{-1}-n^{-2}) 21 \left (\frac{94}{93}\right )^6, \quad \forall n \geq 189. \]
	We checked that the right side is less than $29/42$ for any $n \geq 2236$, and we explicitly computed that for $n \leq 2236$, $\|\bq^{n,0}\|_1  \leq 29/42$. As conclusion,
	\[ \|\bq^{n,0}\|_1  = \|\bq^{n,1}\|_1 \leq 29/42, \qquad \forall n\geq 2, \]
	and the equality is reached only for $n=4$.
	
	We tried to prove $\cC^1$ regularity with this technique by applying the results to the divided difference schemes, $S_{2 \bq^n}$, but they do not satisfy \eqref{teo:d0:1}.
	
	\subsection{Convergence of the subdivision schemes based on weighted least squares with $d=2,3$ and a general function $\phi(x)$} \label{sec:conv_d3}
	In this situation, we will study the convergence only for large $n$ values, so that we will not calculate $n_0$, because we have not been able to perform the \emph{direct inspection} without specifying $\phi$.
	In order to compute
	\(r(t) := \lim_{n\to\infty} (a^{n,0}_{t n} - a^{n,1}_{t n})n^2,\)
	we will define a $\cC^1$ function $U^{n}_j$ such that \(a^{n,i}_{j} = U^{n}_j(1-i/2)\), $i=0,1$,
	which will allow to write
	\[a^{n,0}_{tn} - a^{n,1}_{tn} = U^{n}_{tn}(1) - U^{n}_{tn}(1/2) = \frac12(U^{n}_{tn})'(\xi_{t,n}), \quad \xi_{t,n}\in(1/2,1).\]
	For that purpose, we define \(\sigma_{\lambda_n}(\phi,x,k) := \sum_{l=1}^{n} \phi(\frac{l-x}{{\lambda_n}/2}) (l-x)^{k}\), $k\in\N$, $x\in[1/2,1]$.
	Recall that $w^{\lambda_n}_0 = 1$, $w_l^{\lambda_n} =\omega\left(\frac {l}{\lambda_n}\right)$ and $\omega(x) = \phi(|x|)$. Observe that the sub-masks  \eqref{eq:d23_1_coeffs} and \eqref{eq:d23_0_coeffs} can be expressed as
	\begin{align*}
		a^{n,1}_j &= \phi\left (\frac{j-1/2}{{\lambda_n}/2}\right ) \frac{\sigma_{\lambda_n}(\phi,1/2,4) - (j-\frac12)^2\sigma_{\lambda_n}(\phi,1/2,2)}{\|\bw^1\|_1\sigma_{\lambda_n}(\phi,1/2,4) - 2\sigma_{\lambda_n}(\phi,1/2,2)^2} = U^{n}_j(1/2),\\
		a^{n,0}_j &=w^{\lambda_n}_{2j} \frac{\sum_{l=1}^{n} w^{\lambda_n}_{2(l-1)} (l-1)^{4} - j^2\sum_{l=1}^{n} w^{\lambda_n}_{2(l-1)} (l-1)^2}{\|\bw^0\|_1\sum_{l=1}^{n} w^{\lambda_n}_{2(l-1)} (l-1)^{4} - 2\left(\sum_{l=1}^{n} w^{\lambda_n}_{2(l-1)} (l-1)^2\right)^2}\\
		&= \phi\left (\frac{j-0}{{\lambda_n}/2}\right ) \frac{\sigma_{\lambda_n}(\phi,1,4) - (j-0)^2\sigma_{\lambda_n}(\phi,1,2)}{\|\bw^0\|_1\sigma_{\lambda_n}(\phi,1,4)) - 2\sigma_{\lambda_n}(\phi,1,2)^2} = U^{n}_j(1).
	\end{align*}
	Thus, we may define the link function as
	\[
	U^{n}_j(x) := \phi\left (\frac{j+x-1}{{\lambda_n}/2}\right ) \frac{\sigma_{\lambda_n}(\phi,x,4) - (j+x-1)^2\sigma_{\lambda_n}(\phi,x,2)}{(\|\bw^1\|_1 + (2x-1)(\|\bw^0\|_1-\|\bw^1\|_1))\sigma_{\lambda_n}(\phi,x,4) - 2\sigma_{\lambda_n}(\phi,x,2)^2}.
	\]
	Observe that $U^{n}_j\in\cC^1([1/2,1])$ provided that $\phi\in\cC^1((0,1))$ ($\phi'$ may not exist at 0 or 1).
	To follow more easily the next computations, we write \(U^{n}_j(x) = \phi(\frac{j+x-1}{{\lambda_n}/2})U_{\text{num}}(x)/U_{\text{den}}(x)\), where \(U_{\text{num}}(x),U_{\text{den}}(x)\) are the numerator and denominator that appear in the last formula.
	
	Taking into account that \[\frac{\partial}{\partial x} \sigma_{\lambda_n}(\phi,x,k) = -\frac2{\lambda_n} \sigma_{\lambda_n}(\phi',x,k) - k\sigma_{\lambda_n}(\phi,x,k-1), \qquad k>1,\]
	we proceed to compute the derivative.
	\begin{align*}
		(U^{n}_j)^\prime(x) &= \frac2{\lambda_n}\phi'\left (\frac{j+x-1}{{\lambda_n}/2}\right ) \frac{U_{\text{num}}(x)}{U_{\text{den}}(x)}  + \phi\left (\frac{j+x-1}{{\lambda_n}/2}\right )\frac{U_{\text{num}}'(x)}{U_{\text{den}}(x)}  - \phi\left (\frac{j+x-1}{{\lambda_n}/2}\right )\frac{U_{\text{num}}(x)U_{\text{den}}'(x)}{U_{\text{den}}^2(x)} ,
	\end{align*}
	where
	\begin{align*}
		U_{\text{num}}'(x) =& -\frac2{\lambda_n} \sigma_{\lambda_n}(\phi',x,4) - 4\sigma_{\lambda_n}(\phi,x,3) - 2(j+x-1)\sigma_{\lambda_n}(\phi,x,2) \\
		&- (j+x-1)^2\left (-\frac2{\lambda_n} \sigma_{\lambda_n}(\phi',x,2) - 2\sigma_{\lambda_n}(\phi,x,1)\right ), \\
		U_{\text{den}}'(x) =&
		2(\|\bw^0\|_1-\|\bw^1\|_1)\sigma_{\lambda_n}(\phi,x,4)+(\|\bw^1\|_1 + (2x-1)(\|\bw^0\|_1-\|\bw^1\|_1))\left (-\frac2{\lambda_n} \sigma_{\lambda_n}(\phi',x,4) - 4\sigma_{\lambda_n}(\phi,x,3)\right ) \\
		& - 4\sigma_{\lambda_n}(\phi,x,2)\left (-\frac2{\lambda_n} \sigma_{\lambda_n}(\phi',x,2) - 2\sigma_{\lambda_n}(\phi,x,1)\right ).
	\end{align*}
	
	Finally, we proceed to compute
	\(r(t) = \lim_{n\to\infty} 	\frac{n^2}2 (U^{n}_{tn})'(\xi_{t,n}).\)
	To this purpose, we define
\begin{equation}\label{funcionalI}						
I_k(\phi) := \int_0^1 \phi(x) x^k dx, \quad k\in\N,
\end{equation}
	we observe $\lim_{n\to\infty}{2n}/{\lambda_n} = 1$ and we use the following composite integration rule
	\begin{align*}
		n^{-k-1} \sigma_{\lambda_n}(\phi,x,k)
		&=n^{-1}\sum_{l=1}^n \phi\left (\frac{l-x}{n}\frac{2n}{{\lambda_n}}\right )\left (\frac{l-x}n\right )^k = I_k(\phi) + \cO(n^{-1}), \qquad \forall k\in\N\cup\{0\}, \quad \forall x\in[\frac12,1].
	\end{align*}
	Defining \(\sigma_{\lambda_n}(\phi,x,0) := \sum_{l=1}^{n} \phi(\frac{l-x}{{\lambda_n}/2})\), so that \(\frac{\partial \sigma_{\lambda_n}}{\partial x}(\phi,x,0) = -\frac2{\lambda_n}\sigma_{\lambda_n}(\phi',x,0)\), we note that
	\[ n^{-1}\|\bw^1\|_1 = 2n^{-1}\sum_{l=1}^{n} \phi(\frac{l-1/2}{{\lambda_n}/2}) = 2n^{-1}\sigma_{\lambda_n}(\phi,1/2,0) = 2I_0(\phi)+ \cO(n^{-1}), \quad i=0,1, \] and
	\[
	\sigma_{\lambda_n}(\phi,1,0)-\sigma_{\lambda_n}(\phi,1/2,0)
	=  \frac12\frac{\partial \sigma_{\lambda_n}}{\partial x}(\phi,\xi_n,0) = - \frac12 \frac{2}{\lambda_n} \sigma_{\lambda_n}(\phi',\xi_n,0)
	= -\frac12 \int_0^1 \phi'(x)dx + \cO(n^{-1}) = \frac12(\phi(0) - \phi(1))+ \cO(n^{-1}),
	\]
	so that
	\[	\|\bw^0\|_1-\|\bw^1\|_0 = 2\sigma_{\lambda_n}(\phi,1/2,1)-\phi(0)-2\sigma_{\lambda_n}(\phi,1/2,0) = -\phi(1) + \cO(n^{-1}).\]
	Taking these comments into account and taking $j = t n$, we find out that
	\begin{align*}
		\lim_{n\to\infty}& \phi\left (\frac{t n+\xi_{t,n}-1}{{\lambda_n}/2}\right ) = \phi(t),  \quad\lim_{n\to\infty} \phi'\left (\frac{t n+\xi_{t,n}-1}{{\lambda_n}/2}\right ) = \phi'(t),\\
		\|\bw^1\|_1 &+ (2x-1)(\|\bw^0\|_1-\|\bw^1\|_1) = 2nI_0(\phi)+ \cO(n^{0}) + (2x-1)(-\phi(1) + \cO(n^{-1})) = 2nI_0(\phi)+ \cO(n^{0})\\
		U_{\text{num}}(x) &= n^5 I_4(\phi)  - t^2 n^5 I_2(\phi) + \cO(n^{4}),\\
		U_{\text{den}}(x)
&= 2n^6 I_0(\phi)I_4(\phi) - 2n^6 I_2(\phi)^2 + \cO(n^4),\\
		U_{\text{num}}'(x) &= -n^4 I_4(\phi') - 4n^4 I_3(\phi) - 2tn^4 I_2(\phi) - t^2 n^2(-n^2 I_2(\phi') - 2n^2 I_1(\phi)) + \cO(n^3),\\
		U_{\text{den}}'(x) &= 2 n^5 (-\phi(1)) I_4(\phi)+2nI_0(\phi)(-n^4 I_4(\phi') - 4n^4 I_3(\phi)) - 4n^3 I_2(\phi)(-n^2 I_2(\phi') - 2n^2 I_1(\phi))  + \cO(n^{4}).
	\end{align*}
	Hence,
	\begin{align*}
		r(t) &= \lim_{n\to\infty} \frac12n^2(U^{n}_{tn})^\prime(\xi_{t,n}) = \frac12\phi'(t)\lim_{n\to\infty}n \frac{n^5}{n^6} \frac{I_4(\phi)  - t^2 I_2(\phi)}{2I_0(\phi)I_4(\phi) - 2I_2(\phi)^2} \\
		&+ \frac12\phi(t)\lim_{n\to\infty}n^2\frac{n^4}{n^6} \frac{-I_4(\phi') - 4I_3(\phi) - 2tI_2(\phi) - t^2 (- I_2(\phi') - 2 I_1(\phi))}{2I_0(\phi) I_4(\phi) - 2 I_2(\phi)^2} \\
		&- \frac12\phi(t)\lim_{n\to\infty}n^2 n^5 (I_4(\phi) - t^2 I_2(\phi)) \cdot \frac{n^5}{n^{12}} \frac{
			2(-\phi(1)) I_4(\phi) + 2I_0(\phi)(-I_4(\phi') - 4I_3(\phi)) - 4 I_2(\phi)(- I_2(\phi') - 2 I_1(\phi))
		}{(2I_0(\phi) I_4(\phi) - 2I_2(\phi)^2)^2}\\
		&= \frac14\phi'(t)\frac{I_4(\phi)  - t^2 I_2(\phi)}{I_0(\phi)I_4(\phi) - I_2(\phi)^2}
		- \frac14\phi(t)\frac{I_4(\phi') + 4I_3(\phi) + 2tI_2(\phi) - t^2 ( I_2(\phi') + 2 I_1(\phi))}{I_0(\phi) I_4(\phi) - I_2(\phi)^2} \\
		&- \frac14\phi(t) (I_4(\phi) - t^2 I_2(\phi))  \frac{
			(-\phi(1))I_4(\phi) - I_0(\phi)(I_4(\phi') + 4I_3(\phi)) + 2 I_2(\phi)( I_2(\phi') + 2 I_1(\phi))
		}{(I_0(\phi) I_4(\phi) - I_2(\phi)^2)^2}.
	\end{align*}
	Clearly, the former expression is valid provided that $I_0(\phi) I_4(\phi) - I_2(\phi)^2 \neq 0$. Fortunately, we can use the Schwartz's inequality for the inner product $\langle f,g \rangle := \int_0^1 f(x) g(x) \phi(x)dx$ to deduced that
	\[I_0(\phi) I_4(\phi) - I_2(\phi)^2 = \langle1,1\rangle \langle x^2,x^2\rangle - \langle1,x^2\rangle^2 > 0.\]
	
	We gather in Table \ref{tabla2convergencia} the computation of $r(t)$ and $\|R\|_1$ for several choices of $\phi$. Since $\|R\|_1<1$ for all of them, we conclude that, for $n$ large enough, any of the corresponding subdivision schemes converge. We realized that the value of $\|R\|_1$ could be greater than one for some extreme choices of $\phi$. An example is $\phi(x)=1+1000x^2$, but thus kind of functions were discarded in Section \ref{section3} due to its practical meaning.
	
	\begin{table}[!h]
		\begin{equation*}
			\arraycolsep=3.5pt\def\arraystretch{1.5}
			\begin{array}{lll}
				\hline
				\phi(x) & r(t) & \|R\|_1 \\\hline
				1				& -\frac{45 t^2}{16}-\frac{15 t}{8}+\frac{9}{16} & \frac{1}{10} \left(3 \sqrt{15}-5\right) \simeq 0.661895\\
				1-x			& \frac{45 t | t| }{7}-\frac{6 \text{sgn}(t)}{7}-\frac{30 t}{7} & \frac{1}{70} \left(32 \sqrt{10}-59\right) \simeq 0.602756 \\
				1-x^2			& \frac{105 t^3}{16}-\frac{75 t}{16}	& \frac{12 \sqrt{\frac{3}{7}}}{7}-\frac{1}{2} \simeq 0.622263\\
				(1-x^2)^2		& -\frac{945 t^5}{64}+\frac{735 t^3}{32}-\frac{525 t}{64} & \frac{1}{36} \left(23 \sqrt{3}-18\right) \simeq 0.606588\\
				(1-x^3)^3		& \frac{889350 t | t| ^9}{32099}-\frac{229635 t | t| ^7}{32099}-\frac{1940400 t | t| ^6}{32099}+\frac{459270 t | t| ^4}{32099}\\
				&+\frac{1212750 t | t| ^3}{32099}-\frac{229635 t | t| }{32099}-\frac{161700 t}{32099} & \frac{30889956929079600 \sqrt{2310}-719074739206009649}{1241215209364900000} \simeq 0.616793\\
				(1-x^2)^3		& \frac{3465 t^7}{128}-\frac{8505 t^5}{128}+\frac{6615 t^3}{128}-\frac{1575 t}{128} & \frac{2799 \sqrt{\frac{3}{11}}}{1331}-\frac{1}{2} \simeq 0.598219\\
				(1-x^p)^q& \text{(large expression involving $\Gamma$ function)}\\
				e^{-x}& \frac{e^{1-| t| } \left((e (20 e-69)+40) \left((2 e-5) t^2-24 e+65\right) \text{sgn}(t)-2 (e-4) (11 e-30) t^2\right)}{4 (e (20 e-69)+40)^2} &  \\
				& +\frac{e^{1-| t| } \left(-2 (2 e-5) (e (20 e-69)+40) t+4 (30-11 e)^2\right)}{4 (e (20 e-69)+40)^2} & \sim 0.621749 \\
				e^{-10x}& \text{(explicit but large expression)} & \sim 0.529404 \\
				e^{-\xi x}& \text{(explicit but large expression)}\\
				1+1000x^2& -\frac{23822324150625 t^4}{734488968098}-\frac{15776250 t^3}{606007}+\frac{81269240847795 t^2}{5875911744784}\\
				&+\frac{44999895 t}{4848056}+\frac{81459819441}{5875911744784} & \sim 1.00621 \\
				\hline
			\end{array}
		\end{equation*}\caption{The function $r(t)$ and the value $\|R\|_1$ of Theorem \ref{teo:d0} for $S_{3,\bw^\lambda}$ several choices of $\phi$ and $2n-1<\lambda_n<2n$ .}
		\label{tabla2convergencia}
	\end{table}
		
	The next two sections are devoted to study the approximation and the noise suppression capability depending on the chosen weight function

\section{Approximation capability}\label{approx}

To study the approximation capability, we consider the subdivision scheme $S_{d,\mathbf{w^\lambda}}$ defined in \eqref{esquemasubdivision} with $d\geq 0$ and $\lambda$ satisfying the conditions requested in Proposition \ref{propnumero}. Let $F \in \mathcal{C}^{d+2}$ be and consider the initial data $\mathbf{f}^h=\{f^{h}_j\}_{j\in\mathbb{Z}}$ with $h>0$ and
$$f^{h}_{j}=F\left(j h\right),\quad j\in\mathbb{Z}.$$
Let $j_0\in\mathbb{Z}$ be any integer, we calculate the approximation error between  $(S_{d,\mathbf{w^\lambda}}\mathbf{f}^h)_{2j_0+i}$ and $F((j_0+i/2)h)$, with $i=0,1$, and analyse the largest contribution term.
By Taylor's theorem, we have that there exist $p_{i}\in\Pi_d$ such that:
\[f^{h}_{j}=F(jh)=p_i(jh)+\frac{F^{(d+1)}((j_0+i/2)h)}{(d+1)!}(j-(j_0+i/2))^{d+1}h^{d+1}+\cO(h^{d+2}). \]
Applying the subdivision operator and considering its polynomial reproduction capability,
\begin{equation}
\begin{split}
(S_{d,\mathbf{w^\lambda}}\mathbf{f}^h)_{2j_0+i}&=\sum_{l=1-n}^{L_n+i} a^{i}_{l} f^h_{j_0+l}=\sum_{l=1-n}^{L_n+i} a^{i}_{l}\left(p_i((j_0+l)h)+\frac{F^{(d+1)}((j_0+i/2)h)}{(d+1)!}(l-i/2)^{d+1}h^{d+1}+\cO(h^{d+2})\right)\\
&=\sum_{l=1-n}^{L_n+i} a^{i}_{l}p_i((j_0+l)h)+\frac{F^{(d+1)}((j_0+i/2)h)}{(d+1)!}h^{d+1}\sum_{l=1-n}^{L_n+i} a^{i}_{l}(l-i/2)^{d+1}+\cO(n h^{d+2})\\
&=p_i((j_0+i/2)h)+\frac{F^{(d+1)}((j_0+i/2)h)}{(d+1)!}h^{d+1}\sum_{l=1-n}^{L_n+i} a^{i}_{l}(l-i/2)^{d+1}+\cO(n h^{d+2})\\
&=F((j_0+i/2)h)+\frac{F^{(d+1)}((j_0+i/2)h)}{(d+1)!}h^{d+1}\sum_{l=1-n}^{L_n+i} a^{i}_{l}(l-i/2)^{d+1}+\mathcal{O}(n h^{d+2})
\end{split}
\end{equation}
Therefore, the largest contribution to the approximation error is given by
		\[
\frac{F^{(d+1)}((j_0+i/2)h)}{(d+1)!}h^{d+1}\sum_{l=1-n}^{L_n+i} a^{i}_{l} (l-i/2)^{d+1}.
		\]
We conclude that if two linear schemes are given, with the same approximation order, then the scheme with lesser value of
\begin{equation}\label{etavalue}
	\eta = \max\left\{\sum_{l=1-n}^{L_n} a^0_l l^{d+1},\sum_{l=1-n}^{L_n+1} a^1_l (l - \frac12)^{d+1}\right\}
\end{equation}
provides better approximators, in general. We observe that, if $a^i_l = n^{-1}H(l/n) + \cO(n^{-2})\approx n^{-1}H(l/n)$, for some function $H$, $i=0,1$, (in that case, $H(t):= \lim_{n\to\infty} n a^i_{t n}$), then
		\[
		\sum_{l=1-n}^{L_n+i} a^i_l (l - i/2)^{d+1}
		= n^{-1} \sum_{l=1-n}^{L_n+i} H(l/n) (l - i/2)^{d+1}
		= n^{d} \sum_{l=1-n}^{L_n+i} H(l/n) (l/n - \frac{i}{2n})^{d+1} = n^{d+1} \int_{-1}^1 t^{d+1} H(t) dt + \cO(n^{d}).
		\]
		Since the proposed schemes are odd-symmetric, then $H(t)=H(-t)$ and $\int_{-1}^1 t^{d+1} H(t) dt = 2I_{d+1}(H)$ and the approximation error is given by
		\begin{equation} \label{eq:aprox_error}
		2h^{d+1} n^{d+1}I_{d+1}(H)\frac{F^{(d+1)}((j_0+i/2)h)}{(d+1)!} + \cO(n h^{d+2})+ \cO(n^{d}h^{d+1}),
		\end{equation}
		which increases with $h,n$ and $I_{d+1}(H)$. We will test this formula in Section \ref{sec:lambda_increase}.

Now, we explore how the selection of $\phi$ influences $H$, with the aim of determining which $\phi$ is the best from an approximation point of view.

For $d=0,1$, it is easy to compute $H$ from the expression of $\ba^0,\ba^1$ in \eqref{eq:definition_d1_1} and \eqref{eq:definition_d1_2}. For instance, for $2n-1<\lambda<2n$,
\[
H(t) = \lim_{n\to\infty} n a^i_{t n} = \lim_{n\to\infty} n  \frac{\phi(|2 t n+i|/\lambda)}{\sum_{j=1-n}^{L_n}\phi(|2j+i|/\lambda)}
= \lim_{n\to\infty} n  \frac{\phi(|2 t n+i|/\lambda)}{2 n \int_0^1 \phi(t) dt + \cO(1)}
= \frac{\phi(|t|)}{2I_0(\phi)}.
\]
Hence, $2I_2(H)=I_2(\phi)/I_0(\phi)$. In Table \ref{tab:noise1}, we see that the smallest values are reached for $\phi(x) = e^{-\xi x}$ with large $\xi$ and for $\phi(x) = (1-x^p)^q$ with large $q$ or small $p$. We add for comparison $\|H\|_2^2$, that according to Section \ref{noise}, the smaller it is, the greater is its noise reduction capability. We can see for any scheme that the greater is the approximation capability, the smaller is the noise reduction capability. As conclusion, approximation and noise reduction are incompatible, in this sense, and some equilibrium may be found. This is further discussed in Section \ref{sec:pareto}.

For $d=2,3$, using the results in Section \ref{sec:conv_d3}:
\[
H(t) = \phi(|t|)\frac12\frac{I_4(\phi)-t^2 I_2(\phi)}{I_0(\phi)I_4(\phi)-I_2(\phi)^2}.
\]
Then, $2I_4(H)=-(I_2(\phi)I_6(\phi) - I_4(\phi)^2)/(I_0(\phi)I_4(\phi)-I_2(\phi)^2)$. The same conclusion can be obtain as in the case $d=0,1$ from Table \ref{tab:approx3}: The smallest values are reached for $\phi(x) = e^{-\xi x}$ with large $\xi$ and for $\phi(x) = (1-x^p)^q$ with large $q$ or small $p$. A great approximation power implies a low noise reduction capability, which will be studied in Section \ref{sec:pareto}.

\section{Noise reduction}\label{noise}
	
	In this section, we study the application of a subdivision operator to purely noisy data, $S_\ba \bepsilon$ where all the values $\epsilon_j$ follows a random distribution $E$, and are mutually uncorrelated. The results of this study can be applied to any data contaminated with noise due to Remark \ref{rmk:separation}.

	A direct result is that
	\[ \|S_\ba \bepsilon\|_\infty \leq  \|S_\ba\|_\infty \|\bepsilon\|_\infty. \]
	Since \(\|S_\ba\|_\infty \geq 1\) for any convergent schemes, the best condition is reached for $d=0,1$, for which $\|S_{d,\bw^\lambda}\|_\infty = 1$, since the mask is positive. Hence, it cannot be concluded from this formula that the noise is reduced.	

\begin{table}[!h]
	\begin{equation*}
		\arraycolsep=2.5pt\def\arraystretch{1.5}
		\begin{array}{llll}
			\hline
			\phi(x)  		& H(t)	&  2|I_2(H)|	& \|H\|_2^2 \\
			1				& \frac12	& \frac13 \simeq 0.333333 & \frac12 = 0.5 \\
			1-x 			& 1-|t| & \frac16 \simeq 0.166667 & \frac{2}{3} \simeq 0.666667\\
			1-x^2			& \frac{3}{4} \left(1-t^2\right)	& \frac15 = 0.2	& \frac35 = 0.6\\
			(1-x^2)^2		& \frac{15}{16} \left(1-t^2\right)^2 & \frac{1}{7} \simeq 0.142857 &\frac{5}{7} \simeq 0.714286\\
			(1-x^3)^3		& \frac{70}{81} \left(1-|t|^3\right)^3 &\frac{35}{243}\simeq 0.144033	& \frac{175}{247} \simeq 0.708502\\
			(1-x^2)^3		& \frac{35}{32} \left(1-t^2\right)^3 & \frac19 \simeq 0.111111	& \frac{350}{429} \simeq 0.815851\\
			(1-x^p)^q 		& \frac{\left(1-t^p\right)^q \Gamma \left(q+\frac{1}{p}+1\right)}{2 \Gamma \left(1+\frac{1}{p}\right) \Gamma (q+1)} & \frac{\Gamma \left(\frac{p+3}{p}\right) \Gamma \left(q+\frac{1}{p}+1\right)}{3 \Gamma \left(1+\frac{1}{p}\right) \Gamma \left(q+\frac{3}{p}+1\right)} & \frac{(2 q)! \Gamma \left(q+\frac{1}{p}+1\right)^2}{2 (q!)^2 \Gamma \left(1+\frac{1}{p}\right) \Gamma \left(2 q+\frac{1}{p}+1\right)} \\
							& & \text{Increases with $p$. Decreases with $q$.} & \text{Decreases with $p$. Increases with $q$.}\\
							& & \lim_{p\to+\infty}2|I_2(H)| = \frac13,\lim_{q\to+\infty}2|I_2(H)| = 0 & \lim_{q\to0}\|H\|_2^2 = \frac12, \lim_{q\to+\infty}\|H\|_2^2 = +\infty\\
			e^{-\xi x}		& \frac{\xi  e^{\xi  t}}{2 \left(e^{\xi }-1\right)} & \frac{2-\frac{\xi (\xi+2)}{e^\xi-1}}{\xi^2} &  \frac{1}{4} \xi  \coth \left(\frac{\xi }{2}\right) \\
							&  & \text{Decreases with $\xi$} &  \text{Increases with $\xi$} \\
			&  & \lim_{\xi\to0}2|I_2(H)| = \frac13, \lim_{\xi\to +\infty}2|I_2(H)| = 0 & \lim_{\xi\to0}\|H\|_2^2 = \frac12, \lim_{\xi\to +\infty}\|H\|_2^2 = +\infty \\
			\hline
		\end{array}
	\end{equation*}
	\caption{Computations of $H$, $2|I_2(H)| = |\int_{-1}^1 t^2 H(t) dt|$ and $\|H\|_2^2$ for several choices of $\phi$, $d=0,1$.} \label{tab:noise1}
\end{table}

\begin{table}[!h]
			\begin{equation*}
				\arraycolsep=2.5pt\def\arraystretch{1.5}
				\begin{array}{llll}
					\hline
					\phi(x)  		& H(t)	&  2|I_4(H) |                                                                                       & \|H\|_2^2                              \\\hline
					1				& \frac{3}{8} \left(3-5 t^2\right)	& \frac{3}{35} \simeq 0.0857143                                         & \frac98 \simeq 1.125                   \\
					1-x 			& \frac{6}{7} \left(5 t^2-2\right) (| t| -1) & \frac{19}{490} \simeq 0.0387755                              & \frac{456}{343} \simeq 1.32945         \\
					1-x^2			& \frac{15}{32} \left(t^2-1\right) \left(7 t^2-3\right)		& \frac{1}{21} \simeq 0.047619                  & \frac54 \simeq 1.25                    \\
					(1-x^2)^2		& \frac{-105}{64} \left(t^2-1\right)^2 \left(3 t^2-1\right) & \frac{1}{33} \simeq 0.030303                  & \frac{805}{572} \simeq 1.40734         \\
	(1-x^3)^3		& \frac{210 \left(770 t^2-243\right) \left(| t| ^3-1\right)^3}{32099}	& \frac{130734}{4590157} \simeq 0.0284814           & \frac{141820427509425}{99507706423177}\simeq 1.42522 \\
					(1-x^2)^3		& \frac{315}{512} \left(t^2-1\right)^3 \left(11 t^2-3\right)	& \frac{3}{143} \simeq 0.020979             & \frac{3780}{2431} \simeq 1.55492 \\
					(1-x^p)^q		& \text{(large formula}  & \text{Increases with $p$, decreases with $q$}            &  \text{Decreases with $p$, increases with $q$}\\
& \text{with gamma function)} & 0 \leq |I_d(H) | \leq \frac{3}{35} & \lim_{q\to +\infty}|I_d(H) | = +\infty, \\
					&  &
 & \lim_{p\to +\infty}|I_d(H) | = \lim_{q\to 0}|I_d(H) | = \frac98 \\ 										e^{-\xi x}		& \text{(explicit but large formula)} & \text{Decreases with $\xi$ and}  & \lim_{\xi\to0}\|H\|_2^2 = \frac98\\
					&  &
					0 \leq |I_d(H) | \leq \frac{3}{35} & \lim_{\xi\to +\infty}\|H\|_2^2 = +\infty
					\\
					\hline
				\end{array}
			\end{equation*}
			\caption{Computations of $H$, $2|I_4(H) | = |\int_{-1}^1 t^4 H(t) dt|$ and  $\|H\|_2^2$ for several choices of $\phi$, $d=2,3$.} \label{tab:approx3}
		\end{table}

	To reveal the denoising capabilities, a basic statistical analysis can be carried out. If the variance of the refined data is lesser than the variance of the given data, $\var(E)$, it indicates a reduction of randomness. Using that
	\[\var(\alpha X + \beta Y) = \alpha^2 \var(X) + \beta^2 \var(Y), \qquad \alpha,\beta\in\R,\]
	provided that $X,Y$ are two uncorrelated random distributions, the variance after one subdivision step is
	\[
	\var\left(\sum_{l\in \Z} a_{2l+i}E\right) = \sum_{l\in \Z} a_{2l+i}^2\var\left(E\right) = \|\ba^i\|_ 2^2\var\left(E\right), \quad i=0,1.
	\]
	Hence, the variance reduction is given by
	\[
	\|S_\ba\|_2^2 = \max\{\|\ba^0\|_2^2,\|\ba^1\|_2^2\}.
	\]
	For some schemes studied in this work, this quantity is: For $d=0,1$, if $2n-1<\lambda<2n$,
	\begin{equation*}
		\|S_{1,\mathbf{w^\lambda}}\|_2^2= \max\left\{\sum_{l=-n+1}^{n-1} \left(\frac{w^\lambda_{2l}}{||\mathbf{w^\lambda_0}||_1}\right)^2 , \
		\sum_{l=-n+1}^{n}\left(\frac{w_{2l-1}^\lambda}{||\mathbf{w^\lambda_1}||_1}\right)^2 \right\} < 1.
	\end{equation*}
	The last quantity is less than one owned to the constant reproduction and the positivity of the coefficients.
	In case that $\phi(x) = 1$, then $\|S_{1,\texttt{rect}^\lambda}\|_2^2 = \lfloor \lambda \rfloor^{-1} = (2n-1)^{-1}$, which is the lowest value that can be obtained with a rule of this length. For $d=2,3$, $\phi(x)=1$ and $2n-1<\lambda<2n$,
	\[\|S_{3,\texttt{rect}^\lambda}\|_2^2 = \frac{9 n^2-9 n-3}{8 n^3-12 n^2-2 n+3} > (2n-1)^{-1}, \qquad \forall n\geq 2,\]
	which maximum is achieved for $n=2$ (i.e. $3<\lambda<4$, corresponding to the interpolatory DD4 scheme), which is 1.
	
	Two results can be derived: First, if the variance is reduced in each iteration by a factor $\|S_{d,\mathbf{w^\lambda}}\|_2^2<1$, then the limit function has variance 0. Second, since $\lim_{n\to\infty}\|S_{3,\texttt{rect}^\lambda}\|_2^2 = 0$, the noise tends to be completely remove when the mask support tends to $\infty$.	
	
	For any choice of $\phi(x)$, an asymptotic result can be given for the noise reduction using an argument  similar to Section \ref{approx}: If $a^i_l = n^{-1}H(l/n) + \cO(n^{-2})$, for some function $H$, $i=0,1$, then
	\[
	\lim_{n\to\infty} \|\ba^i\|_2^2 =\lim_{n\to\infty} n^{-2}\sum_{l=1-n}^{L_n+i} H(l/n)^2 = n^{-1}\int_{-1}^1 H(t)^2 dt,
	\]
	so that the noise reduction factor behaves asymptotically as
	\begin{equation} \label{eq:variance_limit}
		\|S_\ba\|_2^2 = n^{-1} \|H\|_2^2 + \cO(n^{-2}).
	\end{equation}
	Under these assumptions, we observe that the noise is always removed after an iteration when $n\to\infty$.
	In the Tables \ref{tab:noise1} and \ref{tab:approx3} we compute $H(t):= \lim_{n\to\infty} n a^i_{t n}$ and the factor $\|H\|_2^2$ for several $\phi$ functions, $d=0,1,2,3$.
	
	\subsection{An equilibrium between approximating and denoising} \label{sec:pareto}

	We have seen that, in order to maximize the approximation and denoising capabilities, the values $I_4(H)$ and $\|H\|_2$ should be minimized. This is a multi-objective minimization problem, which solutions form a Pareto front that we have estimated using the MATLAB optimization toolbox. Here we will only consider the case $d=2,3$, but a similar analysis can be performed with $d=0,1$.
	
	First, observe Figure \ref{fig:paretoxicomparison}-left. We find out that $\phi(x) = (1-x^p)^q$ is always more convenient than $\phi(x) = e^{-\xi x}$, meaning that for each value of $\xi$ there exists some pair $(p,q)$ for which $\phi(x) = (1-x^p)^q$ approximates and denoises better than $\phi(x) = e^{-\xi x}$.
	It can also be affirm that $\phi(x)=1$ is in the Pareto front and it the best for noise reduction and the worst for approximating. In the other extreme would be an interpolatory scheme, with the best approximation capability but the worst denoising power.
	
	The Pareto-optimal values $(p,q)$ for $\phi(x) = (1-x^p)^q$ form a curve (see Figure \ref{fig:paretoxicomparison}-right) which seems to interpolate the integer values $(2,1)$ and $(4,5)$.
	\begin{figure}[!h]
		\centering
			\includegraphics[width=0.47\linewidth]{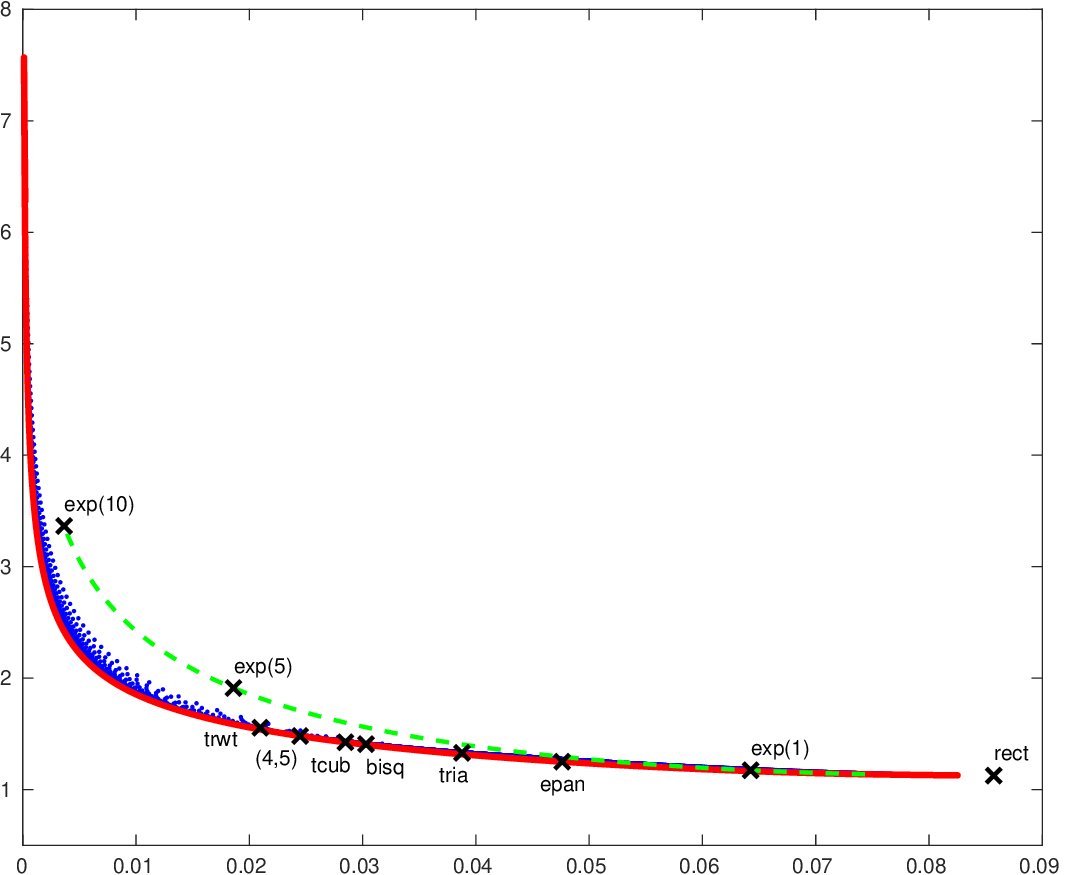}  \includegraphics[width=0.47\linewidth]{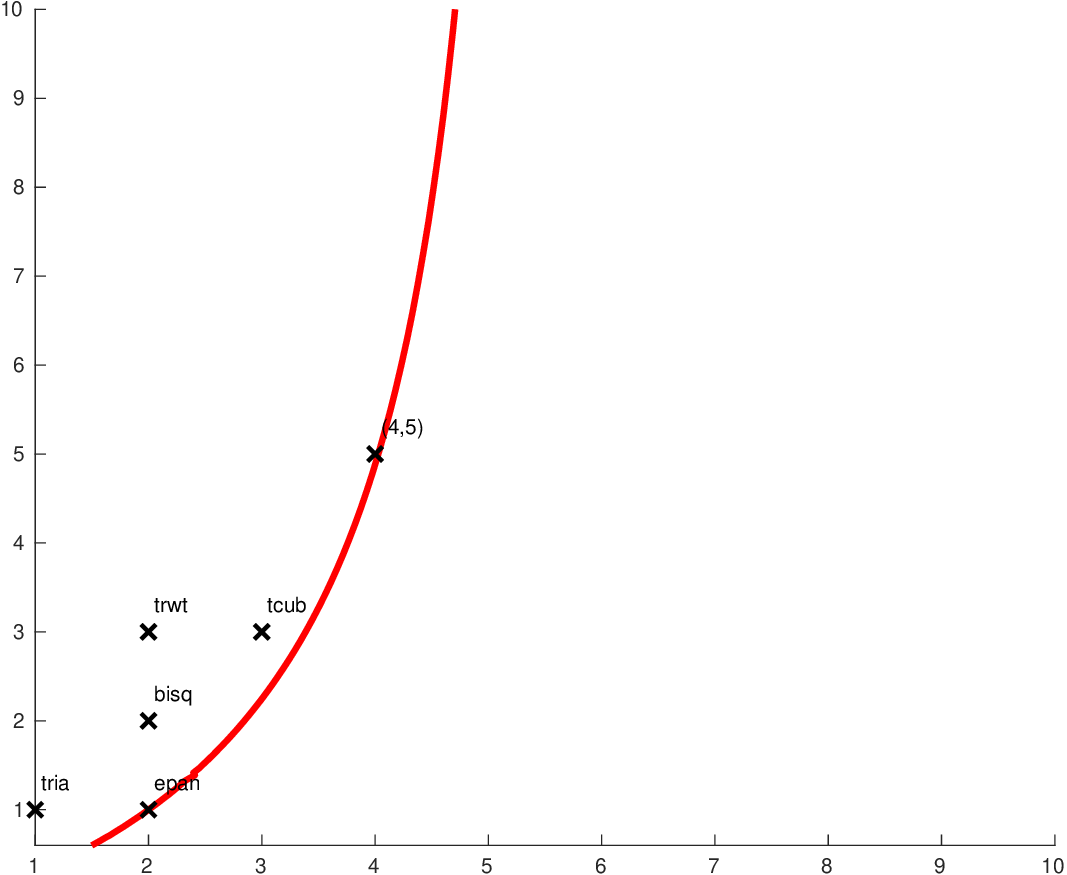}
		\caption{Left, the pair of values $(2I_4(H),\|H\|_2^2)$ for several choices of $\phi$. Thus, the lower $x$ and $y$ axis values, the better approximation and denoising capabilities, respectively. Blue, $\phi(x) = (1-x^p)^q$ for several values $(p,q)$ pairs such that $1\leq p \leq 20$, $\frac12\leq q \leq 20$; red, the Pareto front of the previous pairs; green, $\phi(x) = \exp(-\xi x)$ for $\frac12\leq \xi\leq 10$. Right, the red line represents the pairs of values ($p,q$) for which $\phi(x) = (1-x^p)^q$ is Pareto-optimal.}
		\label{fig:paretoxicomparison}
	\end{figure}

	In conclusion, we recommend the use of {\tt rect} to obtain the best denoising. However, with {\tt epan} the noise increases by $11.11 \%$ while the approximation error is reduced by $44.44 \%$ compared to {\tt rect}. If the approximation is desired to be prioritized, $\phi(x) = (1-x^4)^5$ is a good choice, since the noise increases by $31.58 \%$ while the approximation error is reduced by $71.43 \%$, compared to {\tt rect}. The rest of the $(p,q)$ values related to Table \ref{tabla1nucleos} are near to be optimal and can be used as well for other approximating-denoising balances. We recommend to never use {\tt exp($\xi$)}.
	
	Just to mention that for $d=0,1$ similar conclusions can be obtained.  For that polynomial degrees, the weight functions $\phi(x) = \exp(-\xi x)$ are also worse than $\phi(x) = (1-x^p)^q$. The weight function {\tt epan} is still Pareto optimal, but the pair  $(p,q)=(4,5)$ is not.

	\section{Numerical experiments}\label{expnum}

	In this section, we present some numerical examples to show how the new schemes work for the generation of curves. We check that the subdivision schemes are convergent for $d=0,1,2,3$ and that the curve present $\mathcal{C}^1$ smoothness (but not $\mathcal{G}^1$, meaning that kinks can be produced). We analysed the approximating and denoising capabilities to numerically validate the results in Sections \ref{approx} and \ref{noise}. Only for $d=0,1$, we test the conservation of the monotonicity applying the schemes to fit a non-decreasing initial data. Finally, we perform a numerical test using the discretization of a discontinuous function and observe that the proposed methods avoid Gibbs phenomenon in the neighbourhood of an isolated discontinuity for $d=0,1$.
	
	\begin{figure}[!h]
		\centering
		\begin{tabular}{cccc}
			$S_{1,{\tt rect}^{9.5}}$ & $S_{3,{\tt rect}^{15.5}}$ & $S_{1,{\tt p4q5}^{9.5}}$ & $S_{3,{\tt p4q5}^{15.5}}$ \\ \includegraphics[width=0.21\linewidth]{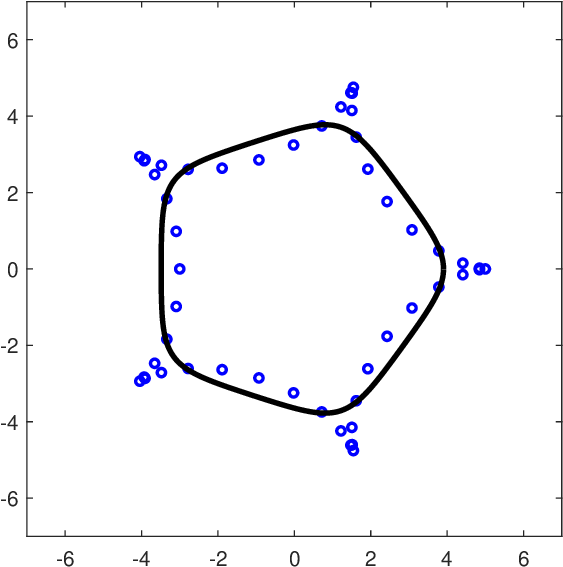} &
			\includegraphics[width=0.21\linewidth]{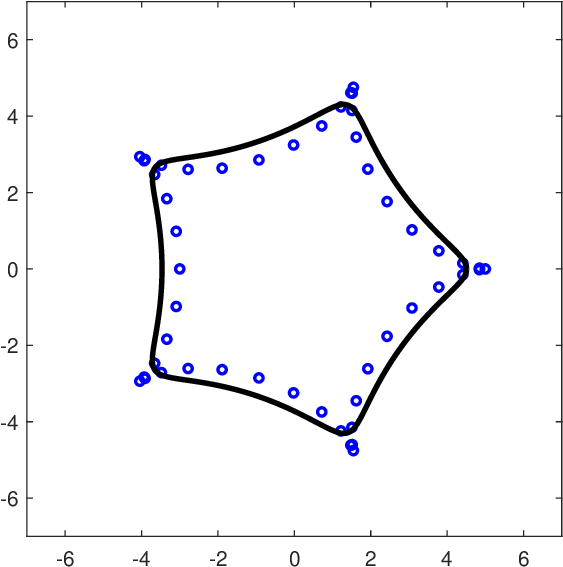} &
			\includegraphics[width=0.21\linewidth]{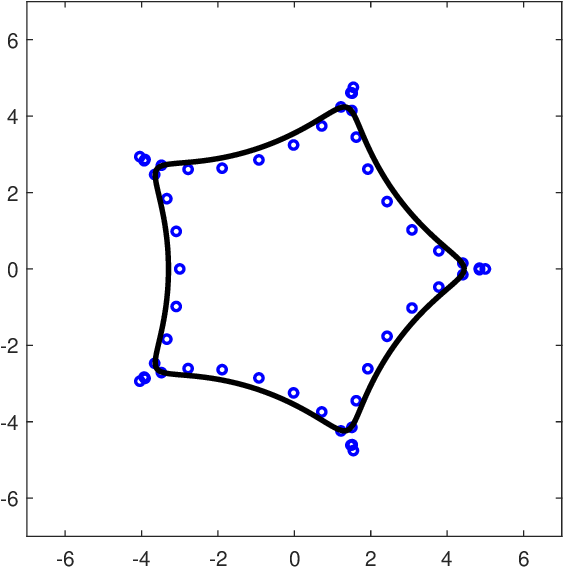} &
			\includegraphics[width=0.21\linewidth]{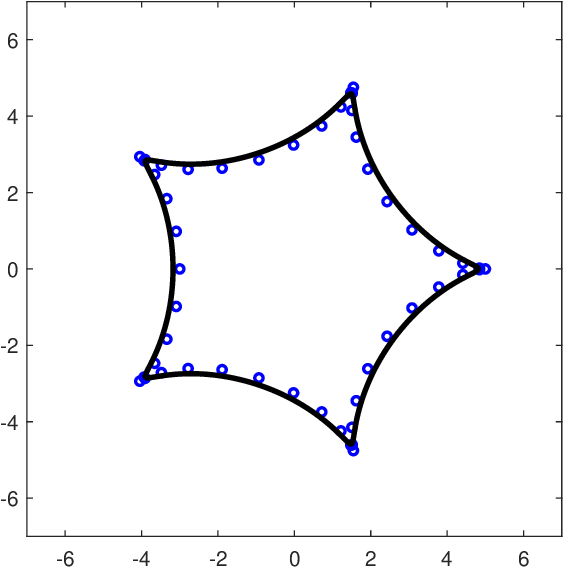} \\
			\includegraphics[width=0.21\linewidth]{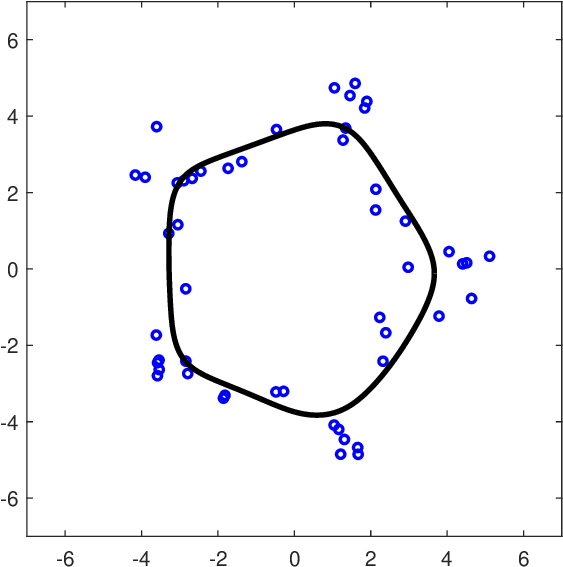} & \includegraphics[width=0.21\linewidth]{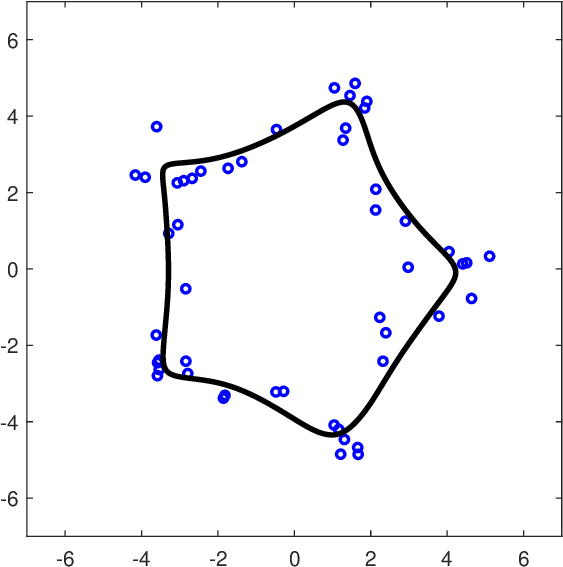} &
			\includegraphics[width=0.21\linewidth]{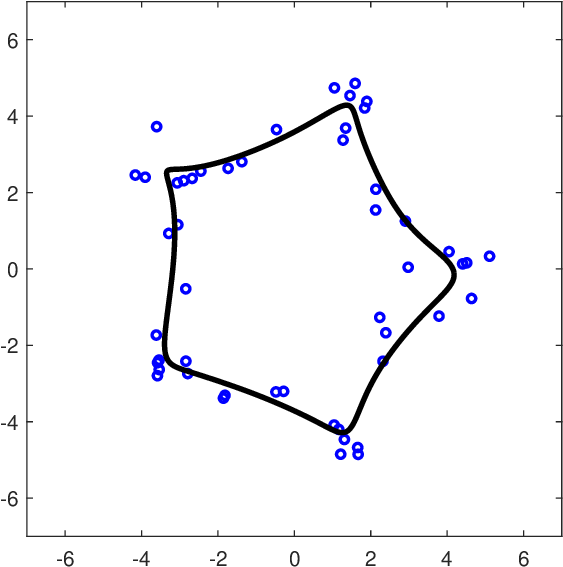} &
			\includegraphics[width=0.21\linewidth]{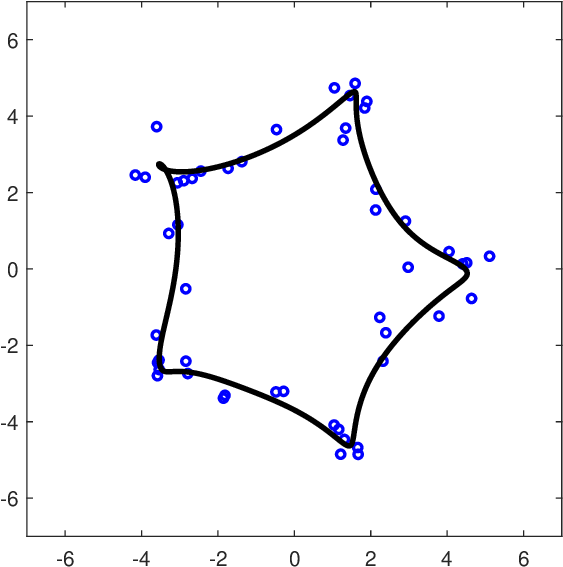} \\
			\includegraphics[width=0.21\linewidth]{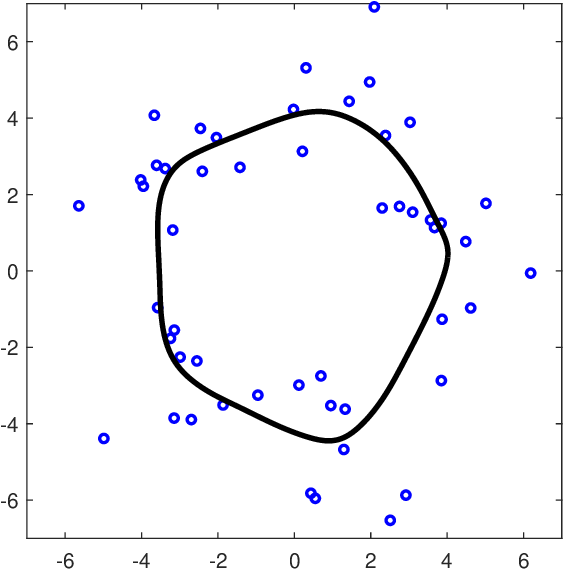} & \includegraphics[width=0.21\linewidth]{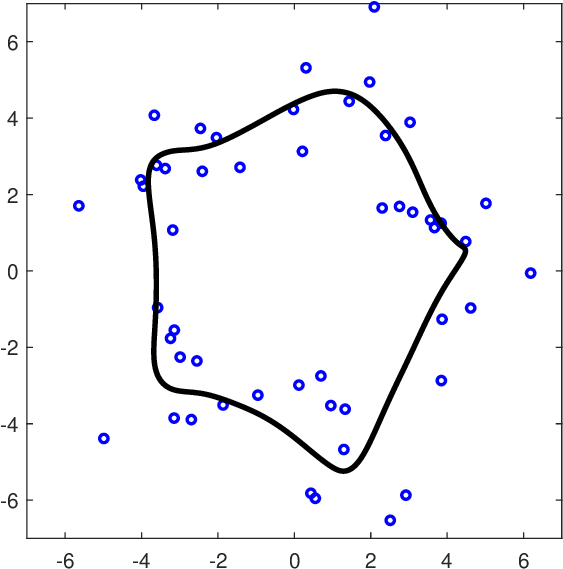} &
			\includegraphics[width=0.21\linewidth]{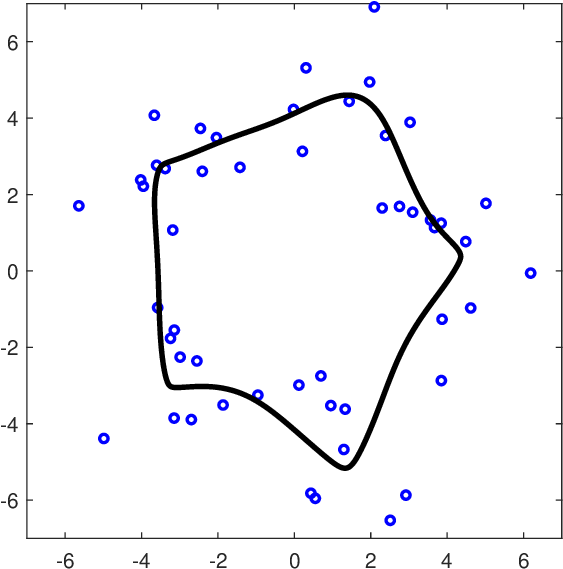} &
			\includegraphics[width=0.21\linewidth]{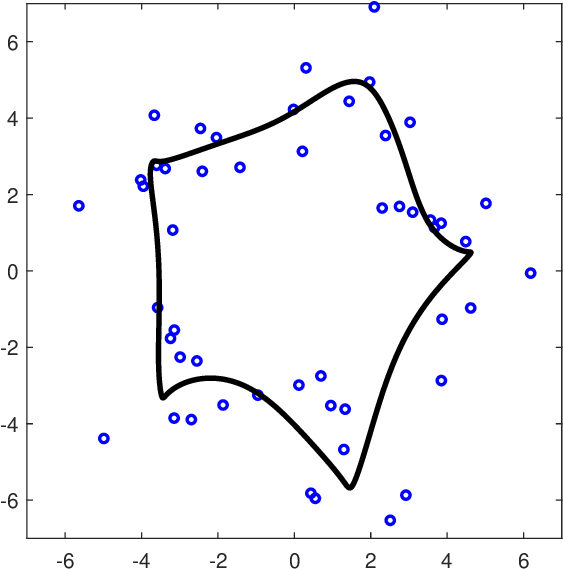} \\
		\end{tabular}
		\caption{Several subdivision schemes (by columns) applied to the star-shaped data in
			\eqref{eqnumexpej2}. In the first row, they are applied to the original data. In the second and third row, the data is contaminated by normal noise with $\sigma=0.5$ and $\sigma=1$, respectively.}
		\label{fig:experimento4}
	\end{figure}

\subsection{Application to noisy geometric data}

We start with one of the experiments presented in \cite{dynheardhormannsharon} which consists of a star-shaped curve given by:
\begin{equation}\label{eqnumexpej2}
	F(t) = (4\cos(t)+\cos(4t), 4\sin(t)-\sin(4t)),
\end{equation}
with samples taken at $t^0_j=j\pi/25$ with $j\in\Z$. That is, we consider $\mathbf{f}^0 := F|_{{\mathbf t}^0}$, ${\bf t}^0 =\{t^0_j\}_{j\in\Z}$, i.e. $f^0_j = F(t^0_j)$. Because of the periodicity of the function, we can focus on $j=0,\hdots, 49$. We add Gaussian noise in each component, defining $\tilde{\mathbf{f}}^0=\mathbf{f}^0+\bepsilon^\sigma$ with $\bepsilon^\sigma=\{(\varepsilon^{\sigma,1}_j,\varepsilon^{\sigma,2}_j)\}_{j=0}^{49}$,   being $\varepsilon^{\sigma,l}_j \sim \mathcal{N}(0,\sigma)$, $l=1,2,$ $j=0,\hdots,49$ and $\sigma\in\{0.5,1\}$.
In Figure \ref{fig:experimento4}, we illustrate the results only for two interesting choices of $\phi$, according to the conclusions in Section \ref{sec:pareto}. Nevertheless, the results obtained with the rest of weight functions are graphically similar and they are shown in detail in Table \ref{tab:errores2}.

Without noise, the smaller is $\lambda$ the more accurate are the results for any $\phi$ and $d$. Measuring the approximation error as
\begin{equation*}
	\|S^5 \mathbf{f}^0 - F|_{{\mathbf t}^5}\|_\infty =\max_{j\in \mathbb{Z}}\{\|(S^5 \mathbf{f}^0)_j - F(t^5_j)\|_2\},
\end{equation*}
where $t^5_j := 2^{-5}t^0_j$, $j\in\Z$. As expected, we can see in Table \ref{tab:errores2} that the approximation error is always smaller for $d=2,3$ than $d=0,1$. But also, if we sort the weight function by the approximation order, for any $0\leq d \leq 3$, they would be exactly in the same order as if we sort them by the theoretical approximation power in Tables \ref{tab:noise1} and \ref{tab:approx3}.
Nevertheless, it has to be taken into account that the results in Tables \ref{tab:noise1} and \ref{tab:approx3} have an asymptotic nature, when $h\to0$.
These behaviours are also visible in the first row of Figure \ref{fig:experimento4}.

We measure the noise reduction capability of the schemes with the quantities $\|S^5\bepsilon^{0.5}\|_\infty$ and $\|S^5\bepsilon^{1}\|_\infty$, in Table \ref{tab:errores2}. The ones that show results closer to zero are the schemes with higher denoising capacity. In general, the noise is more reduced when $\lambda$ is larger or $d$ is smaller. Comparing the sorting of the weight functions by its theoretical denoising capability, according to Tables \ref{tab:noise1} and \ref{tab:approx3}, and by the numbers in Table \ref{tab:errores2}, we see that both orderings are the same, in general. Only in some particular cases this ordering is slightly changed. The reason may be that the results on Tables \ref{tab:noise1} and \ref{tab:approx3} are asymptotical, for $\lambda\to\infty$. But also, the reduction is in terms of the variance of the statistical distribution. Hence, the same experiment should be repeated many times and the results averaged in order to obtain a more consistent comparison.

In Figure \ref{fig:experimento4}, we can see how important the choice of the weight function is to increase the approximation capability (and only losing a bit of denoising capability). In turn, taking $d=2,3$ gives better approximations and $\lambda$ can also be increased to reduce noise.

Of course, during our study we generated much more graphics than the ones here presented. In some of them, specially in presence of noise, artefacts may appear, such as auto-intersections or kinks, proving that it does not provide $\mathcal{G}^1$ curves, even if the scheme is $\cC^1$. By taking $\lambda$ larger, the artefacts usually disappear and curves become softer.

\begin{table}[!h]
	\centering
	\begin{tabular}{|l|rrrr|rrrr|}\hline
		& \multicolumn{4}{c|}{$d=0,1$} & \multicolumn{4}{c|}{$d=2,3$} \\
		$\lambda$ & 3.7 & 5.8 & 9.5 & 15.5 &3.7 & 5.8 & 9.5 & 15.5\\\hline
		 & \multicolumn{4}{c|}{{\tt rect}}  &  \multicolumn{4}{c|}{{\tt rect}}\\
		{$\|S^5 \mathbf{f}^0 - F|_{\mathbf{t}^5}\|_\infty$} & 1.943e-1 &4.578e-1 &1.095e-0 &1.844e-0 &1.487e-3 &1.038e-2 &9.402e-2 &4.899e-1 \\
		$\|S^5\bepsilon^0_{0.5}\|_\infty$                     & 7.496e-1 &5.256e-1 &3.272e-1 &2.506e-1 &1.459e-0 &9.363e-1 &7.312e-1 &4.073e-1   \\
		$\|S^5\bepsilon^0_{1}\|_\infty$                       & 1.263e-0 &9.790e-1 &6.786e-1 &4.712e-1 &3.143e-0 &1.691e-0 &1.151e-0 &9.408e-1     \\ \hline
		 & \multicolumn{4}{c|}{{\tt tria}}  &  \multicolumn{4}{c|}{{\tt tria}}\\
		{$\|S^5 \mathbf{f}^0 - F|_{\mathbf{t}^5}\|_\infty$} & 1.158e-1 &2.695e-1 &6.393e-1 &1.254e-0 &1.487e-3 &6.683e-3 &4.927e-2 &2.624e-1 \\
		$\|S^5\bepsilon^0_{0.5}\|_\infty$                     & 7.604e-1 &6.518e-1 &4.235e-1 &2.859e-1 &1.459e-0 &9.048e-1 &8.035e-1 &5.298e-1   \\
		$\|S^5\bepsilon^0_{1}\|_\infty$                       & 1.514e-0 &1.170e-0 &8.941e-1 &6.159e-1 &3.143e-0 &1.957e-0 &1.327e-0 &1.086e-0     \\
		\hline
		 & \multicolumn{4}{c|}{{\tt bisq}}  &  \multicolumn{4}{c|}{{\tt bisq}}\\
		{$\|S^5 \mathbf{f}^0 - F|_{\mathbf{t}^5}\|_\infty$} & 1.012e-1 &2.363e-1 &5.648e-1 &1.152e-0 &1.487e-3 &5.986e-3 &3.876e-2 &2.157e-1  \\
		$\|S^5\bepsilon^0_{0.5}\|_\infty$                     & 7.785e-1 &6.816e-1 &4.603e-1 &2.957e-1 &1.459e-0 &9.140e-1 &8.382e-1 &5.642e-1   \\
		$\|S^5\bepsilon^0_{1}\|_\infty$                       & 1.580e-0 &1.209e-0 &9.301e-1 &6.546e-1 &3.143e-0 &1.959e-0 &1.379e-0 &1.101e-0   \\ \hline
		  & \multicolumn{4}{c|}{{\tt trwt}}  &  \multicolumn{4}{c|}{{\tt trwt}}\\
		{$\|S^5 \mathbf{f}^0 - F|_{\mathbf{t}^5}\|_\infty$} & 7.892e-2 &1.859e-1 &4.551e-1 &9.729e-1 &1.487e-3 &4.134e-3 &2.725e-2 &1.575e-1 \\
		$\|S^5\bepsilon^0_{0.5}\|_\infty$                     & 8.353e-1 &7.111e-1 &5.256e-1 &3.068e-1 &1.459e-0 &9.860e-1 &8.553e-1 &6.289e-1  \\
		$\|S^5\bepsilon^0_{1}\|_\infty$                       & 1.794e-0 &1.290e-0 &1.002e-0 &7.340e-1 &3.143e-0 &2.128e-0 &1.440e-0 &1.130e-0  \\ \hline
		  & \multicolumn{4}{c|}{{\tt epan}}  &  \multicolumn{4}{c|}{{\tt epan}}\\
		{$\|S^5 \mathbf{f}^0 - F|_{\mathbf{t}^5}\|_\infty$}  & 1.402e-1 &3.209e-1 &7.481e-1 &1.416e-0 &1.487e-3 &8.265e-3 &6.033e-2 &3.161e-1\\
		$\|S^5\bepsilon^0_{0.5}\|_\infty$                      & 7.497e-1 &6.224e-1 &3.738e-1 &2.832e-1 &1.459e-0 &9.054e-1 &7.975e-1 &4.861e-1  \\
		$\|S^5\bepsilon^0_{1}\|_\infty$                        & 1.395e-0 &1.113e-0 &8.341e-1 &5.575e-1 &3.143e-0 &1.798e-0 &1.279e-0 &1.051e-0   \\
		\hline
		  & \multicolumn{4}{c|}{{\tt tcub}}  &  \multicolumn{4}{c|}{{\tt tcub}}\\
		{$\|S^5 \mathbf{f}^0 - F|_{\mathbf{t}^5}\|_\infty$} & 1.010e-1 &2.382e-1 &5.716e-1 &1.171e-0 &1.487e-3 &5.726e-3 &3.656e-2 &2.072e-1  \\
		$\|S^5\bepsilon^0_{0.5}\|_\infty$                     & 7.787e-1 &6.872e-1 &4.554e-1 &3.023e-1 &1.459e-0 &9.214e-1 &8.547e-1 &5.677e-1   \\
		$\|S^5\bepsilon^0_{1}\|_\infty$                       & 1.547e-0 &1.203e-0 &9.221e-1 &6.453e-1 &3.143e-0 &1.935e-0 &1.391e-0 &1.104e-0  \\ \hline
  & \multicolumn{4}{c|}{{\tt p4q5}}  &  \multicolumn{4}{c|}{{\tt p4q5}}\\
		{$\|S^5 \mathbf{f}^0 - F|_{\mathbf{t}^5}\|_\infty$} & 9.509e-2 &2.286e-1 &5.533e-1 &1.147e-0 &1.487e-3 &4.666e-3 &3.188e-2 &1.840e-1  \\
		$\|S^5\bepsilon^0_{0.5}\|_\infty$                     & 7.928e-1 &6.993e-1 &4.649e-1 &3.080e-1 &1.459e-0 &9.606e-1 &8.729e-1 &5.885e-1   \\
		$\|S^5\bepsilon^0_{1}\|_\infty$                       & 1.569e-0 &1.214e-0 &9.299e-1 &6.542e-1 &3.143e-0 &1.984e-0 &1.413e-0 &1.118e-0  \\ \hline
	\end{tabular}
	\caption{Analysis of the approximation and denoising capabilities for the different subdivision schemes with $d=0,1,2,3$
		and $\lambda=3.7, 5.8, 9.5$ and 15.5.}\label{tab:errores2}
\end{table}

\subsection{Approximation error when $\lambda$ is being increased} \label{sec:lambda_increase}

In this section we challenge formula \eqref{eq:aprox_error} with a suited experiment.
Let us consider $G(x) = \cos(\pi x)$ and the initial data $\mathbf{g}^{0,h}=\{g^{0,h}_j\}_{j\in\mathbb{Z}}$ and  $\widetilde{\mathbf{g}}^{0,h}=\{\widetilde{g}^{0,h}_j\}_{j\in\mathbb{Z}}$ with
\(
g^{0,h}_j = G(j h),
\)
\(
\widetilde{g}^{0,h}_j = g^{0,h}_j + \epsilon_j,
\)
\(
\epsilon_j\sim U\left (\left [-\frac14,\frac14\right ]\right ),
\)
where $U(I)$ is the uniform distribution in the interval $I$. We consider the spacings $h_k=10^{-k}$ and the support parameters $\lambda_k = 3.5 + 10^{k-1} = 3.5 + 0.1/h_k$, $k=1,2,3,4$. The value $\lambda_k$ is modified accordingly to $h_k$ to maintain almost constant the support of the basic limit function, which determines the influence of each data point on the limit function.
The results of applying 5 iterations of the scheme $S_{3,\texttt{rect}^{\lambda_k}}$ to $\widetilde{\mathbf{g}}^{0,h_k}$, for $k=1,2,3,4$, are shown in Figure \ref{fig:noisereductiond2n3}. On the one hand, it shows how the noise after five iterations tends to 0 if $k\to\infty$, but slowly, since the variance decay speed is $\cO(n^{-1})$. On the other hand, the approximation error does not decay to zero, as can be observed in Table \ref{tab:noisereductiond2n3}, where the numbers are never smaller than (and seems to tend to) the asymptotic error estimation in \eqref{eq:aprox_error}, which is (for $j=0$, $i=0$)
\[ \left |2I_4(H)\frac{G^{(4)}(0)}{4!} h_k^4 n_k^{4} \right | = \frac{3}{35}\frac{|G^{(4)}(0)|}{24} h_k^{4} (3+0.1/h_k)^{4} \overset{k\to+\infty}{\longrightarrow} \frac{\pi^{4}}{24} \cdot 0.1^{4}\cdot \frac{3}{35} \simeq \text{3.4789e-05}.
\]
This threshold is not a real constrain in practice, since the noise is usually greater than the approximation error (see first row of Table \ref{tab:noisereductiond2n3}). If an approximation error tending to zero is needed,  $n \propto h^{-\frac12}$ can be chosen, for instance.

\begin{table}[!h]
	\centering
	\begin{tabular}{lcccc} \hline
		& $h_1=10^{-1}$ &  $h_2=10^{-2}$ &  $h_3=10^{-3}$ & $h_4=10^{-4}$ \\
		$|(S_{3,\texttt{rect}^{\lambda_k}}^5 \widetilde{\mathbf{g}}^{0,h_k})_0 - G(0)|$ & 4.8215e-2	& 2.3619e-2	& 8.8646e-3	& 1.9704e-4	\\[5pt]
		$|(S_{3,\texttt{rect}^{\lambda_k}}^5 \mathbf{g}^{0,h_k})_0 - G(0)|$			& 5.9734e-3	& 9.6240e-5	& 4.1201e-5	& 3.7387e-5	\\\hline
	\end{tabular}
	\caption{The approximation error at $x=0$ after five iterations of $S_{3,\texttt{rect}^{\lambda_k}}$ applied to data with and without noise.}
	\label{tab:noisereductiond2n3}
\end{table}

\begin{figure}[!h]
	\centering
	\includegraphics[width=0.45\linewidth]{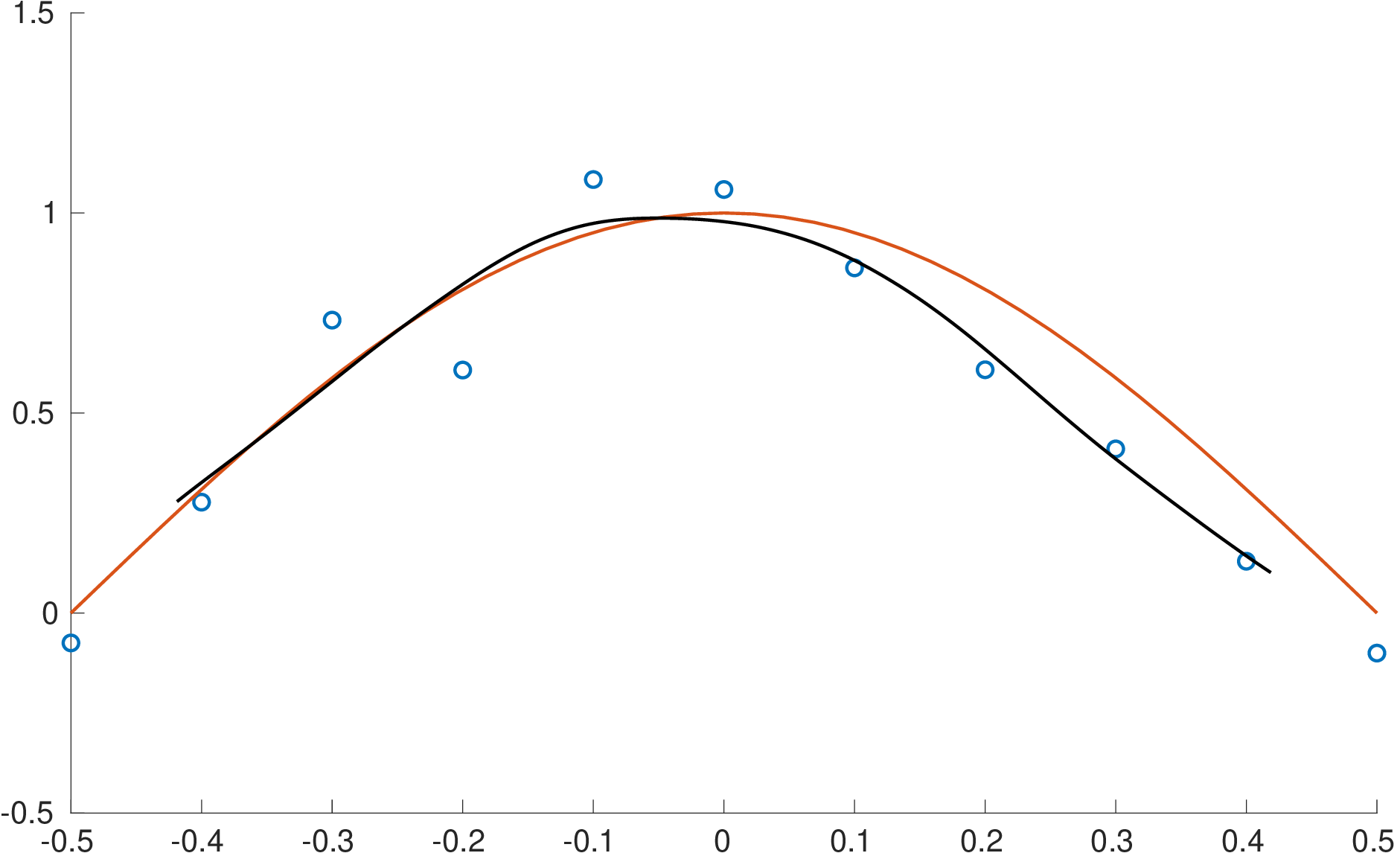}
	\includegraphics[width=0.45\linewidth]{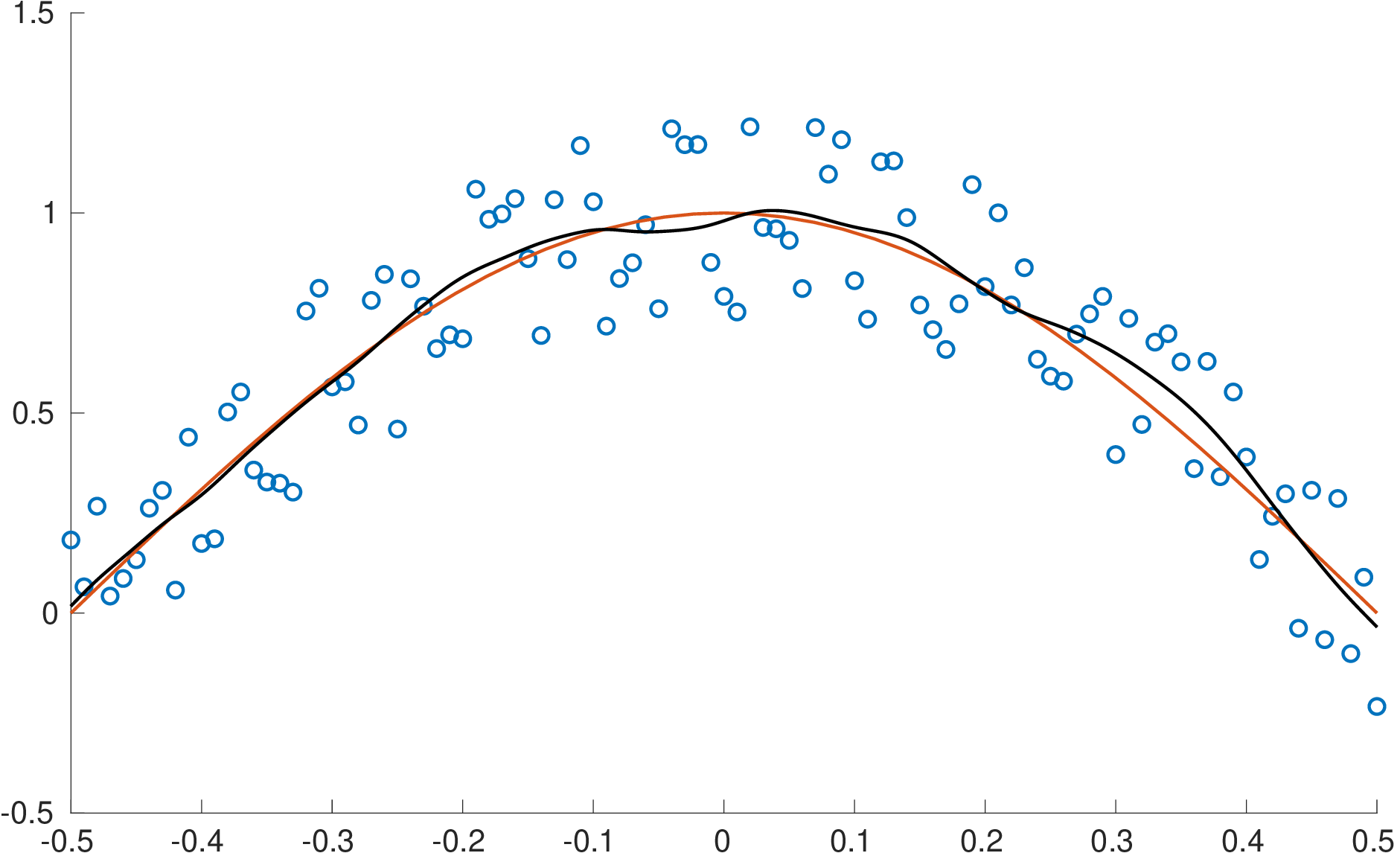}
	\includegraphics[width=0.45\linewidth]{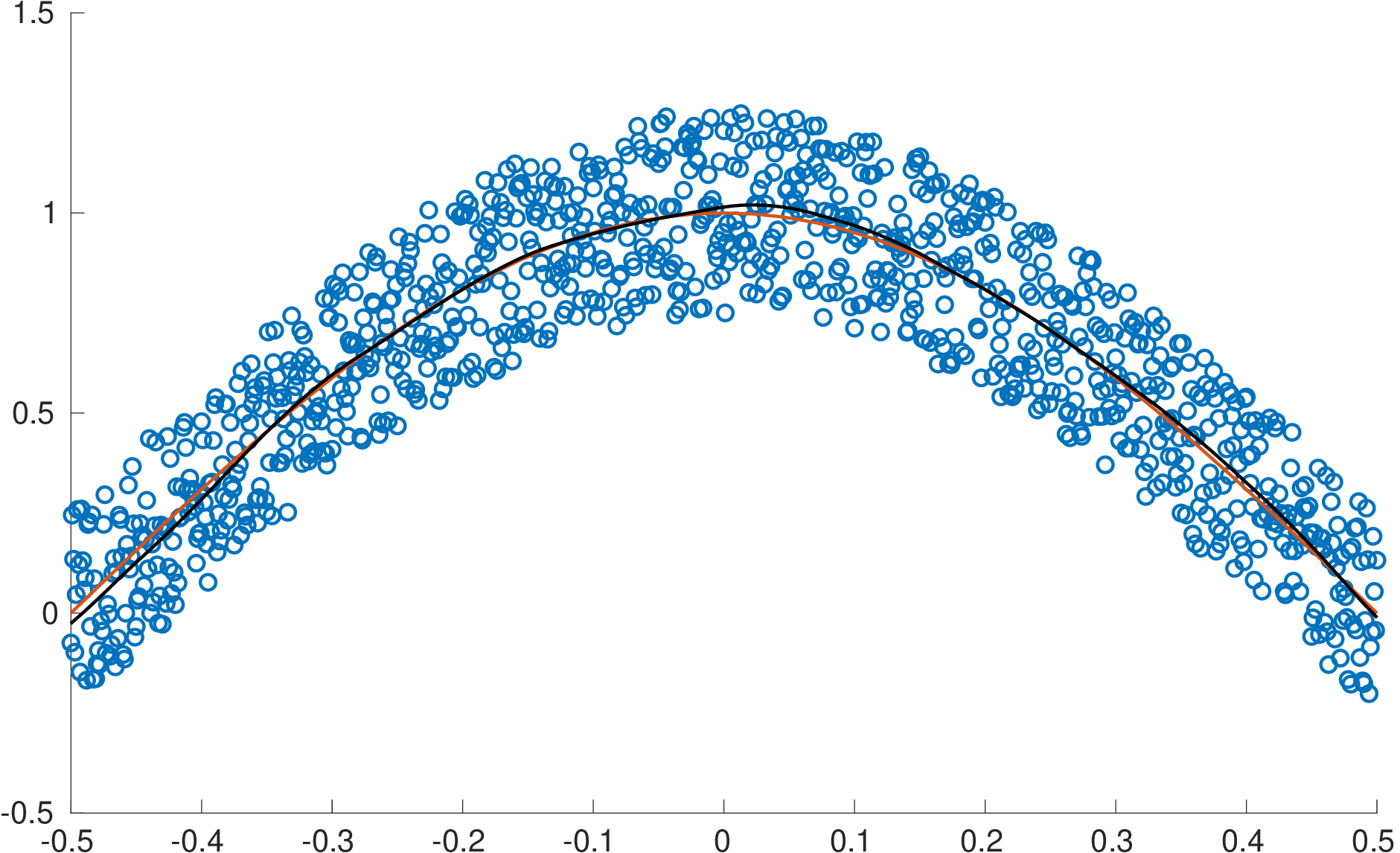}
	\includegraphics[width=0.45\linewidth]{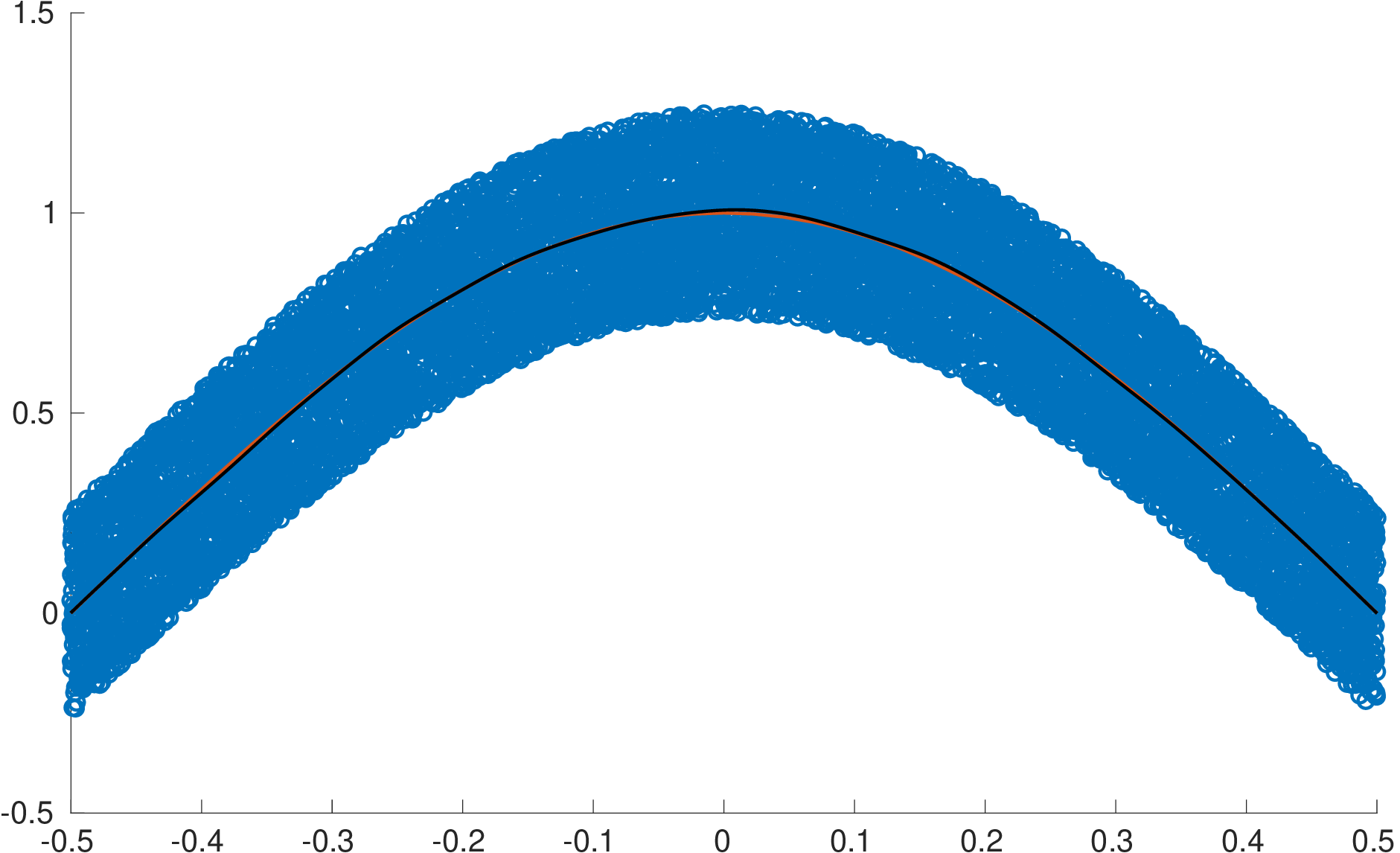}
	\caption{Five iterations of $S_{3,\texttt{rect}^{\lambda_k}}$ applied to $\widetilde{\mathbf{g}}^{0,h_k}$, for $k=1,2,3,4$. The blue circles are the initial data, the red line is the smooth function $G$ and the black line represents the limit function. The parameters for each graphic are (by rows): $h_1=10^{-1}$, $\lambda_1 = 4.5$;  $h_2=10^{-2}$, $\lambda_2 = 13.5$;  $h_3=10^{-3}$, $\lambda_3 = 103.5$;  $h_4=10^{-4}$, $\lambda_4 = 1003.5$.}
	\label{fig:noisereductiond2n3}
\end{figure}

\subsection{Avoiding Gibbs phenomenon}
	
	In this section we confirm that the subdivision schemes based on weighted-least squares with $d=0,1$ avoid Gibbs phenomenon, as stated in Corollary \ref{cor:gibbs}. To study it, we propose the following experiment. We discretize the function:
	\begin{equation}
		f(x)=\left\{
		\begin{array}{ll}
			\sin(\pi x), &  x\in[0,0.5]; \\
			-\sin(\pi x), & x\in (0.5,1],
		\end{array}
		\right.
	\end{equation}
	in the interval $[0,1]$ with 33 equidistant points, $x_i=i\cdot h$, $i=0,\hdots,32$ and $h=\frac{1}{32}$ and apply the subdivision schemes. We show the results in Figure \ref{fig:experimento8}. It is clearly visualize that the Gibbs phenomenon does not appear around the discontinuity, but there is diffusion, instead. The larger is $\lambda$, the more diffusion, specially when {\tt rect} is used.
	\begin{figure}[H]
		\begin{tabular}{cccc}
			$\lambda=2.5$ & $\lambda=4.5$ & $\lambda=5.5$ & $\lambda=6.5$ \\
			\includegraphics[width=0.225\linewidth]{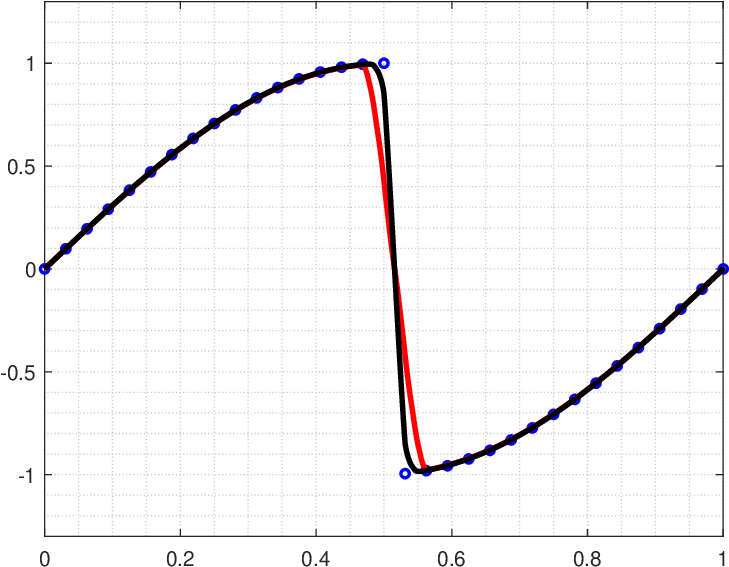} & 	
			\includegraphics[width=0.225\linewidth]{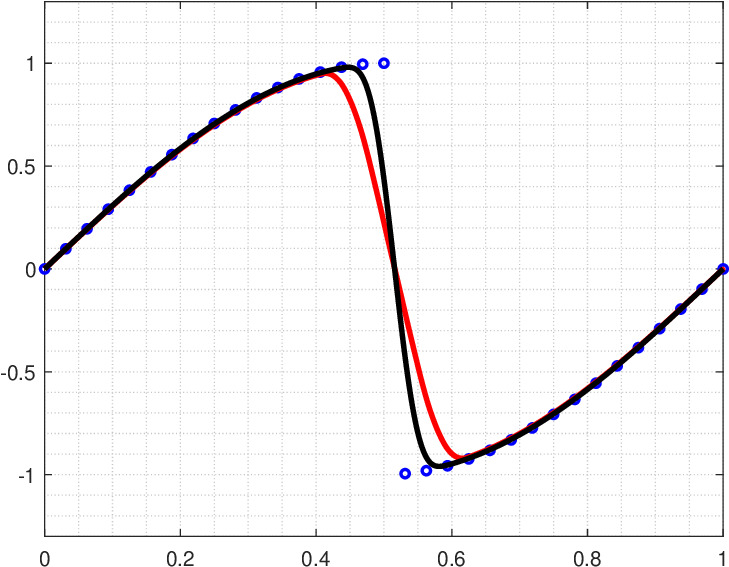} &
			\includegraphics[width=0.225\linewidth]{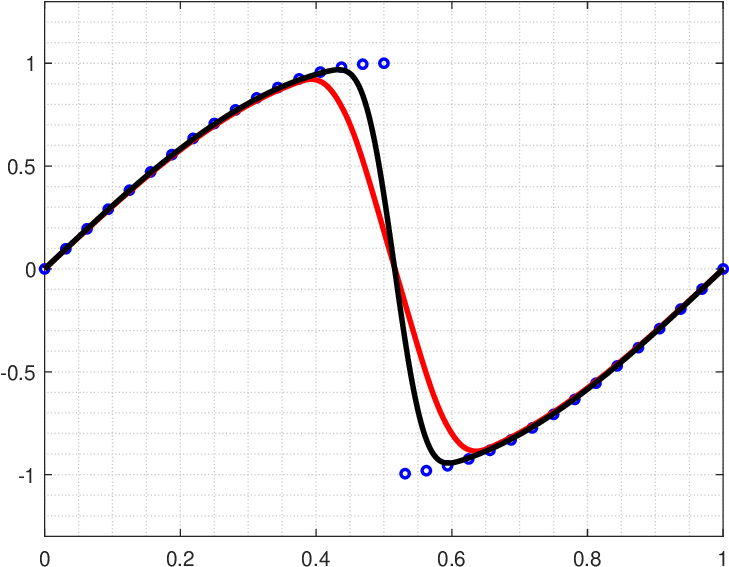} &
			\includegraphics[width=0.225\linewidth]{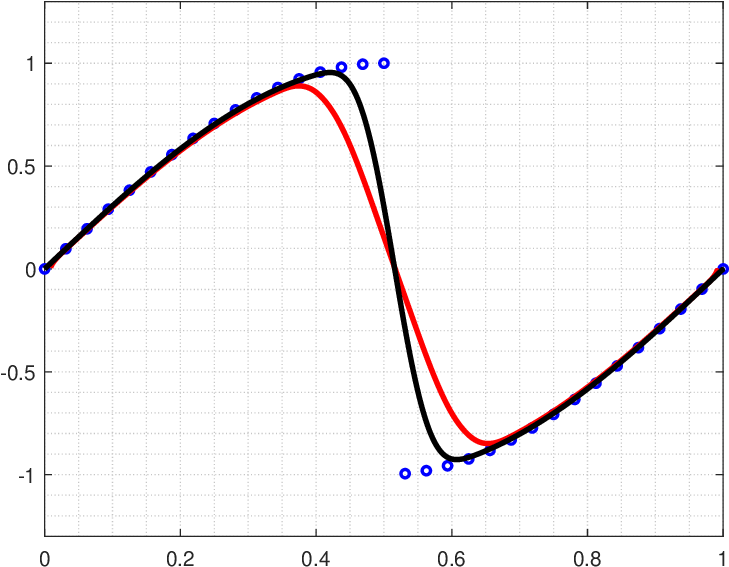}\\[10pt]
		\end{tabular}
		\caption{Limit curves for discontinuous data using subdivision schemes with {\tt rect} (red line) and {\tt trwt} (black line) weight functions, $d=0,1$.}
		\label{fig:experimento8}
	\end{figure}

	\subsection{Monotonicity}
	
	Finally, we introduce the last example in order to see numerically that the new family of the schemes conserves the monotonicity of the data, for $d=0,1$, proved in Corollary \ref{cor:monotone}. We apply $S_{1,{\tt rect}}$ and $S_{1,{\tt trwt}}$ to the data collected in Table \ref{tabladatosmonotonia} (see \cite{arandigadonasantagueda}) and obtain Figure \ref{fig:experimento9}.
\begin{table}[H]
\begin{center}
\begin{tabular}{lrrrrrrrrrrrrrrrrr}
\hline
$x$ & 1 & 2 & 3 & 4 & 5 & 6 & 7 & 8 & 9 & 10 & 11 & 12 & 13 & 14 & 15 & 16 & 17 \\
$f(x)$ &10     & 10 & 10 & 10 & 10 & 10.5 & 10.5 & 10.5 & 10.5 & 15 & 50 & 50 & 50 & 50 & 60 & 85 & 85  \\
\hline
\end{tabular}
\end{center}
\caption{Staircase data.}
\label{tabladatosmonotonia}
\end{table}
	
	\begin{figure}[H]
		\begin{tabular}{cccc}
			$\lambda=2.5$ & $\lambda=4.5$ & $\lambda=5.5$ & $\lambda=6.5$ \\
			\includegraphics[width=0.225\linewidth]{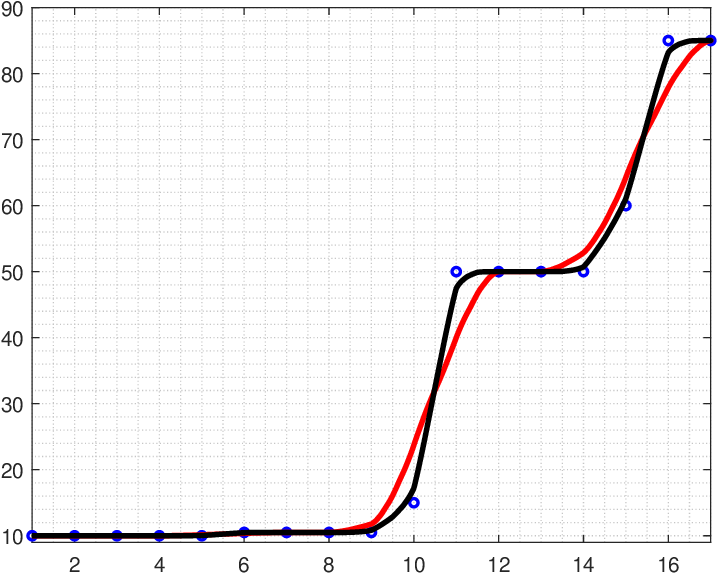} & 	
			\includegraphics[width=0.225\linewidth]{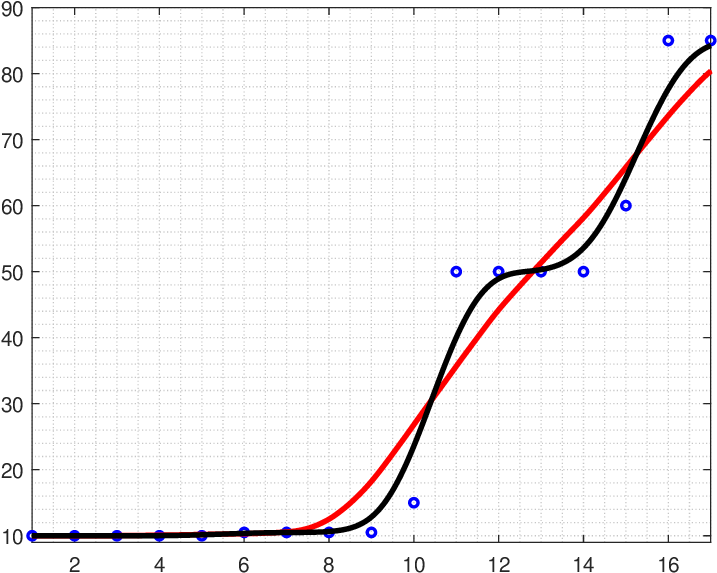} &
			\includegraphics[width=0.225\linewidth]{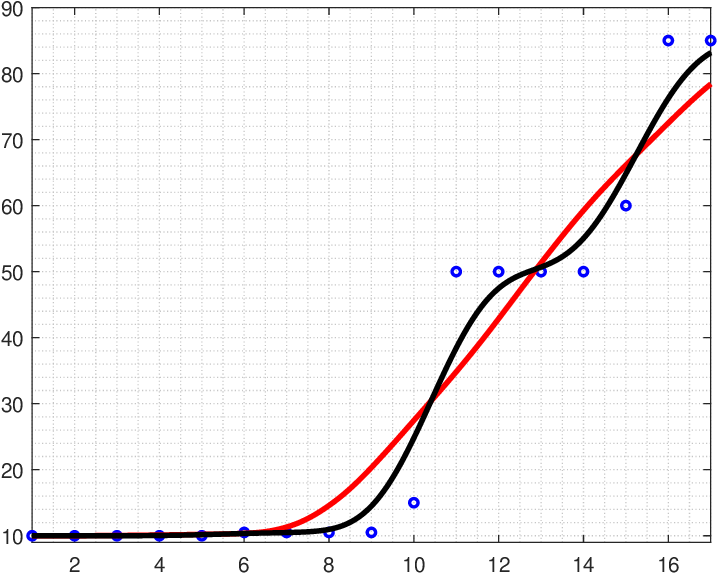} &
			\includegraphics[width=0.225\linewidth]{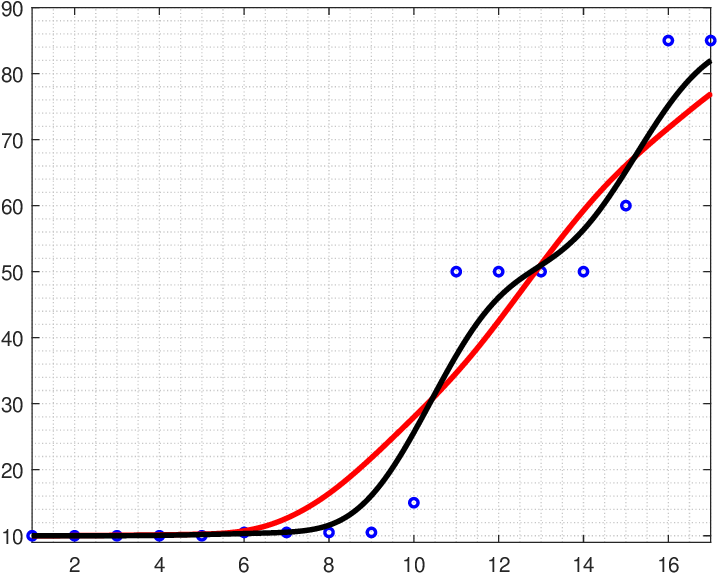}
		\end{tabular}
		\caption{Limit curves for monotone data using subdivision schemes with {\tt rect} (red line) and {\tt trwt} (black line) weight functions, $d=0,1$. }
		\label{fig:experimento9}
	\end{figure}

\section{Conclusions and future work}

In this work, a family of subdivision schemes based on weighted local polynomial regression has been analysed. We introduced the general form of this type of schemes and prove that the schemes corresponding to the polynomial degrees $d=2k$ and $d=2k+1$ coincide, for $k=0,1,2\ldots$  In particular, we analysed in detail the cases $d=0,1,2,3$ with positive weight functions, $\omega$, with compact support.

In the first part of the paper, for $d=0,1$, we took advantage of the positivity of the mask to prove the convergence. Also, under some conditions of the $\omega$ functions, the $\mathcal{C}^1$ regularity of the limit function was demonstrated. Afterward, some properties were proved as monotonicity and elimination of the Gibbs phenomenon effect.
In the second part, we developed a general technique to analyse the convergence of a family of linear schemes and used it in the case $d=2,3$.

The last sections have been dedicated to discussing noise removal and approximation capabilities. We showed how the weight function $\phi$ determines these properties and that it is not possible to find a $\phi$ maximizing both capabilities approximation and noise reduction. This led to a multi-objective optimization problem in which optimal solutions were found along a Pareto front.
Some numerical tests were presented to confirm the theoretical results.

For future works, we can consider the following ideas: The $\cC^1$ regularity of the cases $d=2,3$ were not proven. New theoretical tools such as those presented in Section \ref{sec:tools} and their application to these schemes can be done.

We considered several weight functions $\phi$ from the literature. Now that we know the influence of $\phi$ in the approximation and denoising capabilities, it could be designed $\phi$ trying to improve them. Taking into account that the noise contribution is usually greater than the approximation error on the final curve, the use of an optimized weight function can be even more interesting than augmenting the polynomial degree, since some properties related to the monotonicity and the Gibbs phenomenon are only available for $d=0,1$.
	
If the data present some outliers, a different loss function can provide better results. Mustafa et al. in \cite{mustafa} proposed a variation of Dyn's schemes changing the $\ell^2$-norm by the $\ell^1$-norm in the polynomial regression but they do not prove their properties. The theoretical study of this scheme, as well as the use of different weight functions, can be considered in the future.
\bibliographystyle{these}

\section{Declarations}

\subsection*{Conflict of interest} The authors declare that they have no conflict of interest.

\subsection*{Data Availability Statements} Data sharing not applicable to this article as no datasets were generated or analysed during the current study.

\end{document}